\documentclass[12pt,letterpaper]{amsart}

\usepackage{epsfig,amsmath,amsthm,amssymb,graphicx}
\usepackage[top=1in, bottom=1in, left=1in, right=1in]{geometry}
\usepackage{color}

\usepackage{tikz}

\usetikzlibrary{cd}

\usepackage{amssymb, latexsym, epsfig}
\usepackage{blkarray}
\usepackage{multirow}
\usepackage{graphicx}
\usepackage{caption}
\usepackage{subcaption}

\usepackage{enumitem}

\newtheorem{proposition}{Proposition}[section]
\newtheorem{theorem}[proposition]{Theorem}

\newtheorem{lemma}[proposition]{Lemma}

\theoremstyle{definition}
\newtheorem{definition}[proposition]{Definition}

\newtheorem{questions}[proposition]{Questions}
\newtheorem{example}[proposition]{Example}

\makeatletter
\newtheorem*{rep@theorem}{\rep@title}
\newcommand{\newreptheorem}[2]{%
\newenvironment{rep#1}[1]{%
 \def\rep@title{#2 \ref{##1}}%
 \begin{rep@theorem}}%
 {\end{rep@theorem}}}
\makeatother

\newreptheorem{theorem}{Theorem}
\newreptheorem{proposition}{Proposition}
\newreptheorem{corollary}{Corollary}

\colorlet{green1}{green!40!black!80!}

\newtheorem*{theorem*}{Theorem}
\newtheorem*{proposition*}{Proposition}
\newtheorem*{lemma*}{Lemma}
\newtheorem*{corollary*}{Corollary}

\newcommand{\bdry}{\partial}

\newcommand{\lk}{\operatorname{lk}}

\newcommand{\bk}{\backslash}
\newcommand{\bbF}{\mathbb{F}}
\newcommand{\bbX}{\mathbb{X}}
\newcommand{\x}{\times}

\renewcommand{\int}{\operatorname{int}}
\newcommand{\Id}{\operatorname{Id}}

\newcommand{\inv}{^{-1}}
\newcommand{\Prod}{\displaystyle \prod}

\newcommand{\vect}[1]{\overset{\rightharpoonup}{#1}}

\newcommand{\Skip}{\operatorname{Skip}}
\newcommand{\sgn}{\operatorname{sgn}}

\newcommand{\Sum}{\displaystyle \sum}

\renewcommand{\Cup}{\displaystyle \bigcup}

\newcommand{\pref}[1]{(\ref{#1})}

\begin{document}
\title[Moves relating C-complexes]{moves relating C-complexes:\\ A correction to Cimasoni's ``A geometric construction of the Conway potential function''}

\author{Christopher W.\ Davis}
\address{Department of Mathematics, University of Wisconsin--Eau Claire}
\email{daviscw@uwec.edu}
\urladdr{people.uwec.edu/daviscw}

\author{Taylor Martin}
\address{Department of Mathematics, Sam Houston State University}
\email{taylor.martin@shsu.edu}

\author{Carolyn Otto}
\address{Department of Mathematics, University of Wisconsin--Eau Claire}
\email{ottoa@uwec.edu}

\date{\today}


\maketitle

\begin{abstract} 
In groundbreaking work from 2004, Cimasoni gave a geometric computation of the multivariable Conway potential function in terms of a generalization of a Seifert surface for a link called a C-complex \cite{CC}.  Lemma 3 of that paper provides a family of moves which relates any two C-complexes for a fixed link.  This allows for an approach to studying links from the point of view of C-complexes and in following papers it has been used to derive invariants.  This lemma is false.  We present counterexamples, a correction with detailed proof, and an analysis of the consequences of this error on subsequent works that rely on this lemma.
\end{abstract}

\section{Introduction and statement of results}

A C-complex is a generalization of a Seifert surface to the setting of links, or more broadly, colored links.   Specifically, a \emph{C-complex} is a union of compact, oriented, embedded surfaces in $S^3$ which are allowed to intersect in {clasps} with no triple points. These C-complexes first appeared in work of Cooper \cite{CoopThesis, Cooper} where they were used to compute the Alexander module, signature, and nullity of a 2-component link.  In 2004, Cimasoni extended Cooper's work to compute the multivariable Conway potential function for colored links of arbitrarily many components \cite{CC}.  In subsequent works including \cite{CimConZac}, \cite{CF}, \cite{CimTur}, \cite{ConFrTOf}, and \cite{ConNagTof}, this is extended to other signatures, Alexander modules, Casson-Gordon invariants,  Blanchfield forms, splitting numbers, links in quasi-cylinders, and other topics. 

A \emph{colored link} is an oriented link $L$ in $S^3$ such that each component $K$ of $L$ is assigned a \emph{color} $\sigma(K) \in \{1, \dots, n \}$.  We will express an $n$-colored link as $(L,\sigma)$ or  $L=L_1\cup\dots\cup L_n$ where $L_i = \sigma^{-1}(\{i\})$ is the $i$-colored sublink. Thus, a $1$-colored link is just an oriented link, and an $m$-colored, $m$-component link is an ordered, oriented link. Two $n$-colored links $(L, \sigma)$ and $(L', \sigma')$ are \emph{isotopic} if there exists a color and orientation preserving ambient isotopy from $L$ to $L'$.  A C-complex $F=F_1\cup\dots \cup F_n$ is said to be \emph{bounded} by the colored link $L$ if $\bdry F_i=L_i$ for $i=1,\dots, n$.   In  \cite{CC}  Cimasoni provides a set of geometric moves on C-complexes  which relate any pair of C-complexes for isotopic links.

\begin{lemma}[Lemmas 2 and 3 in \cite{CC}.  See also Lemma 2.2 in \cite{CF}]\label{false lemma}
Let $F$ and $F'$ be C-complexes bounded by isotopic colored links. Then, $F$ and $F'$ can be transformed into each other by a finite number of the following operations and their inverses:
\begin{enumerate}[label=(T\arabic*)] 
\setcounter{enumi}{-1}
\item \label{move: isotopy}\label{move:T0} Ambient isotopy,
\item \label{move: handle}\label{move:T1}  Handle attachment on one surface,
\item \label{move: ribbon+push}\label{move:T2}  Addition of a ribbon intersection followed by a ``push along an arc'' through this intersection, as in Figure \ref{fig: CF (T2)},
\item \label{move: pass through clasp}\label{move:T3}  The pass through a clasp move, as in Figure \ref{fig: CF (T3)}.
\end{enumerate}
\end{lemma}

\begin{figure}
     \centering
     \begin{subfigure}[b]{0.4\textwidth}
         \centering
         \begin{tikzpicture}
         \node at (0,0){\includegraphics[width=.5\textwidth]{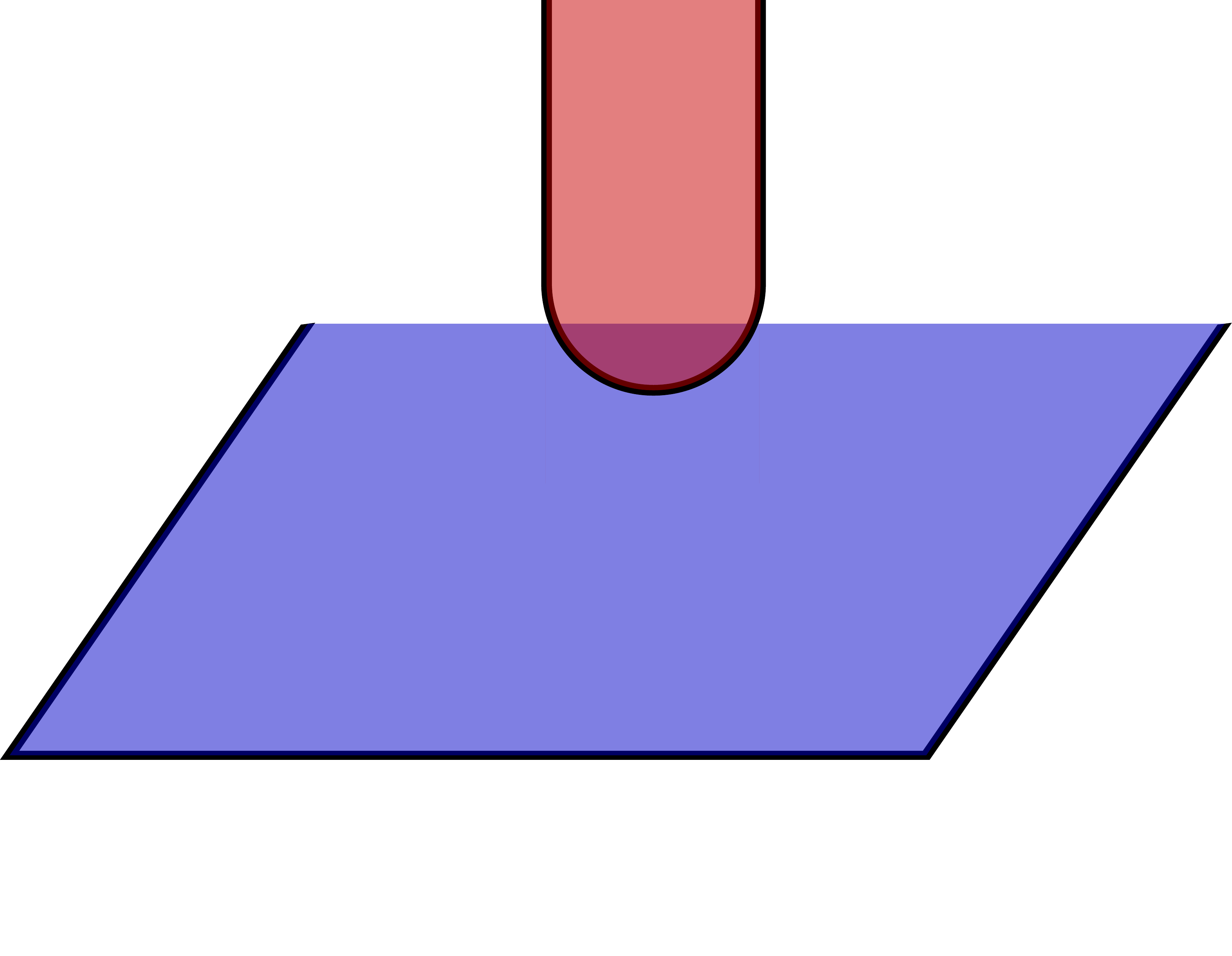}};
        \node at (3,0){\includegraphics[width=.5\textwidth]{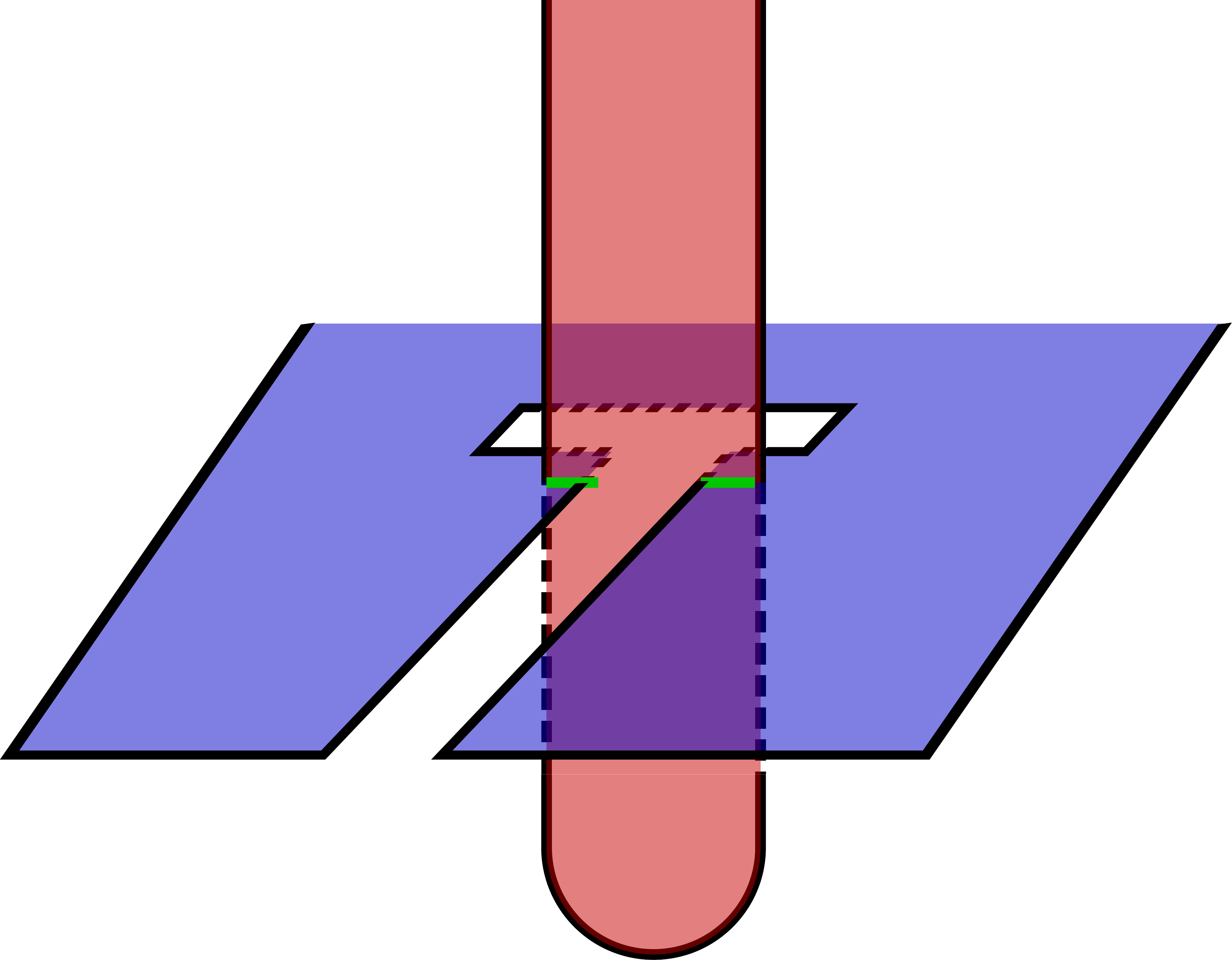}};
        \node at (1.8,.3){$\leftrightarrow$};
         \end{tikzpicture}
         \caption{The \ref{move: ribbon+push} move}
         \label{fig: CF (T2)}
     \end{subfigure}
     \begin{subfigure}[b]{0.4\textwidth}
         \centering
         \begin{tikzpicture}
         \node at (0,0){\includegraphics[width=.5\textwidth]{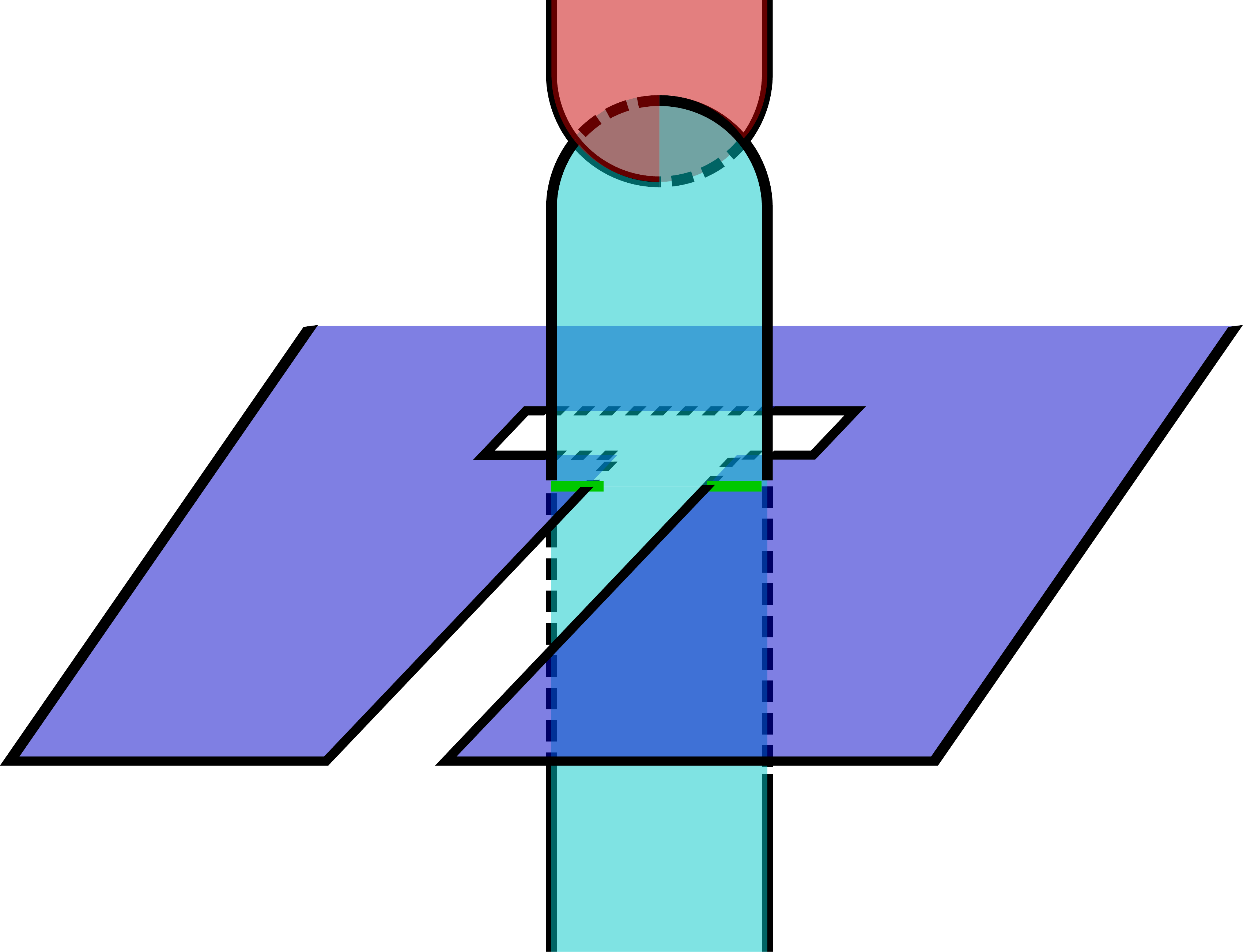}};
        \node at (3,0){\includegraphics[width=.5\textwidth]{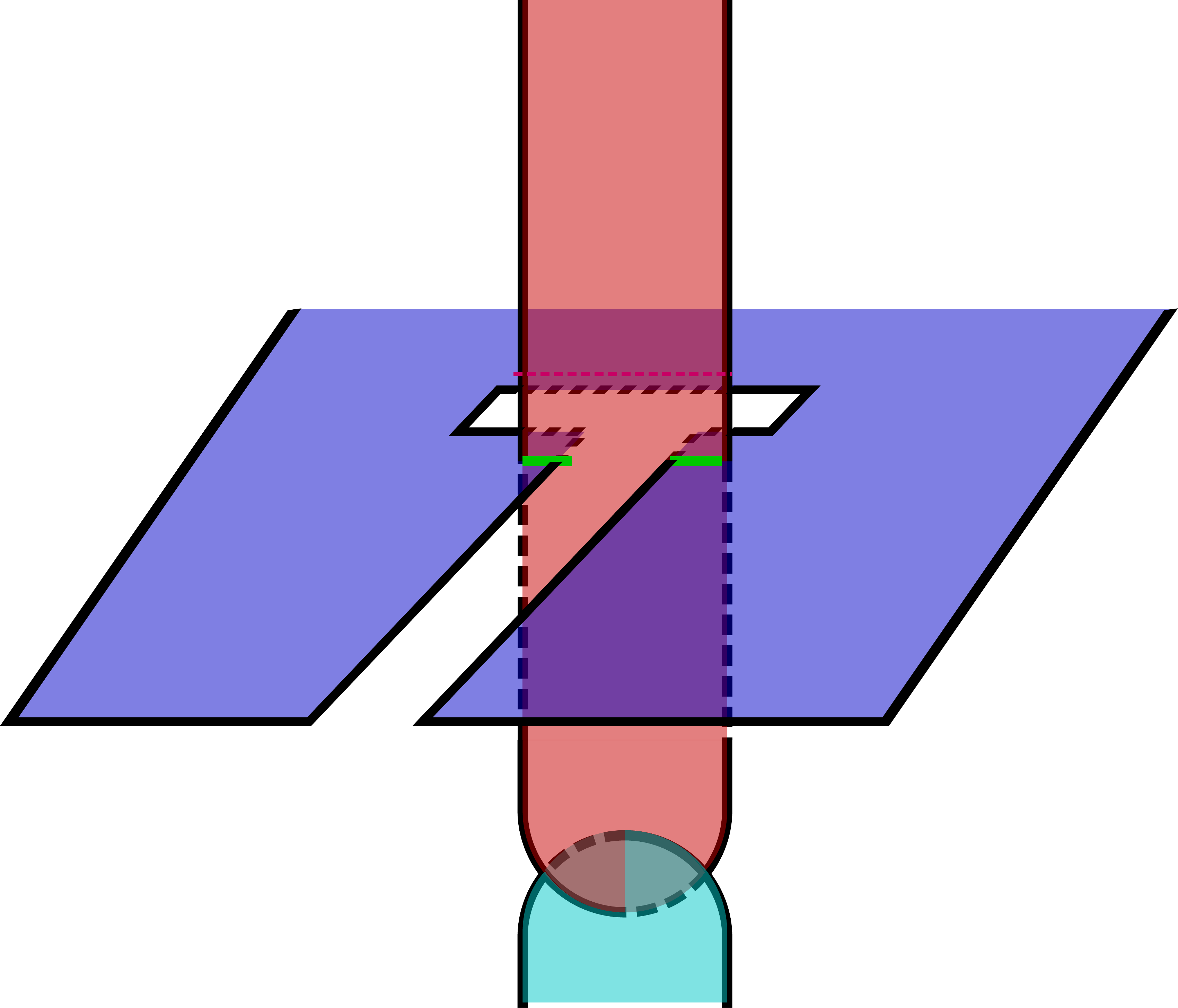}};
        \node at (1.8,.3){$\leftrightarrow$};
         \end{tikzpicture}
         \caption{The \ref{move: pass through clasp} move}
         \label{fig: CF (T3)}
     \end{subfigure}
     
        \caption{}
        \label{fig: CF moves}
\end{figure}

Some noteworthy applications of this lemma are in \cite{CC}, \cite{CF}, and \cite{CimTur} where these C-complex moves are used to verify that certain properties of a C-complex are invariants of the underlying link.  The first goal of this document is to demonstrate that the set of moves in Lemma~\ref{false lemma} is incomplete.  

\begin{theorem}\label{thm: false lemma is false}
The C-complexes $F$ and $F'$ appearing in Figures~\ref{fig: other example 1} and \ref{fig: other example 2} cannot be related by any sequence of the moves \ref{move: isotopy}, \ref{move: handle}, \ref{move: ribbon+push}, and \ref{move: pass through clasp}.  Therefore, Lemma~\ref{false lemma} is false.
\end{theorem}

\begin{figure}[h]
     \centering\begin{subfigure}[b]{0.3\textwidth}
         \centering
         \begin{tikzpicture}
        \node at (0,.2){\includegraphics[width=.8\textwidth]{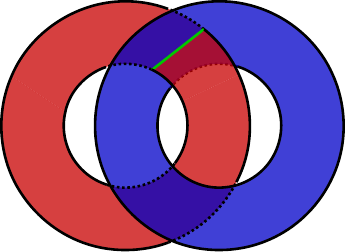}};
         \end{tikzpicture}
     \caption{A 2-colored link $L$ admitting a pair of Seifert surfaces intersecting in a ribbon.}
     \label{fig: other example 0}
     \end{subfigure}
         \hfill
     \begin{subfigure}[b]{0.3\textwidth}
         \centering
         \begin{tikzpicture}
        \node at (0,.2){\includegraphics[width=.8\textwidth]{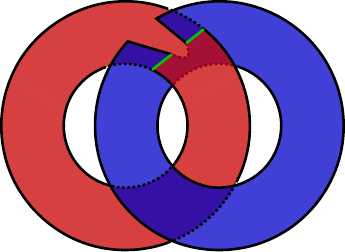}};
         \end{tikzpicture}
     \caption{A C-complex $F$ obtained by splitting the ribbon into clasps.}
     \label{fig: other example 1}
     \end{subfigure}
         \hfill
     \begin{subfigure}[b]{0.3\textwidth}
         \centering
         \begin{tikzpicture}
        \node at (0,.2){\includegraphics[width=.8\textwidth]{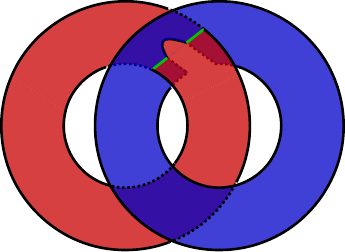}};
         \end{tikzpicture}
     \caption{A C-complex $F'$ obtained by splitting the ribbon into clasps in a different way.}
     \label{fig: other example 2}
     \end{subfigure}
     \hfill
        \caption{
         }
        \label{fig: other example}
\end{figure}

Thus, a correct statement of Lemma~\ref{false lemma} will require an additional move.  Such a move is suggested in Figure~\ref{fig: other example}.  These two C-complexes are realized by starting with a pair of surfaces sharing a ribbon intersection and pushing along an arc to split that ribbon into a pair of clasps in two different ways.  We introduce a new move which we call \ref{move: push along different arc} inspired by this example and prove the following.

\begin{theorem}\label{thm: corrected lemma}
Let $F$ and $F'$ be C-complexes bounded by isotopic colored links $L$ and $L'$. Then, $F$ and $F'$ can be transformed into each other by a finite sequence of the moves \ref{move: isotopy}, \ref{move: handle}, \ref{move: ribbon+push}, \ref{move: pass through clasp}, and 
\begin{enumerate}[label=(T\arabic*)] 
\setcounter{enumi}{3}
\item \label{move: push along different arc}\label{move:T4}  The ``push along a different arc" move: Merge two clasps into a ribbon by the inverse of a push along an arc and then push along a different arc, as in Figure \ref{fig: Move (T4)}.

\end{enumerate}
\end{theorem}

\begin{figure}
     \centering
     \begin{subfigure}[b]{0.25\textwidth}
         \centering
         \begin{tikzpicture}
         \node at (0,0){\includegraphics[width=.75\textwidth]{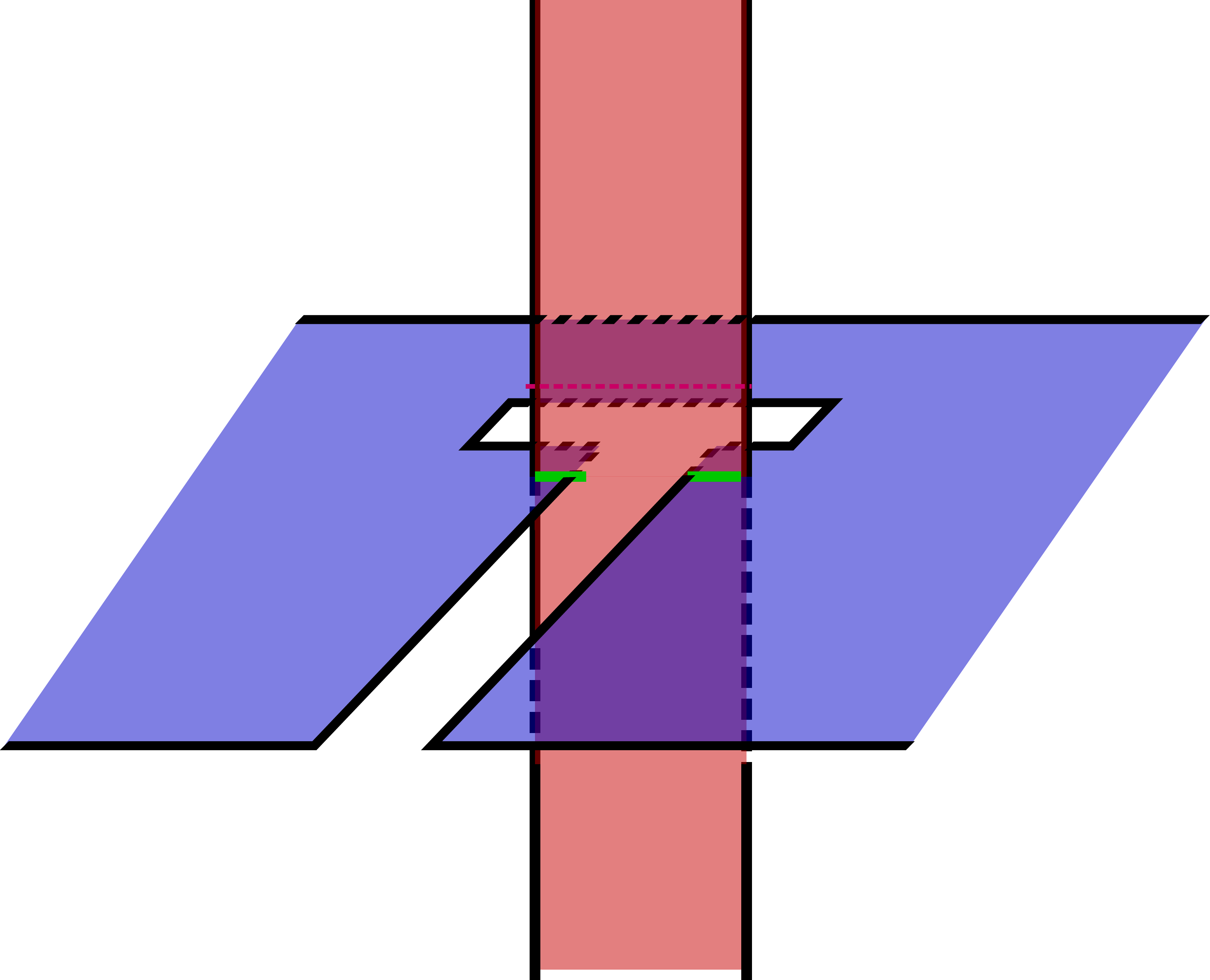}};
         \end{tikzpicture}
         \caption{}
         \label{fig: (T4)A}
     \end{subfigure}
     \begin{subfigure}[b]{0.25\textwidth}
         \centering
         \begin{tikzpicture}
        \node at (0,0){\includegraphics[width=.75\textwidth]{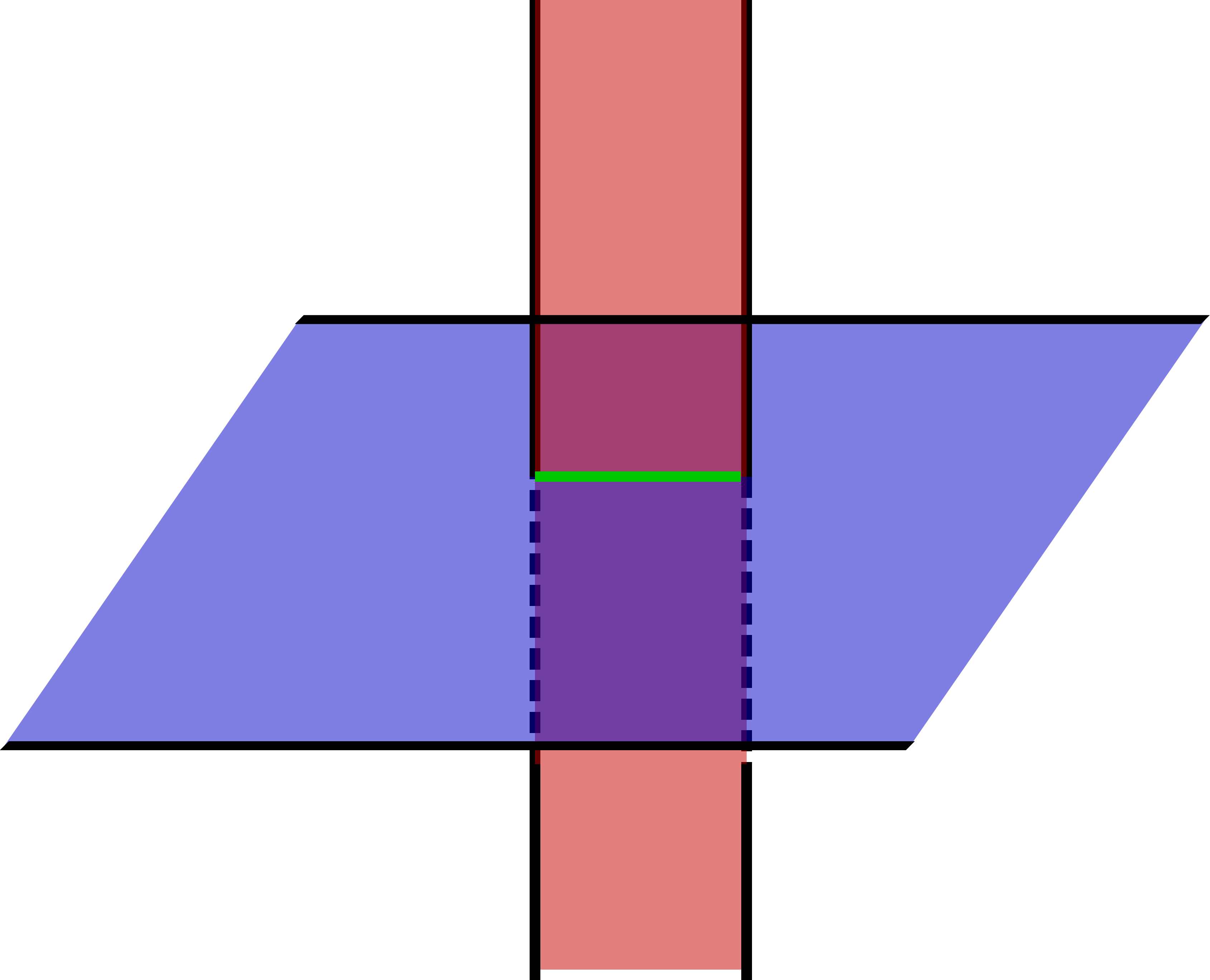}};
         \end{tikzpicture}
         \caption{}
         \label{fig: (T4)B}
     \end{subfigure}
     \begin{subfigure}[b]{0.25\textwidth}
         \centering
         \begin{tikzpicture}
        \node at (0,0){\includegraphics[width=.75\textwidth]{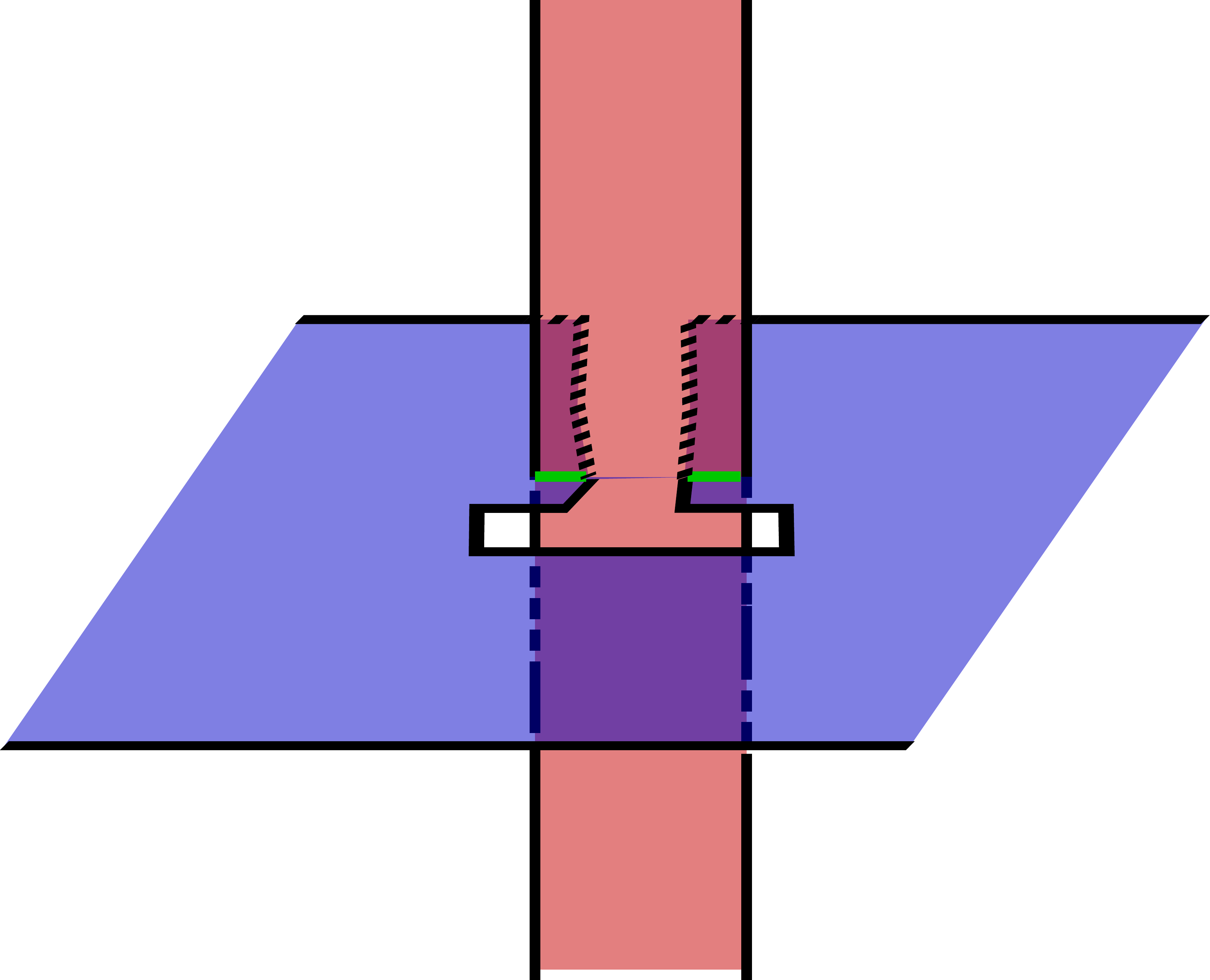}};
         \end{tikzpicture}
         \caption{}
         \label{fig: (T4)C}
     \end{subfigure}
     
        \caption{The ribbon intersection of \ref{fig: (T4)B} may be split into two clasps along any arc running from that ribbon to the boundary (\ref{fig: (T4)A} and \ref{fig: (T4)C} give two such choices).  The \ref{move: push along different arc} move allows the replacement of \ref{fig: (T4)A} by \ref{fig: (T4)C}. }
        \label{fig: Move (T4)}
\end{figure}

\subsection{Outline of paper}
In Section~\ref{sect:false lemma}, we formally define the notion of a C-complex, prove Theorem~\ref{thm: false lemma is false}, and present an infinite family of counterexamples to  Lemma~\ref{false lemma}. We also pose some questions gauging the degree to which this lemma is false.  In Section~\ref{sect:(T4) consequences}, we assume Theorem~\ref{thm: corrected lemma} and complete proofs that Cimasoni's geometric construction of the multivariable Conway potential function for colored links does still hold, as does the invariance of the signature and nullity of Cimasoni and Florens \cite{CF}.  In Section~\ref{sect: corrected lemma}, we present a careful, detailed proof of Theorem~\ref{thm: corrected lemma}.  

\subsection{Acknowledgements}
The authors would like to thank D.~Cimasoni for helpful conversations and encouragement during the progress of this project. They would also like to thank S.~Friedl for helpful conversations.

\section{C-complexes and the counterexample to Lemma~\ref{false lemma}}\label{sect:false lemma}

\begin{definition}For an $n$-colored link 
$L=L_1\cup\dots\cup L_n$, a \emph{C-complex} for $L$ consists of a union of oriented compact, smoothly embedded surfaces $F=F_1\cup\dots \cup F_n$ intersecting transversely in $S^3$ so that the following conditions hold:
\begin{itemize}
\item Each $F_i$ is a Seifert surface with no closed components for the $i$-colored sublink $L_i \subseteq L$. 
\item For any $i\neq j$, $F_i\cap F_j$ consists of a disjoint union of embedded arcs each having one endpoint in $\bdry F_i=L_i$ and the other in $\bdry F_j=L_j$.  These arcs are called clasps; see Figures~\ref{fig: positive clasp} and \ref{fig: negative clasp}.
\item There are no triple points;  for any $i<j<k$, $F_i\cap F_j\cap F_k=\emptyset$.
\end{itemize}
\end{definition}

\begin{figure}
     \centering
     \begin{subfigure}[b]{0.25\textwidth}
         \centering
         \begin{tikzpicture}
        \node at (0,0){\includegraphics[width=.75\textwidth]{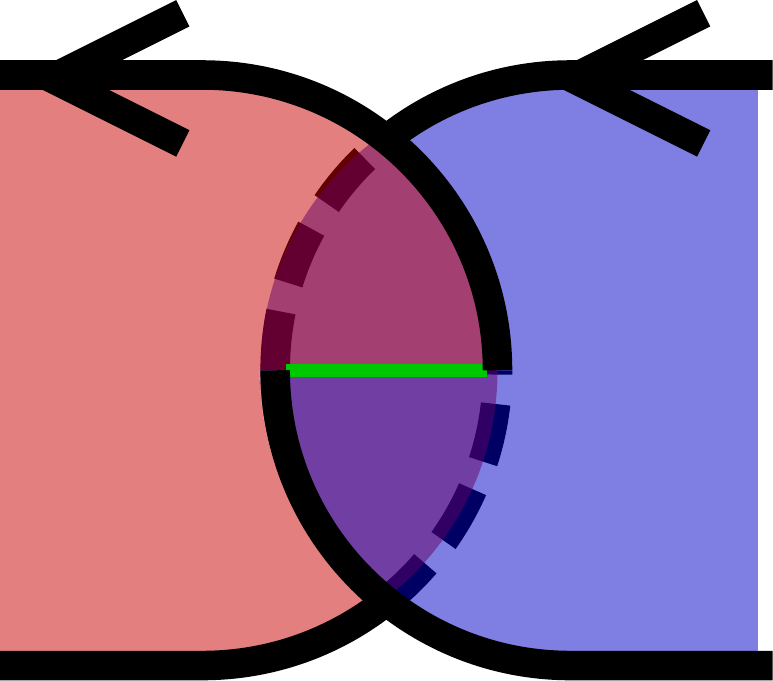}};
         \end{tikzpicture}
         \caption{A positive clasp}
         \label{fig: positive clasp}
     \end{subfigure}
     \begin{subfigure}[b]{0.25\textwidth}
         \centering
         \begin{tikzpicture}
        \node at (0,0){\includegraphics[width=.75\textwidth]{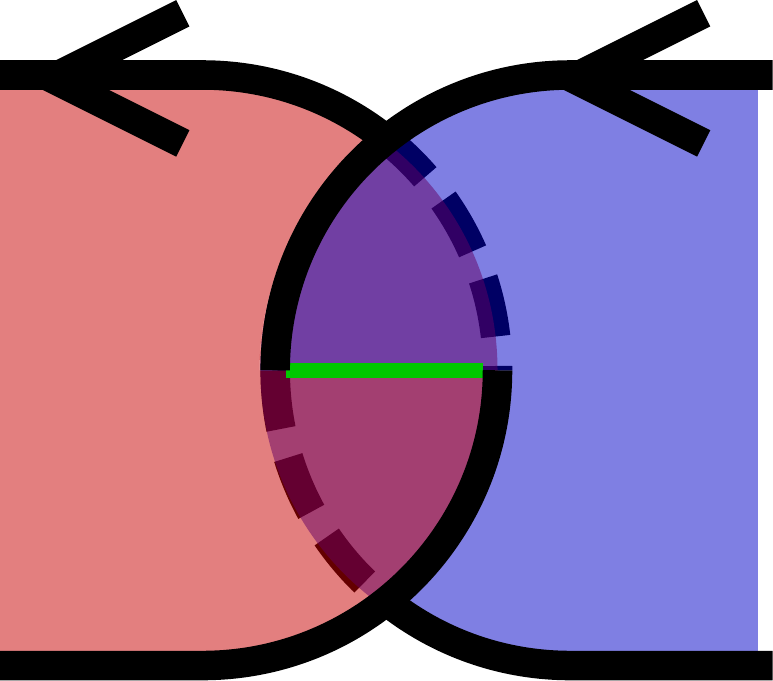}};
         \end{tikzpicture}
         \caption{A negative clasp}
         \label{fig: negative clasp}
     \end{subfigure}
     \begin{subfigure}[b]{0.25\textwidth}
         \centering
         \begin{tikzpicture}
        \node at (0,0){\includegraphics[width=.75\textwidth]{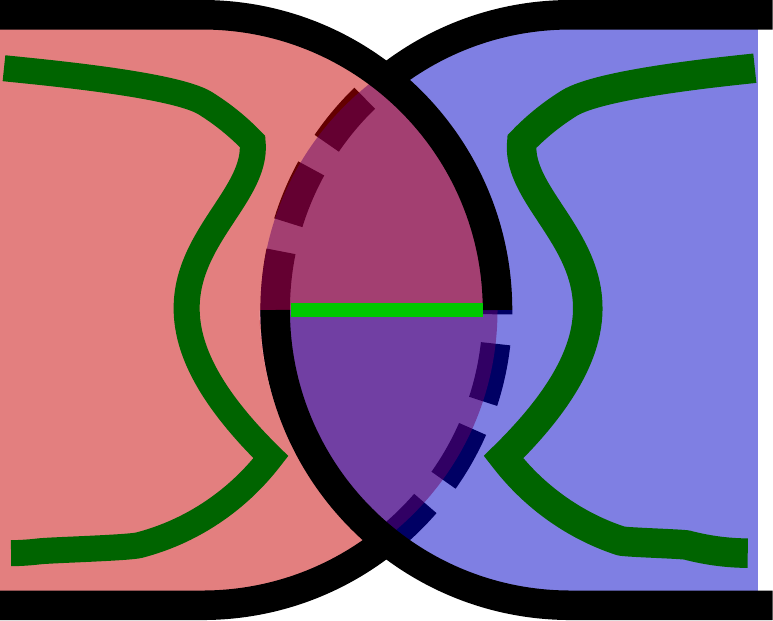}};
         \end{tikzpicture}
         \put(0,5){$L_i$}\put(0,60){\textcolor{green1}{$L_i'$}}
         \caption{$\Skip(F)$ near a clasp}
         \label{fig: skip}
     \end{subfigure}
     
        \caption{}
        \label{fig: clasps}
\end{figure}

Given a C-complex $F$, one can produce a new link, which we call the \emph{skip the clasps link}, and denote $\Skip(F)$.  In order to do so, let $A_i\subseteq F_i$ be a collar neighborhood of $\bdry F_i$ large enough to contain all of the clasps.  Observe that $A_i$ consists of a disjoint union of embedded annuli $S^1\times[0,1]$ bounded by the $i$-colored sublink $L_i$ and some other link, $L_i'$.  Set $\Skip(F) = L_1'\cup\dots\cup L_n'$.    A picture of $\Skip(F)$ near a clasp appears in Figure~\ref{fig: skip}. We show that $\Skip(F)$ is invariant under the moves \ref{move: isotopy}, \ref{move: handle}, \ref{move: ribbon+push},  and \ref{move: pass through clasp} of Lemma~\ref{false lemma}.  

\begin{lemma}\label{lem: skip is invariant}
Let $F$ and $F'$ be C-complexes which differ by a finite sequence of the moves of Lemma~\ref{false lemma}. Then, $\Skip(F) \cong \Skip(F')$.
\end{lemma}
\begin{proof}
An ambient isotopy~\ref{move: isotopy} changes $\cup_i A_i$, and in particular $\Skip(F)$, by an isotopy.  A handle attachment \ref{move: handle} is performed away from a collar neighborhood of the boundary. Thus, it preserves $\cup_i A_i$ and so $\Skip(F)$.  The move \ref{move: ribbon+push} along with its effect on $\Skip(F)$ appears in Figure~\ref{fig: (T2) preserves}.  By inspection, \ref{move: ribbon+push} preserves the isotopy class of $\Skip(F)$.  Similarly in Figure~\ref{fig: (T3) preserves}, we see that  \ref{move: pass through clasp} preserves $\Skip(F)$.  \end{proof}

\begin{figure}[h]
     \centering\begin{subfigure}[b]{0.4\textwidth}
         \centering
         \begin{tikzpicture}
        \node at (0,.3){\includegraphics[width=.5\textwidth]{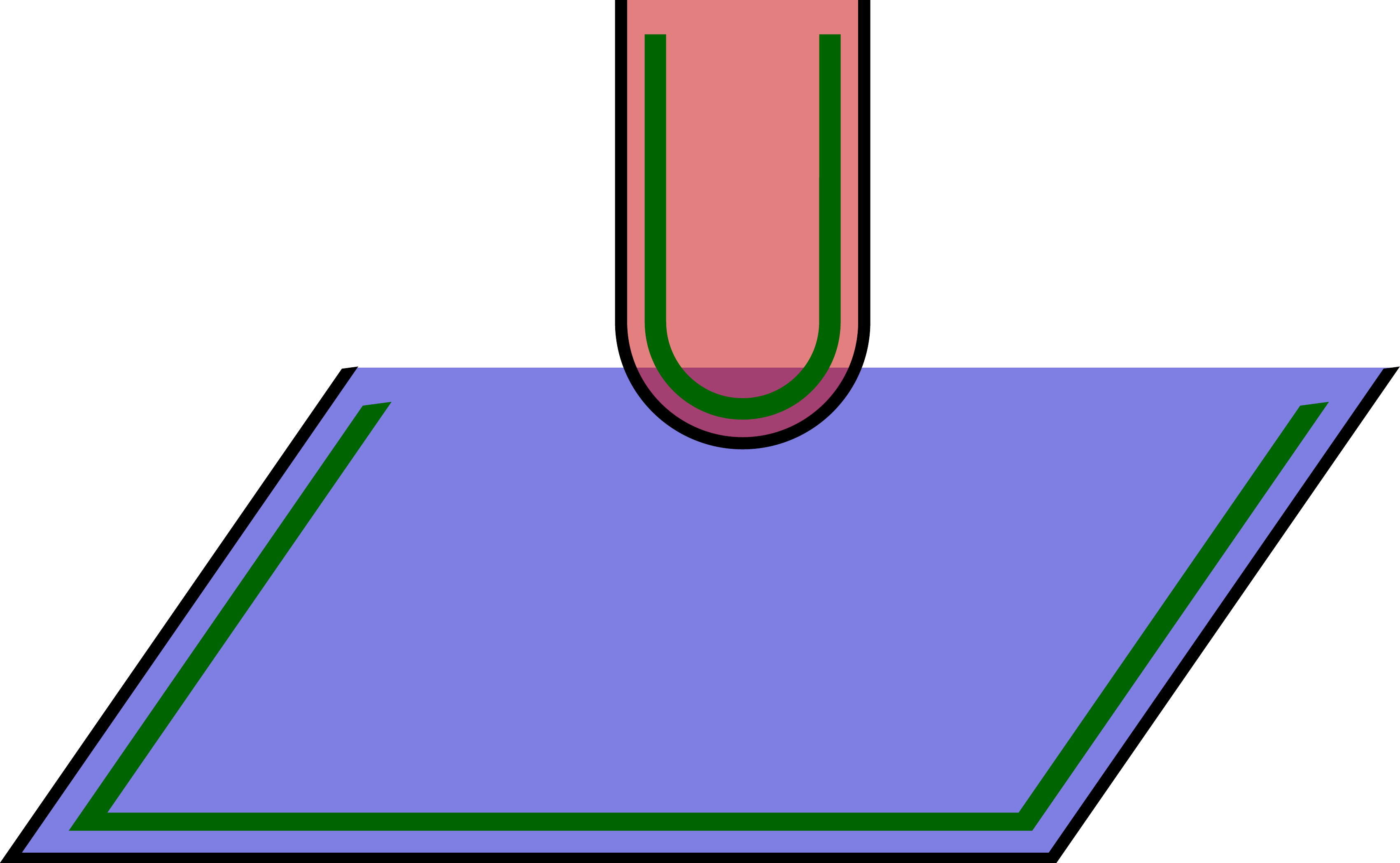}};

        \node at (3.5,0){\includegraphics[width=.5\textwidth]{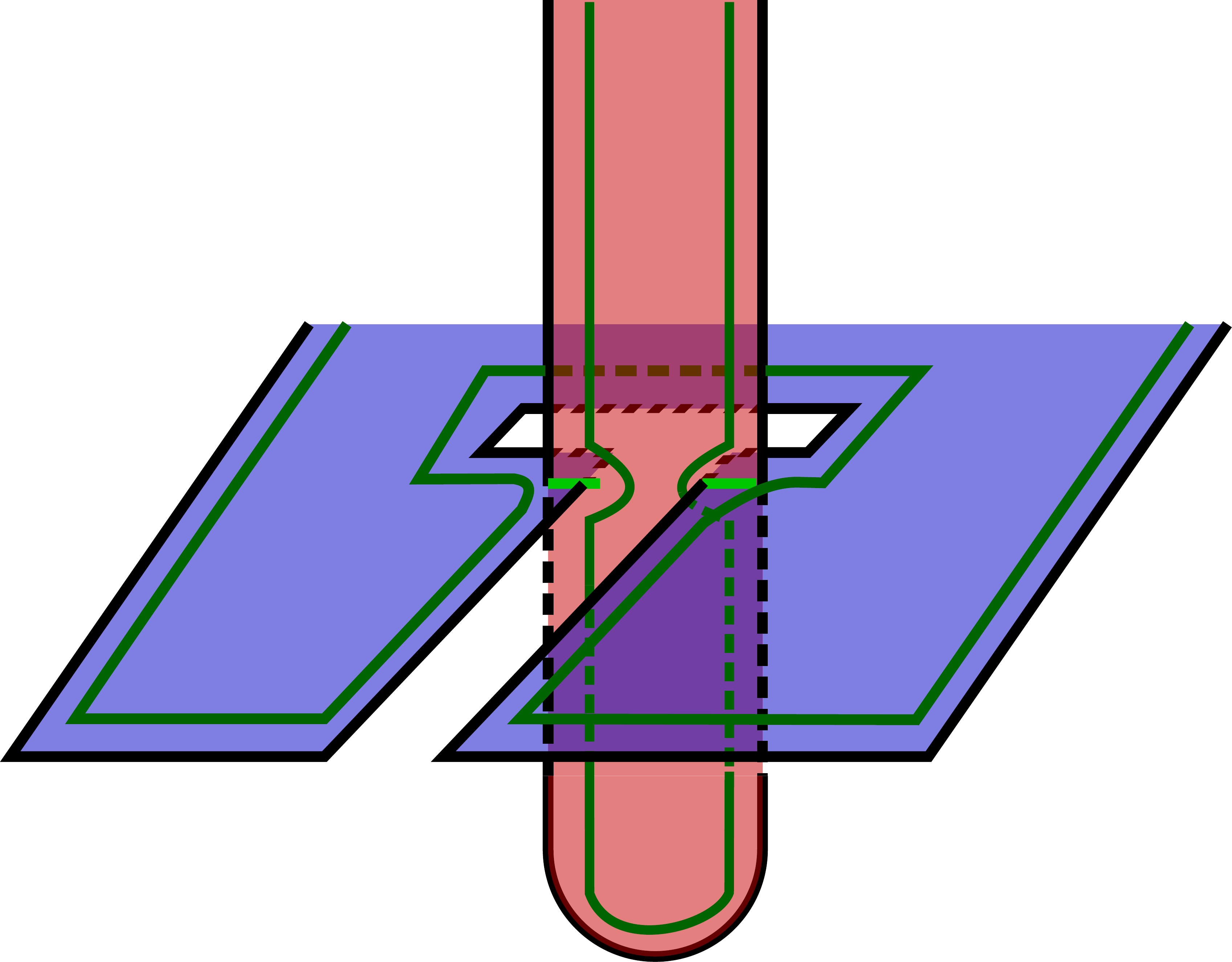}};
         \end{tikzpicture}
     \caption{$F$ and $F'$ related by a \ref{move: ribbon+push} move}
     \label{}
     \end{subfigure}     
     \begin{subfigure}[b]{0.4\textwidth}
         \centering
         \begin{tikzpicture}
        \node at (0,.3){\includegraphics[width=.5\textwidth]{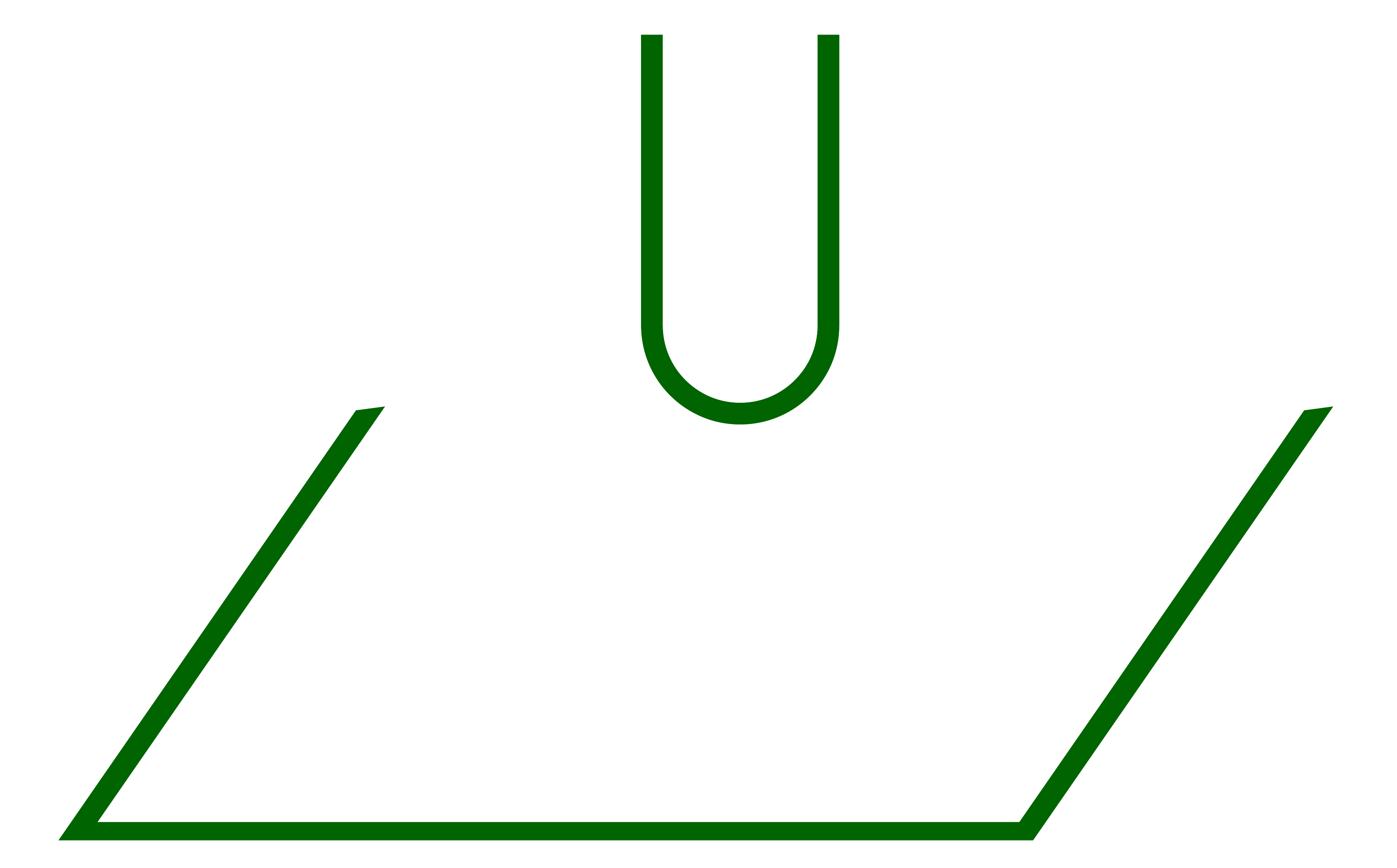}};

        \node at (3.5,0){\includegraphics[width=.5\textwidth]{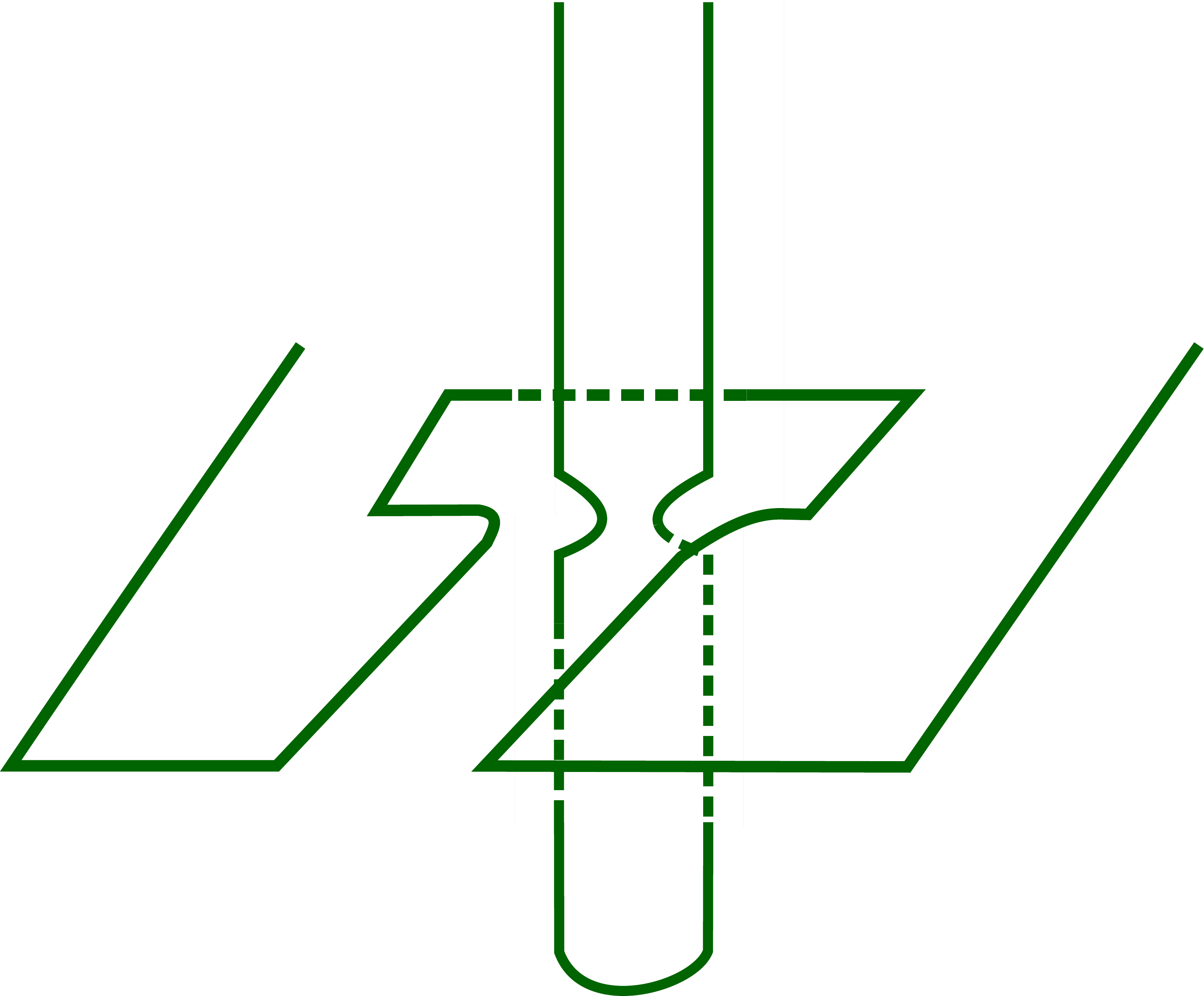}};
         \end{tikzpicture}
     \caption{$\Skip(F)$ and $\Skip(F')$}
     \label{}
     \end{subfigure}
     \hfill
        \caption{$\Skip(F)$ is preserved by the \ref{move: ribbon+push} move.}
        \label{fig: (T2) preserves}
\end{figure}

\begin{figure}[h]
     \centering\begin{subfigure}[b]{0.4\textwidth}
         \centering
         \begin{tikzpicture}
        \node at (0,.2){\includegraphics[width=.5\textwidth]{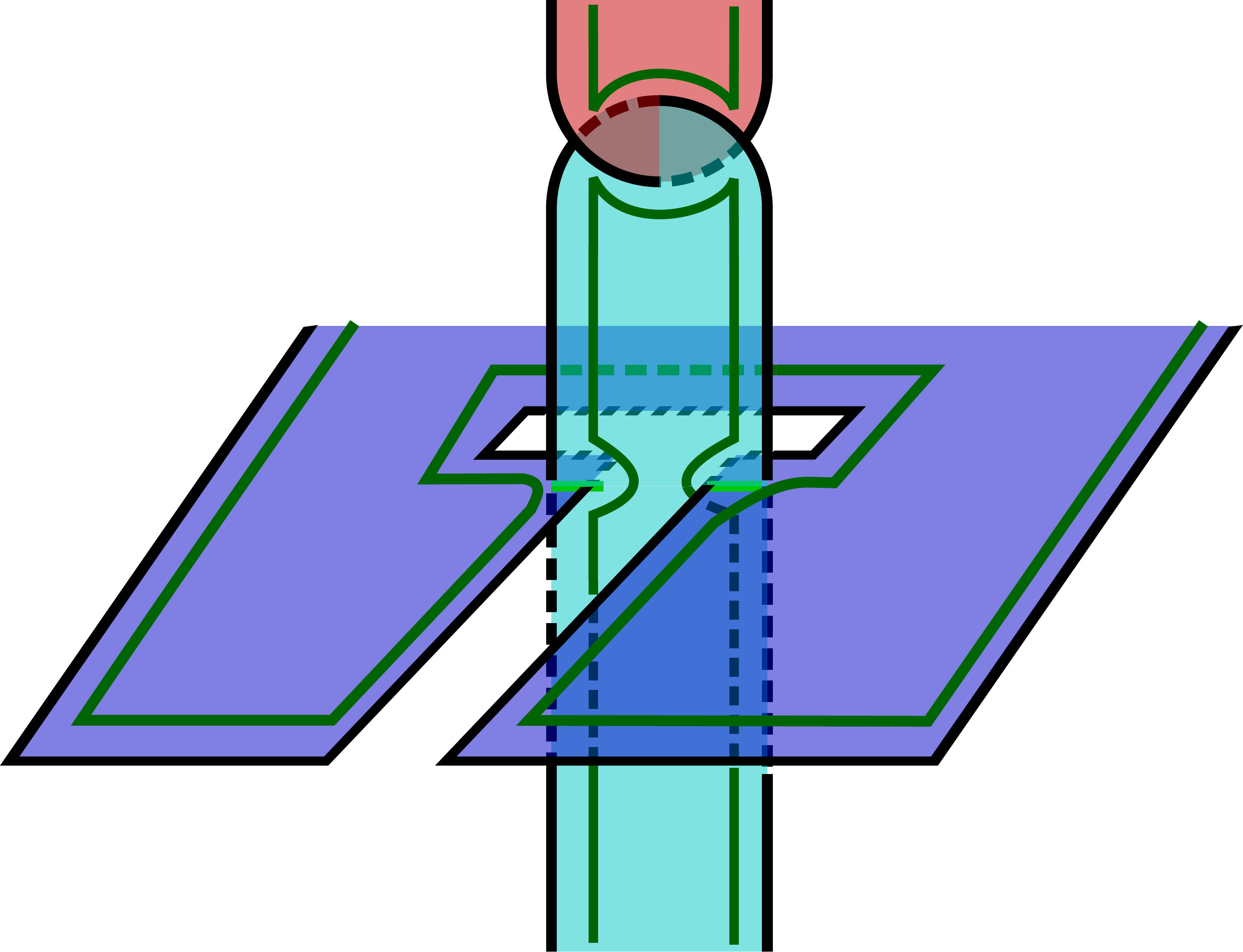}};

        \node at (3.5,0){\includegraphics[width=.5\textwidth]{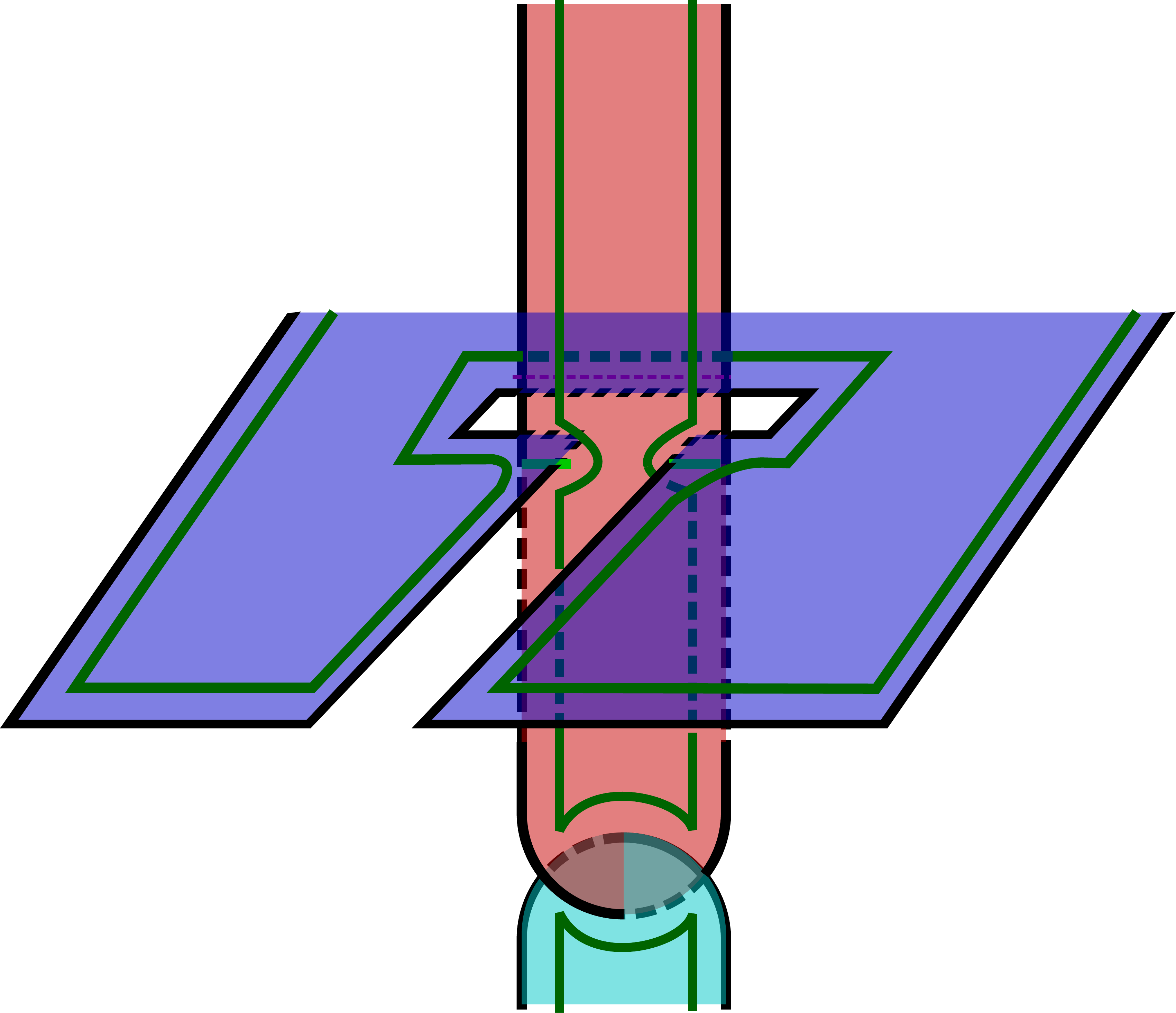}};
         \end{tikzpicture}
     \caption{$F$ and $F'$ related by a \ref{move: pass through clasp} move}
     \label{}
     \end{subfigure}
     \begin{subfigure}[b]{0.4\textwidth}
         \centering
         \begin{tikzpicture}
        \node at (0,.3){\includegraphics[width=.5\textwidth]{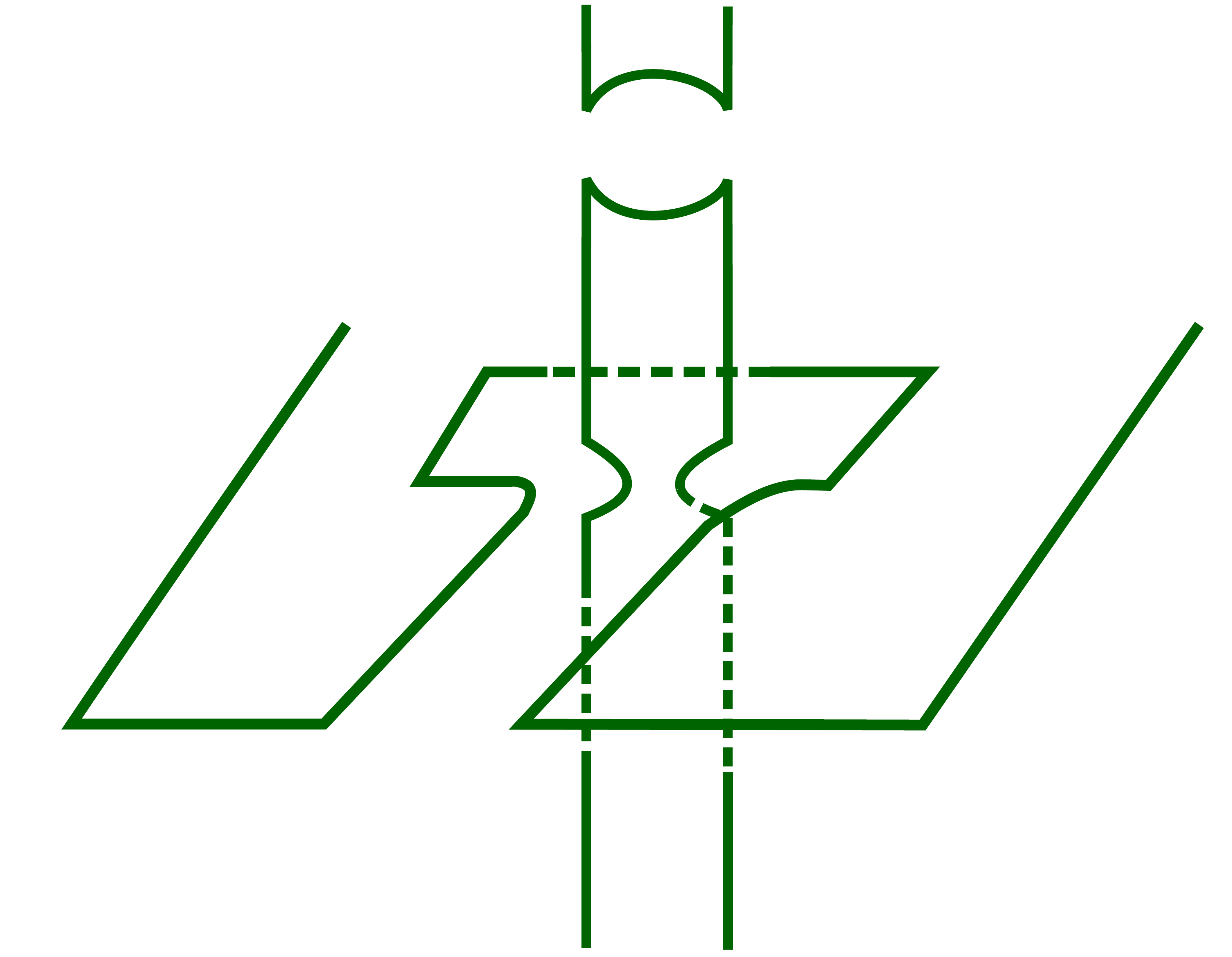}};

        \node at (3.5,0){\includegraphics[width=.5\textwidth]{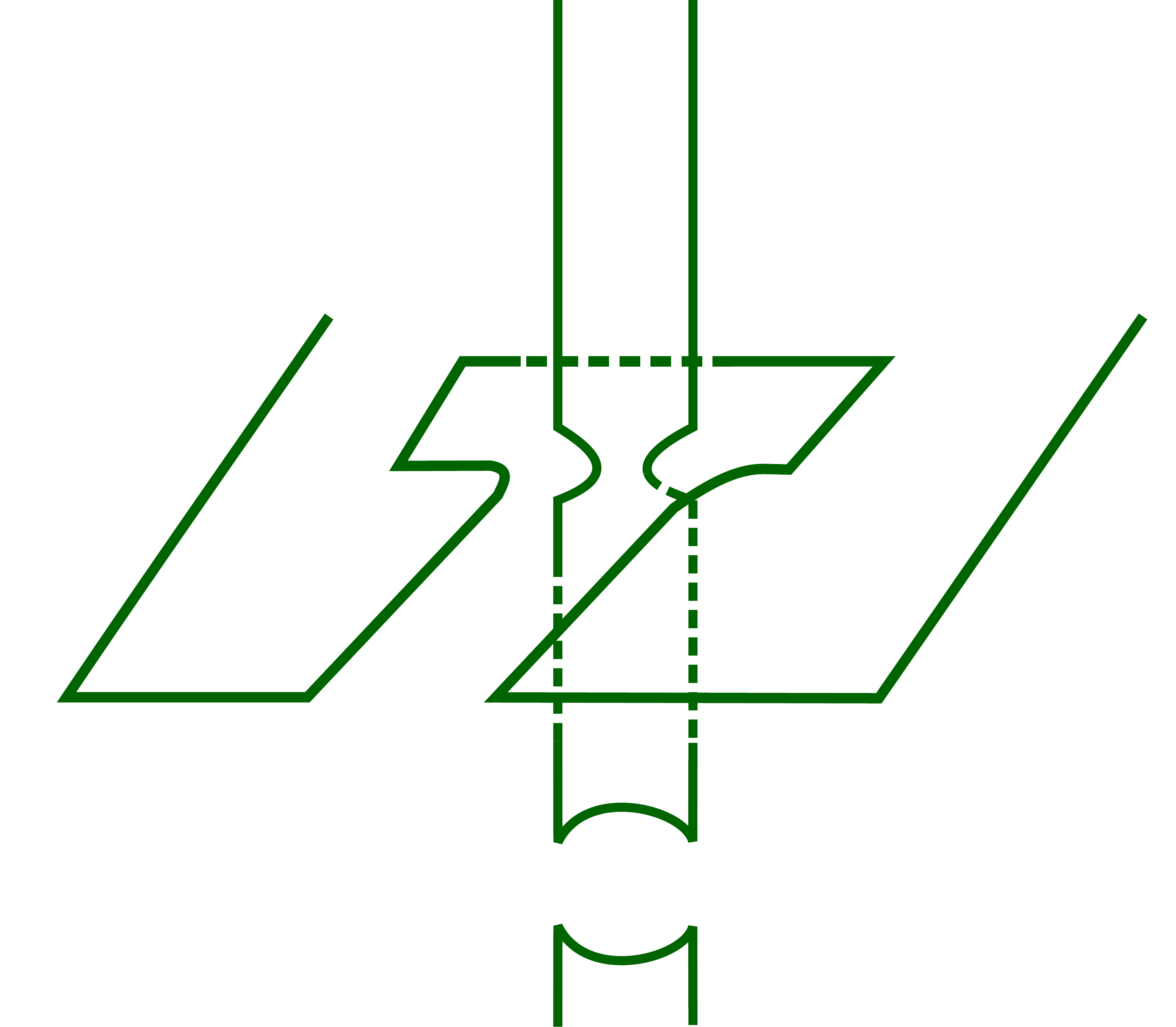}};
         \end{tikzpicture}
     \caption{$\Skip(F)$ and $\Skip(F')$}
     \label{}
     \end{subfigure}
     \hfill
        \caption{$\Skip(F)$ is preserved by the \ref{move: pass through clasp} move.}
        \label{fig: (T3) preserves}
\end{figure}

As a consequence of Lemma~\ref{lem: skip is invariant}, if Lemma~\ref{false lemma} were true, then $\Skip(F)$ would depend only on the link bounded by $F$;  $\Skip(F)$ would be be an invariant of colored links. We use this idea to present a counterexample to Lemma~\ref{false lemma}.  

\begin{proof}[Proof of Theorem~\ref{thm: false lemma is false}]

 Consider the C-complexes $F$ and $F'$ of Figure~\ref{fig: other example}.  In Figure~\ref{fig: Hopf example skip} we see $\Skip(F)$ and $\Skip(F')$.  As $\Skip(F)$ has non-vanishing linking numbers and $\Skip(F')$ is the unlink, Lemma~\ref{lem: skip is invariant} implies that $F$ and $F'$ are not related by any sequence of  the moves from Lemma~\ref{false lemma}.

\begin{figure}[h]
  \hfill
  \begin{subfigure}[b]{0.4\textwidth}
         \centering
\begin{tikzpicture}
    \node at (0,0) {\includegraphics[width=.4\textwidth]{HopfyExample1.pdf}};
    \node at (3.5,0) {\includegraphics[width=.4\textwidth]{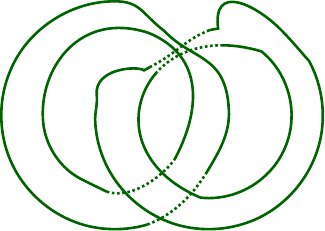}};
\end{tikzpicture}
     \caption{A C-complex $F$ for which $\Skip(F)$ is not the unlink.}
        \label{fig: Hopf Skip 1}
      \end{subfigure}
      \hfill
      \phantom{p}
  \begin{subfigure}[b]{0.4\textwidth}
         \centering
\begin{tikzpicture}
    \node at (0,0) {\includegraphics[width=.4\textwidth]{HopfyExample2.pdf}};
    \node at (3.5,0) {\includegraphics[width=.4\textwidth]{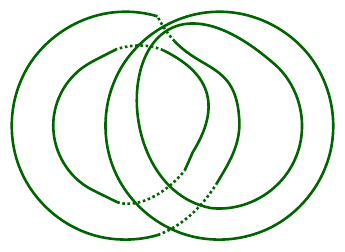}};
\end{tikzpicture}
     \caption{A C-complex $F'$ with for which $\Skip(F')$ is the unlink.}
        \label{fig: Hopf Skip 2}
      \end{subfigure}  
     \caption{Two C-complexes bounded by the same link with different skip the clasps links.}
         \label{fig: Hopf example skip}
\end{figure}

\end{proof}

\begin{example}\label{other example}

In order to evidence that counterexamples to Lemma~\ref{false lemma} abound, we present an infinite family.  Let $K$ be any knot and $L=L_1\cup L_2=BD(K)$ be the Bing double of $K$ realized as a 2-colored link.  In Figure~\ref{fig: Bing double ex} we see a C-complex $F=F_1\cup F_2$ for $L$, which satisfies that $\Skip(F)$ is the 2-component unlink.  On the other hand, $L=BD(K)$ is a boundary link, see for example \cite{Ci}, and so $L$ admits a C-complex $F'=F_1'\cup F_2'$ with no clasps.  Therefore, the collar neighborhood $A'$ appearing in the definition of $\Skip(F')$ is an embedded disjoint union of annuli and parametrizes an isotopy from $BD(K)$ to $\Skip(F')$.  Since $BD(K)$ is not isotopic to the unknot when $K$ is not unknotted, $\Skip(F)$ is not isotopic to $\Skip(F')$.  By Lemma~\ref{lem: skip is invariant}, $F$ and $F'$ are not related by any sequence of moves \ref{move: isotopy}, \ref{move: handle}, \ref{move: ribbon+push}, and \ref{move: pass through clasp}.
  
\end{example}

\begin{figure}[h]
  \centering\begin{subfigure}[b]{0.4\textwidth}
         \centering
\begin{tikzpicture}
    \node at (0,0) {\includegraphics[height=.15\textheight, angle=90]{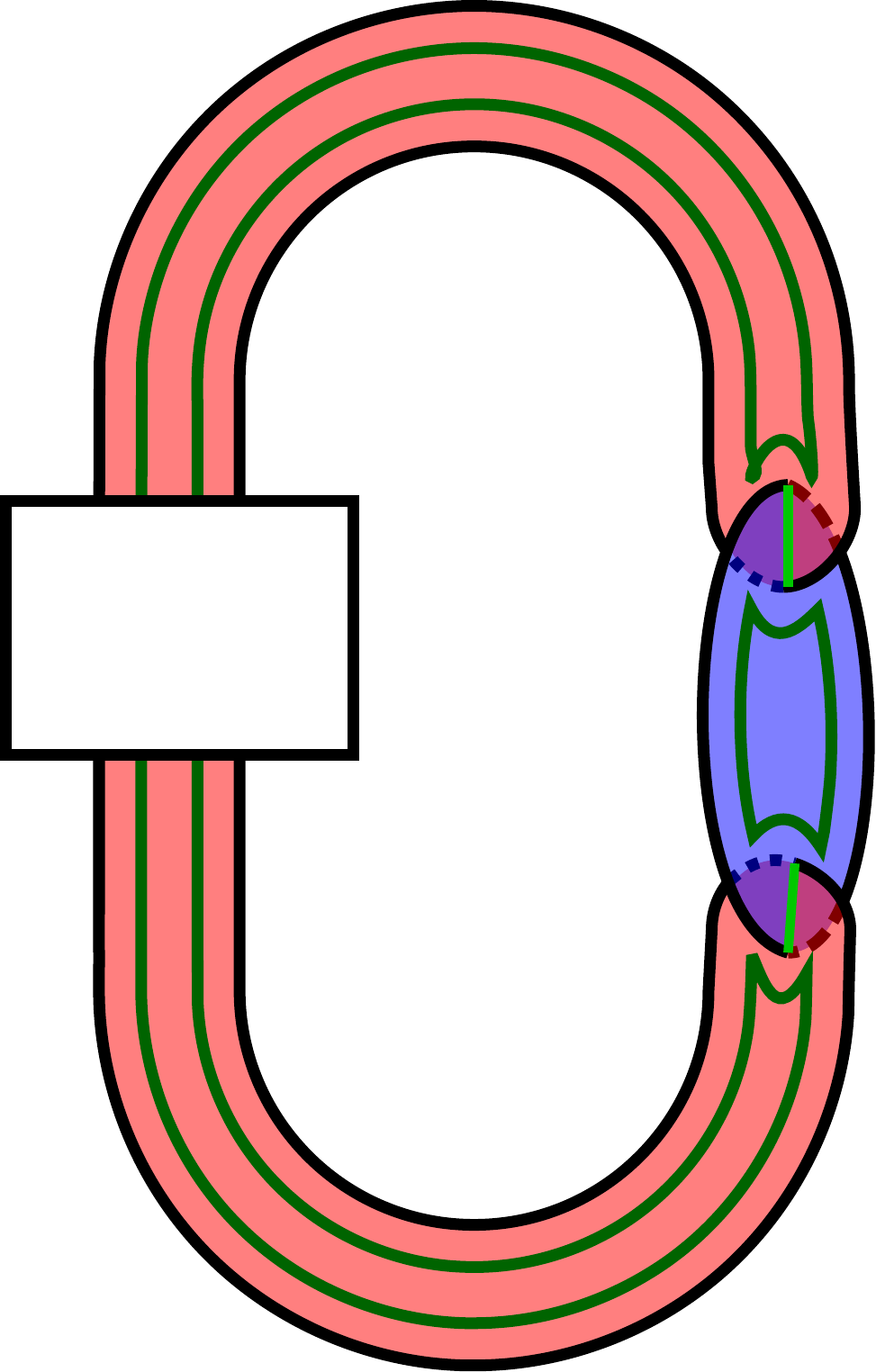}};
    \node at (-.15,-.6) {$K$};
\end{tikzpicture}
     \caption{}
        \label{fig: Bing double}
            \end{subfigure}
       \begin{subfigure}[b]{0.4\textwidth}
         \centering
         \begin{tikzpicture}
    \node at (0,0) {\includegraphics[height=.15\textheight, angle=90]{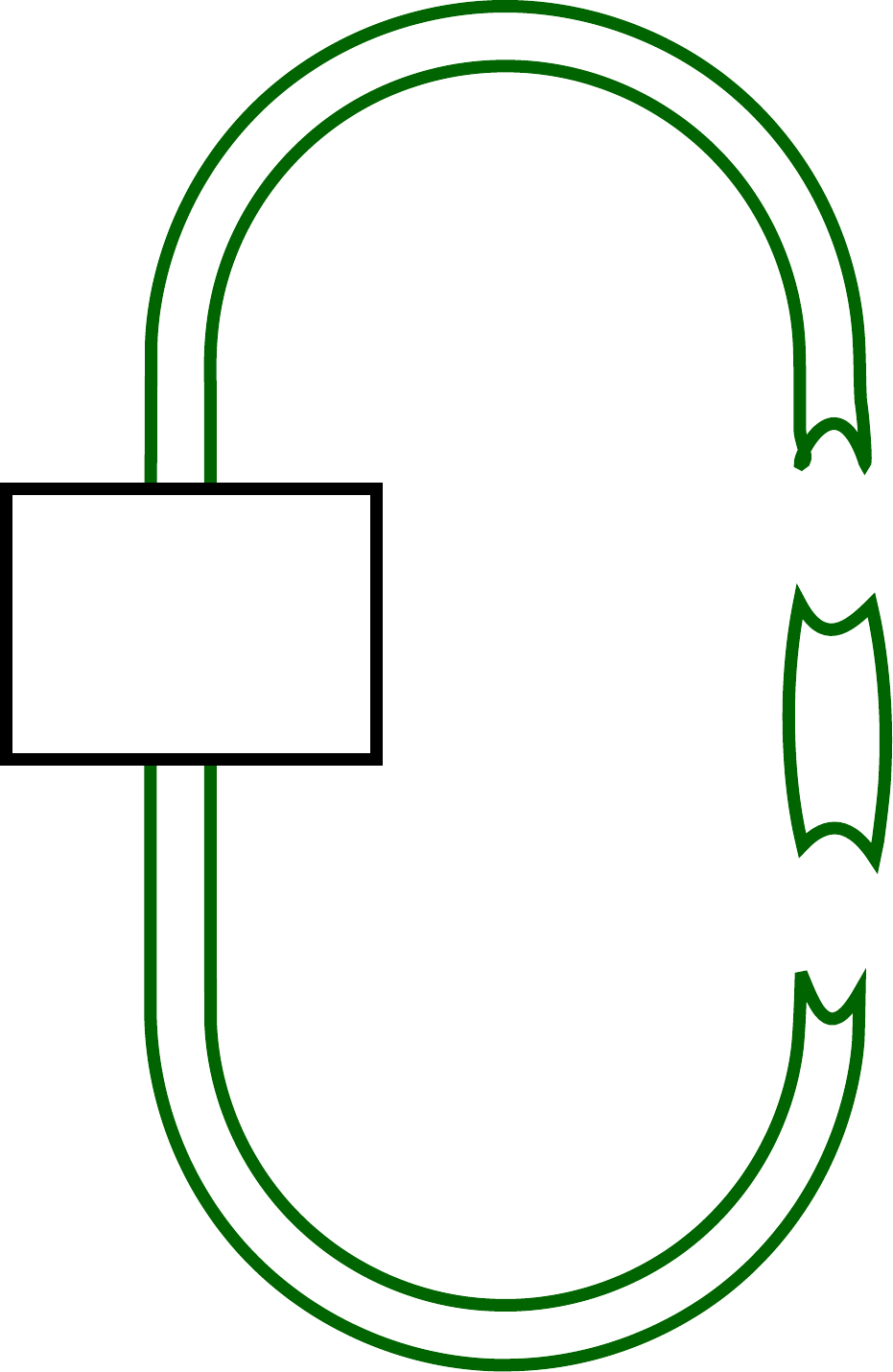}};
    \node at (-.15,-.6) {$K$};
\end{tikzpicture}
     \caption{}
        \label{fig: Bing double Skip}
            \end{subfigure}
     \hfill
     \caption{The Bing double of $K$ admits a C-complex $F$ for which $\Skip(F)$ is the unlink.}
         \label{fig: Bing double ex}
\end{figure}

We conclude this section with the following questions about the skip the clasps link, justifying further study.  

\begin{questions}
For any colored link $L$, set $\Skip(L) = \{\Skip(F)\mid F\text{ is a C-complex for }L\}$.  Does $\Skip(L)$ contain a split link for every $L$?  More generally, does $\Skip(L)$ contain every colored link $J$ bounding a C-complex with no clasps (known as a colored boundary link) and for which $J_i$ is isotopic to $L_i$ for all $i$?  Does there exist a link satisfying that $\Skip(L)$ contains only one link?  If $F$ and $F'$ are C-complexes for $L$ and  $\Skip(F)=\Skip(F')$ then are $F$ and $F'$ related by moves \ref{move: isotopy}, \ref{move: handle}, \ref{move: ribbon+push} and \ref{move: pass through clasp}?
\end{questions}

\section{Link invariants from C-complexes}\label{sect:(T4) consequences}

Lemma~\ref{false lemma}, which we have just seen to be false, is used in the papers \cite{CC}, \cite{CF}, and \cite{CimTur} to show that several quantities computed in terms of C-complexes are invariant under moves \ref{move: isotopy} through \ref{move: pass through clasp}, and thus are colored link invariants.  In \cite[Lemma 4]{CC}, this is done for a multivariable polynomial $\Omega_F$ associated to a C-complex $F$.  This polynomial is then shown to recover the Conway potential function. In \cite[Lemma 2.1]{CF}, this is done for the signature $\sigma_L$ and nullity $\eta_L$ functions appearing in that paper. In this section, we complete the proofs that $\Omega_F$, $\sigma_L$, and $\eta_L$ are invariants of colored links by verifying their invariance under the \ref{move: push along different arc} move.  
{In \cite[Section 7]{CimTur} Cimasoni and Turaev introduce a Seifert triple associated to a C-complex and use Lemma~\ref{false lemma} to conclude that it is an invariant of colored links in quasi-cylinders.  While we do not provide the details here, one can check that their Seifert triple is preserved by \ref{move:T4}.}

\subsection{The multivariable Conway potential function}
\label{sect:Conway}

In \cite{CC}, Cimasoni provides a formula for the multivariable Conway potential function in terms of the linking numbers of curves sitting on a C-complex.  While the fact that this formula gives an invariant of links can be deduced from Hartley's proof that the fact that the multivariable Conway potential function is an invariant \cite{Hartley83}, one of the goals of \cite{CC} is a proof of invariance more in line with the argument of Kauffman \cite{Kauffman81} of the invariance of the one variable Conway potential function.  

We  complete Cimasoni's proof of invariance, assuming Theorem~\ref{thm: corrected lemma}.  We begin by recalling Cimasoni's formula for the potential function in terms of a C-complex.  Given an $n$-colored link $(L,\sigma)$ bounding a C-complex $F$ and any choice of $\epsilon:\{1,\dots, n\}\to \{\pm1\}$ one can define a linking matrix analogous to the Seifert matrix.  Let $\alpha$ be a curve in $F$ which interacts nicely with the clasps of $F$.  In \cite{CC} $\alpha$ is called a \emph{loop}, and the reader is directed to \cite{CC}, \cite{CF}, \cite{CoopThesis} or \cite{Cooper} for a more complete discussion; see in particular \cite[Figure 2]{CF}.  One can then define a pushoff $\alpha^\epsilon$ of $\alpha$ by pushing $\alpha$ off of each surface $F_i$ in the $\epsilon_i$-normal direction.  The condition that $\alpha$ is a loop is required to arrange that these pushoffs agree near the clasps.  After picking a basis $\{\alpha_i\}$ for $H_1(F)$ consisting of loops, the $\epsilon$-\emph{linking matrix} is the square matrix with $(i,j)$-entry $(A_F^\epsilon)_{i,j} = \lk(\alpha_i, \alpha_j^\epsilon)$.  Cimasoni \cite{CC} gives a  polynomial $\Omega_F$ in $n$ variables in terms of these linking matrices and proves that $\Omega_F$ recovers the multivariable Conway potential function \cite[Lemma 5]{CC}.

\begin{equation}\label{eqn:Omega}\Omega_F(t_1, \dots, t_n) = \sgn(F)\Prod_{i=1}^{n} (t_i - t_i\inv)^{\chi(F \bk F_i)-1} \det(-A_F).\end{equation}

Here, $A_F = \Sum_{\epsilon} \epsilon(1)\cdots \epsilon(n)\cdot t_1^{\epsilon(1)} \cdots t_n^{\epsilon(n)} \cdot A_F^{\epsilon}$ and $\text{sgn}(F)$ is the product of the signs of all clasps of $F$.

\begin{lemma}[Lemma 4 in \cite{CC}]
Let $F$ and $F'$ be two C-complexes for isotopic colored links $(L,\sigma)$ and $(L', \sigma ')$. Then, 
$\Omega_F$ and $\Omega_{F'}$ are equal. 
\end{lemma}

\begin{proof}

Cimasoni proves that $\Omega$ is preserved by the C-complex moves \ref{move: isotopy}, \ref{move: handle}, \ref{move: ribbon+push},  and \ref{move: pass through clasp}. We complete this proof by showing the move \ref{move: push along different arc} also preserves $\Omega$. 

The \ref{move: push along different arc} move preserves $\sgn(F)$ as well as the Euler characteristics appearing in equation \pref{eqn:Omega}. Thus, it suffices to show that $\det(-A_F) = \det(-A_{F'})$.  Without loss of generality, we assume this \ref{move: push along different arc} move is between $F_1$ and $F_2$ as in Figure~\ref{fig: A_F matrices}. 

We begin with a choice of basis elements for $H_1(F)$, taking care to choose a basis including $\alpha_1, \alpha_2, \alpha_3, $ and $\alpha_4$, as pictured in Figure \ref{fig:A_f}. We assign the positive side of the surfaces $F_1$ and $F_2$ to be facing outwards. We note the curve $\alpha_1$ is fully contained in the portion of surface $F_2$ pictured in Figure \ref{fig:A_f}. We now compute $A_F$ with respect to this basis.

\begin{figure}
     \centering
     \begin{subfigure}[b]{0.4\textwidth}
         \centering
         \includegraphics[scale=.3]{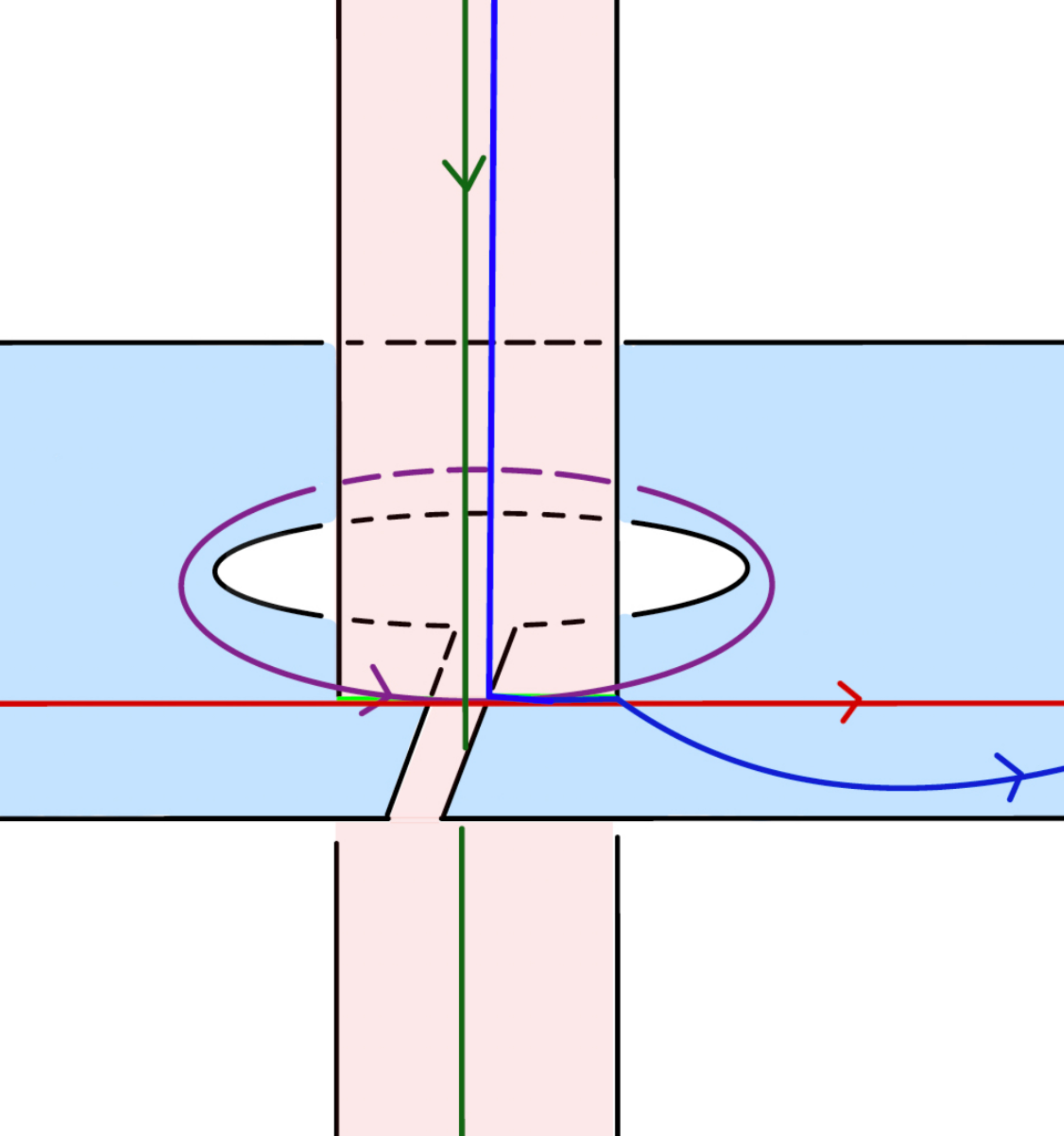}
         \put(-150,95){$\alpha_1$}
         \put(-109,135){$\alpha_2$}
         \put(-40,75){$\alpha_3$}
         \put(-40,40){$\alpha_4$}
         \put(-133, 150){$F_1$}
         \put(-25,133){$F_2$}
         \caption{Basis elements for $H_1(F)$}
         \label{fig:A_f}
     \end{subfigure}
     \begin{subfigure}[b]{0.4\textwidth}
         \centering
         \includegraphics[scale=.23]{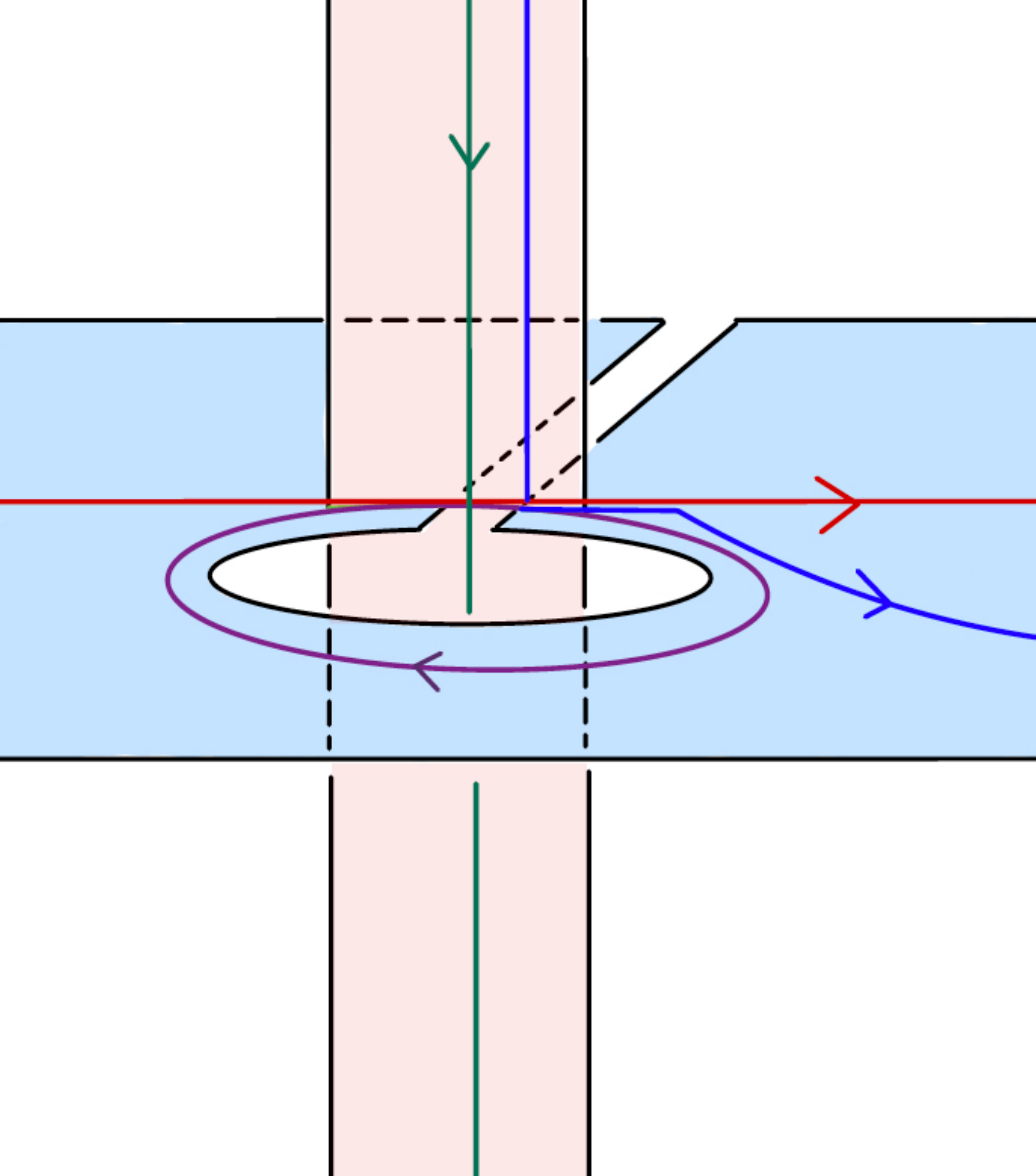}
         \put(-145,72){$\alpha_1'$}
         \put(-105,140){$\alpha_2'$}
         \put(-40,110){$\alpha_3'$}
         \put(-40,75){$\alpha_4'$}
         \put(-138, 150){$F_1'$}
          \put(-25,136){$F_2'$}
         \caption{Basis elements for $H_1(F')$}
         \label{fig:A_F'}
     \end{subfigure}

        \caption{}
        \label{fig: A_F matrices}
\end{figure}

\[A_F = \left(
\begin{blockarray}{cccc|cccc}
\begin{block}{cccc|cccc@{\hspace*{5pt}}}
0 & (-t_1t_2+t_1t_2\inv)z & 0 & (-t_1t_2)z & 0 & 0 & \dots & 0 \\
(-t_1\inv t_2\inv + t_1\inv t_2)z & a_{22} & a_{23} & a_{24} & \BAmulticolumn{4}{c}{\multirow{3}{*}{$B$}} \\
0 & a_{32} & a_{33} & a_{34} & &&& \\
(-t_1\inv t_2\inv)z & a_{42} & a_{43} & a_{44} & &&&\\
 \hline
0 & \BAmulticolumn{3}{c|}{\multirow{4}{*}{$C$}} & \BAmulticolumn{4}{c}{\multirow{4}{*}{$A$}} \\
0 & &&&&&&& \\
\vdots & &&&&&&&&\\
0 &&&&&&&&\\
\end{block}
\end{blockarray}
 \right)
\]
Here, $z=(t_3 - t_3\inv)(t_4 - t_4\inv) \cdots(t_n - t_n\inv)$. The relationship between first row and column of $A_F$ reflects the fact that $(A^{\epsilon})^T = A^{-\epsilon}$, so that $A_F(t_1,\dots, t_n)^T = A_F(-t_1^{-1},\dots, -t_n^{-1})$.
To compute $A_{F'}$, we use the basis elements for $H_1(F')$  pictured in Figure \ref{fig:A_F'}. 

\[A_{F'} = \left(
\begin{blockarray}{cccc|cccc}
\begin{block}{cccc|cccc@{\hspace*{5pt}}}
0 & (-t_1\inv t_2 + t_1\inv t_2\inv)z & 0 & (-t_1\inv t_2)z & 0 & 0 & \dots & 0 \\
(-t_1t_2\inv+ t_1t_2)z & a_{22} & a_{23} & a_{24} & \BAmulticolumn{4}{c}{\multirow{3}{*}{$B$}} \\
0 & a_{32} & a_{33} & a_{34} & &&& \\
(-t_1t_2\inv)z & a_{42} & a_{43} & a_{44} & &&& \\ \hline
0 & \BAmulticolumn{3}{c|}{\multirow{4}{*}{$C$}} & \BAmulticolumn{4}{c}{\multirow{4}{*}{$A$}} \\
0 & &&&&&&& \\
\vdots & &&&&&&&&\\
0 & &&&&&&&\\
\end{block}
\end{blockarray}
\right) \]

Basic row and column operations show that $\det(-A_F) = \det(-A_{F'})$; multiplying the first row of $A_F$ by $t_1^{-2}$ and the first column of $A_F$ by $t_1^2$ results in $A_{F'}$. Thus, $\Omega_F = \Omega_{F'}$.
\end{proof}

\subsection{The signature and nullity of a colored link}\label{sect:signature}

Given an $n$-colored link $(L, \sigma)$, Cimasoni and Florens define the \emph{signature} and \emph{nullity} functions of $L$ \cite{CF}.  Let $F$ be a C-complex for $L$. For each choice of $\epsilon: \{1, \dots, n\} \rightarrow\{\pm1 \}$ and a fixed basis for $H_1(F)$, compute the linking matrix $A^{\epsilon}_F$ as in subsection~\ref{sect:Conway}. Given a choice of $\vec{\omega} = (\omega_1, \dots, \omega_n)\in \mathbb{C}^n$ such that $|\omega_i|=1$ and $\omega_i \ne 1$, for all $i$, define
$$H_F(\vec{\omega}) = \Sum_{\epsilon} \left( \prod_{i=1}^n \left(1-\bar{\omega}_i^{\epsilon(i)}\right)\right)A^{\epsilon}_F.$$

The fact that $H_F(\vec{\omega})$ is Hermitian follows from the observation that  $A^{-\epsilon}_F = (A^{\epsilon}_F)^T$. Cimasoni and Florens define the \emph{signature} of $L$, $\sigma_L(\vec{\omega})$, to be the signature of $H_F(\vec{\omega})$ and the \emph{nullity} of $L$ to be $\eta_L(\vec{\omega})= \text{null}(H_F(\vec{\omega}))+\beta_0(F)-1$. Here, $\beta_0$ denotes the zero'th Betti number. They then give the following theorem, which is proved by demonstrating that signature and nullity are preserved by the C-complex moves \ref{move: isotopy}, \ref{move: handle}, \ref{move: ribbon+push}, and \ref{move: pass through clasp} and appealing to Lemma~\ref{false lemma}, which is false. We complete their proof by showing that the move \ref{move: push along different arc} also preserves signature and nullity.

\begin{theorem}[Theorem 2.1 of \cite{CF}]
The signature $\sigma_L$ and nullity $\eta_L$ do not depend on the choice of C-complex for a colored link $L$. They are well-defined isotopy invariants of the colored link. 
\end{theorem}

\begin{proof}
Let $F$ be a C-complex for an $n$-colored link $L$, and let $F'$ be a C-complex related to $F$ by a \ref{move: push along different arc} move.

Using the basis elements from Figure \ref{fig: A_F matrices} for $H_1(F)$ and $H_1(F')$, and setting $v=(1-\omega_3)\cdots(1-\omega_n)$, we compute the associated Hermitian matrices $H_F(\vect\omega)$ and $H_{F'}(\vect\omega)$.

\[H_F(\vec{\omega}) = (1-\bar{\omega}_1)\cdots(1-\bar{\omega}_n)\left(
\begin{blockarray}{cccc|cccc}
\begin{block}{cccc|cccc@{\hspace*{5pt}}}
0 & -(1-\omega_2)v & 0 & -v & 0 & 0 & \dots & 0 \\
\omega_1(1-\omega_2)v & b_{22} & b_{23} & b_{24} & \BAmulticolumn{4}{c}{\multirow{3}{*}{$D$}} \\
0 &b_{32} & b_{33} & b_{34} & &&& \\
-(\omega_1\omega_2)v &b_{42} & b_{43} & b_{44} & &&&\\ \hline
0 & \BAmulticolumn{3}{c|}{\multirow{4}{*}{$E$}} & \BAmulticolumn{4}{c}{\multirow{4}{*}{$F$}} \\
0 & &&&&&&& \\
\vdots & &&&&&&&&\\
0 & &&&&&&&\\
\end{block}
\end{blockarray} \right),
\]
\[H_{F'}(\vec{\omega}) = (1-\bar{\omega}_1)\cdots(1-\bar{\omega}_n)\left(
\begin{blockarray}{cccc|cccc}
\begin{block}{cccc|cccc@{\hspace*{5pt}}}
0 & -\omega_1(1-\omega_2)v & 0 & -\omega_1v& 0 & 0 & \dots & 0 \\
(1-\omega_2)v& b_{22} & b_{23} & b_{24} & \BAmulticolumn{4}{c}{\multirow{3}{*}{$D$}} \\
0 &b_{32} & b_{33} & b_{34} & &&& \\
-\omega_2v& b_{42} & b_{43} & b_{44} & &&&\\ \hline
0 & \BAmulticolumn{3}{c|}{\multirow{4}{*}{$E$}} & \BAmulticolumn{4}{c}{\multirow{4}{*}{$F$}} \\
0 & &&&&&&& \\
\vdots & &&&&&&&&\\
0 & &&&&&&&\\
\end{block}
\end{blockarray} \right).
\]

By performing the elementary row and column operations of multiplying the first row of $H_F(\vec{\omega})$ by $\omega_1$ and the first column of $H_F(\vec{\omega})$ by $\bar{\omega}_1$, we transform $H_F(\vec{\omega})$ into $H_{F'}(\vec{\omega})$. This matrix operation preserves the signature and dimension of the nullspace. Further, the \ref{move: push along different arc} move does not change the number of connected components of the C-complex, so $\beta_0(F) = \beta_0(F')$. Thus, $\sigma_L(\vec{\omega})$ and $\eta_L(\vec{\omega})$ are invariants of the colored link $L$. 

\end{proof}

\section{Correcting Lemma~\ref{false lemma}: the proof of Theorem~\ref{thm: corrected lemma}}\label{sect: corrected lemma}

 In this section we prove Theorem \ref{thm: corrected lemma}, thus providing a complete set of moves relating any pair of C-complexes for a fixed link.  We start with a general summary of the proof and then spend the bulk of the section on the technical details.  We have made an effort to include every detail, even at the expense of brevity; we feel that this choice is appropriate when correcting an error that has gone nearly two decades without detection.  Our arguments will pass from C-complexes into the setting of unions of surfaces with more flexibility in their intersections. 

\begin{definition}\label{defn: surface system}
A union of embedded compact oriented surfaces $F=F_1\cup\dots\cup F_n$ with no closed components bounded by the colored link $L=L_1\cup\dots\cup L_n$ is called a \emph{surface system} for $L$ if for all $i, j, k, \ell\in \{1,\dots, n\}$ 
\begin{enumerate}[label=\Roman*] 
\item\label{CP:int} : $F_i$ is transverse to $F_j$,
\item\label{CP:bdry} : $F_i$ is transverse to $L_j=\bdry F_j$,
\item\label{CP:triple}  : $F_i$ is transverse to $F_j\cap F_k$,
\item\label{CP:bdryTriple}  : $F_i\cap \bdry (F_j\cap F_k) = \emptyset$,
\item\label{CP:quadruple}  : $F_i\cap F_j\cap F_k\cap F_\ell = \emptyset$.
\end{enumerate}
\end{definition}

For any surface system $F=F_1\cup\dots\cup F_n$, the stated transversality implies that for any $i\neq j$, $F_i\cap F_j$ is a 1-manifold.  Therefore any component of $F_i\cap F_j$ has one of three types of intersection: clasps, ribbons, and circles.
\begin{itemize}
\item A component $\alpha$ of $F_i\cap F_j$ is a \emph{clasp intersection} if $\alpha$ is an arc with one endpoint in $\bdry F_i$ and the other in $\bdry F_j$; see Figure~\ref{fig: clasps}.
\item A component $\alpha$ of $F_i\cap F_j$ is a \emph{ribbon intersection} if $\alpha$ is an arc with both endpoints in $\bdry F_i$ or both endpoints in $\bdry F_j$;  a ribbon intersection appears in Figure~\ref{fig: (T4)B}.
\item Otherwise a component $\alpha$ of $F_i\cap F_j$ is a \emph{circle} interior to each of $F_i$ and $F_j$.
\end{itemize}    
For any distinct $i,j, \text{ and } k$, $F_i\cap F_j\cap F_k$  is a finite set of points.  Such a point is called a \emph{triple point} of $F$.  

Let $F$ be a smooth compact oriented surface.  By an isotopy $\Phi^t:F\to S^3$ we mean a smooth map $\Phi:F\times[0,1]\to S^3$ so that $\Phi|_{F\times\{t\}}$ is a smooth embedding for all $t$.  We will use $\Phi^t(x)$ to denote $\Phi(x,t)$ and $F^t$ to denote $\Phi^t(F)$.  An isotopy from a (compact, oriented) surface $F$ to $F'$ in $S^3$ is an isotopy $\Phi^t:F\to S^3$ where $\Phi^0:F\to S^3$ is the inclusion map and $\Phi^1(F)=F'$.  

Let $F=F_1\cup\dots\cup F_n$ and $F'=F'_1\cup\dots\cup F'_n$ be C-complexes, or more generally surface systems.  For $i=1,\dots,n$ let $\Phi_i:F_i\times[0,1] \to S^3$ be an isotopy from $F_i$ to $F'_i$.  If the union of restrictions of the various $\Phi_i$ to $\bdry F_i$ gives an isotopy of colored links, then we call $\Phi =\{\Phi_i\}= \{\Phi_1, \dots, \Phi_n\}$ a \emph{collection  of isotopies} from $F$ to $F'$ and we will denote by $F^t$ the union of surfaces $F^t_1\cup\dots\cup F^t_n$ where $F_i^t = \Phi^t_i(F_i)$.  

Given such a collection of isotopies, any $t\in [0,1]$ for which $F^t$ fails to be a surface system is called a \emph{critical time}.  A perturbation of $\Phi$ allows us to arrange that there are only finitely many critical times and that at any critical time there is a unique \emph{critical point} at which exactly one of the  conditions of Definition~\ref{defn: surface system} fails.  For $A \in\{\text{\ref{CP:int},~\ref{CP:bdry},~\ref{CP:triple},~\ref{CP:bdryTriple},~\ref{CP:quadruple}}\}$, we say a critical point is of \emph{type $A$} if the corresponding condition of a surface system fails.

We are now ready to provide an outline of the proof of Theorem~\ref{thm: corrected lemma}.  During this outline we will make reference to several facts.  Afterwards, we  provide precise statements and proofs.  

\begin{proof}[Proof of Theorem~\ref{thm: corrected lemma}]  
Suppose $F=F_1 \cup F_2 \cup \cdots \cup F_n$ and $F'=F_1' \cup F_2' \cup \cdots \cup F_n'$ are C-complexes for two isotopic colored links, $L$ and $L'$.  We need to exhibit a sequence of moves \ref{move: isotopy} through \ref{move: push along different arc} starting at $F$ and ending at $F'$.  After some handle additions staying in the category of C-complexes, we  arrange that $F_i$ is isotopic to $F_i'$ for all $i=1,\dots, n$ (Lemma~\ref{lem:handle}). In other words, after changing $F$ and $F'$ by a collection of \ref{move: isotopy} and \ref{move: handle} moves, there exists a collection of isotopies, $\Phi$ from $F$ to $F'$. By perturbing $\Phi$, we arrange that $\Phi$ has only finitely many critical times and that at each critical time there is a single critical point.   As we shall see in Proposition~\ref{prop: get ambient isotopy}, if $[a,b] \subset [0,1]$ contains no critical times, then there is an ambient isotopy $\Psi: S^3\times[a,b] \rightarrow S^3$ such that $F^t=\Psi(F^a)$ for all $t \in [a,b]$.

In Figure~\ref{fig: push along an arc} we see a move which replaces $\Phi$ by another sequence of isotopies with one fewer type~\ref{CP:int} critical point.  This move is called a push along an arc and is described in Definition~\ref{def:pushmove}.  Using this move we remove all critical points of type~\ref{CP:int}, \ref{CP:triple}, and \ref{CP:quadruple} in Proposition~\ref{lem: no interior CP's}.

In Proposition~\ref{pairingcp} further modify $\Phi$ to eliminate all circle intersections and organize the remaining type~\ref{CP:bdry} and \ref{CP:bdryTriple} critical times into pairs $\{(a_i, b_i)\}$ so that for all $t\in (a_i,b_i)$, $F^t$ contains exactly one ribbon intersection or exactly one triple point but not both.  Away from these subintervals, $F^t$ is a C-complex.  Thus, we need only analyze how such a pair of critical times changes the C-complex.  The Type~\ref{CP:bdry} critical times which increase the number of ribbon intersections appear in Figure~\ref{fig: ribbon appears}.  Following any one of these with the reverse of another results in a \ref{move:T2} or \ref{move:T4} move.   The paired type \ref{CP:bdryTriple} critical times result in a sequence of the moves \ref{move: isotopy} though \ref{move: push along different arc}; See Figures~\ref{fig: type IV paired} and \ref{fig: Find T3 move}.  This completes the proof.
\end{proof}

The remainder of this section is devoted to filling in the details above.

\begin{lemma}\label{lem:handle}
Let $F=F_1\cup\dots\cup F_n$ and $F'=F'_1\cup\dots\cup F'_n$ be C-complexes for a link $L$.  There exist C-complexes $G=G_1\cup\dots\cup G_n$ and $G'=G'_1\cup\dots\cup G'_n$ for $L$ so that $F$ and $F'$ are transformed to $G$ and $G'$, respectively, by a sequence of \ref{move: isotopy} and \ref{move: handle} moves such that $G_i$ is isotopic to $G_i'$ for each $i = 1,\dots, n$.
\end{lemma}
\begin{proof}

By \cite[Chapter 1 Proposition 2.7]{CoopThesis}, for any $i$, there is a third Seifert surface $V_i$ obtained from each of $F_i$ and $F'_i$ by a sequence of ambient isotopies and addition of hollow handles.  By conjugating by an ambient isotopy the arcs along which the hollow handles are added, we  arrange that each sequence consists of all of the handle additions followed by an ambient isotopy.  A further isotopy allows us to arrange that these arcs, and so the resulting hollow handles are disjoint from $F_j$ and $F'_j$ for all $j$.  

Since the addition of a hollow handle along an arc disjoint from each other component of a C-complex is precisely the \ref{move: handle} move, this proves the result.   

\end{proof}

Thus, it suffices to consider C-complexes $F=F_1\cup\dots \cup F_n$ and $F'=F'_1\cup\dots F'_n$ which are related by a collection of isotopies.  This collection of isotopies need not extend to an isotopy of the whole C-complex, and indeed $F^t=\Phi^t(F)$ may leave the category of C-complexes.  The idea of the proof from here on is to restrict how these isotopies change a C-complex, and more generally any surface system.  Proposition~\ref{prop: get ambient isotopy} reveals that away from  critical times, a collection of isotopies extends to an ambient isotopy.  This result is used implicitly in \cite[Lemma 5.1]{Cooper} and \cite[Lemma 3]{CC}.   We state it explicitly here and prove it in subsection~\ref{subsect: technical iso extn}.  

\begin{proposition}\label{prop: get ambient isotopy}
Let $F=F_1\cup\dots \cup F_{k}$ be a  be a surface system.  Let $\Phi$ be a collection of isotopies from $F$ to another surface system.  Suppose also that for all $t\in [0,1]$ we have that 
  $F^t$  is a surface system.  Then there is an ambient isotopy $\Psi:S^3\times[0,1] \to S^3$ with $\Psi^t(F) = F^t$.  
\end{proposition}

It remains to control what happens at critical times. We first eliminate critical points of types \ref{CP:int}, \ref{CP:triple}, and \ref{CP:quadruple}.  A key  is the {push along an arc} move, which we now describe.

\begin{figure}
     \centering
     \begin{subfigure}[b]{0.8\textwidth}
         \centering
         \begin{tikzpicture}
         \node at (0,0){\includegraphics[width=.3\textwidth]{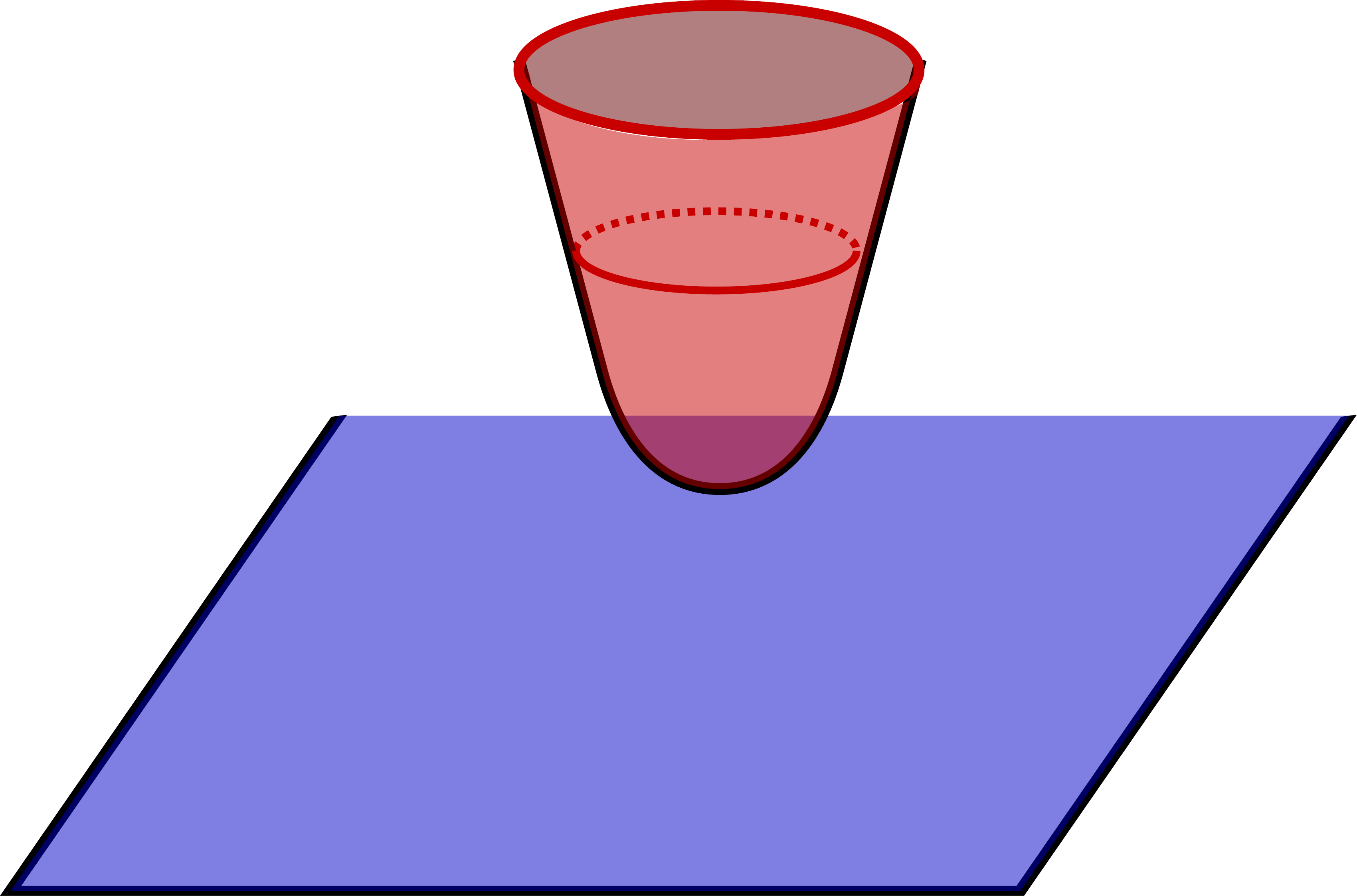}};]
         \node at (4,-.2){\includegraphics[width=.3\textwidth]{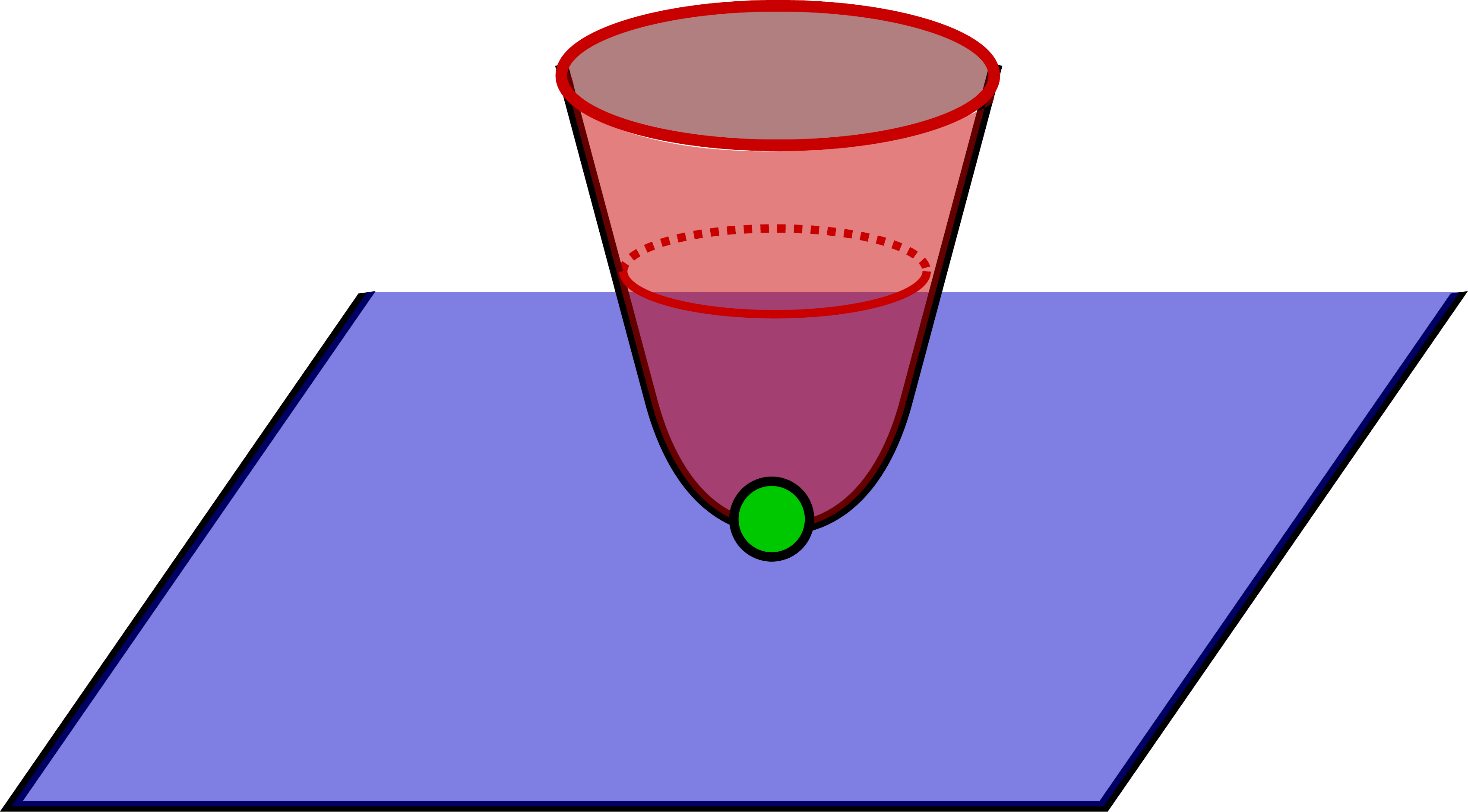}};]
         \node at (8,-.5){\includegraphics[width=.3\textwidth]{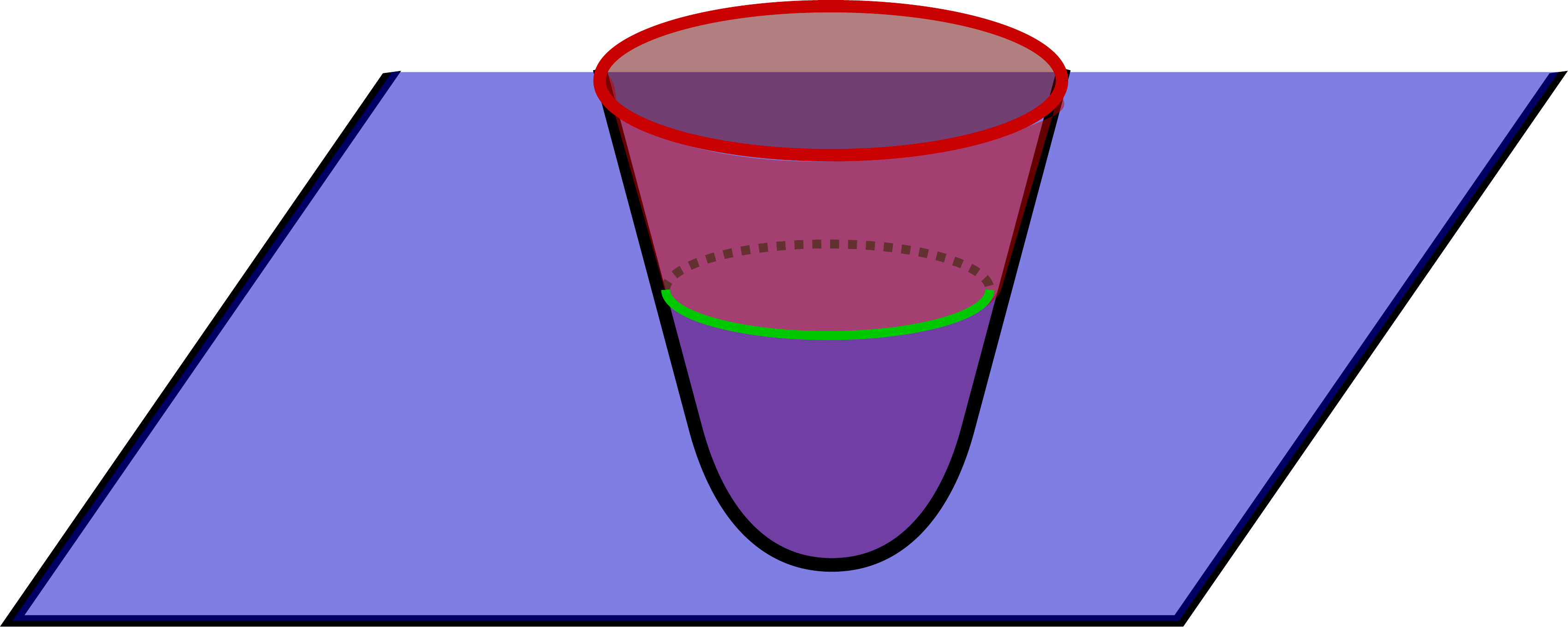}};]
         \end{tikzpicture}
         \caption{The moment before, during, and after a type~\ref{CP:int} critical point.}
         \label{fig: interior crit pt}
     \end{subfigure}

     \begin{subfigure}[b]{0.9\textwidth}
         \centering
         \begin{tikzpicture}
         \node at (0,0){\includegraphics[width=.25\textwidth]{InteriorCP1.pdf}};]
         \draw[black, thick](.1,-.38)--({2.35-3},-1.2);
         \node[black, right] at (-.3, -.8) {$\alpha$};
         \node at (3,0){\includegraphics[width=.25\textwidth]{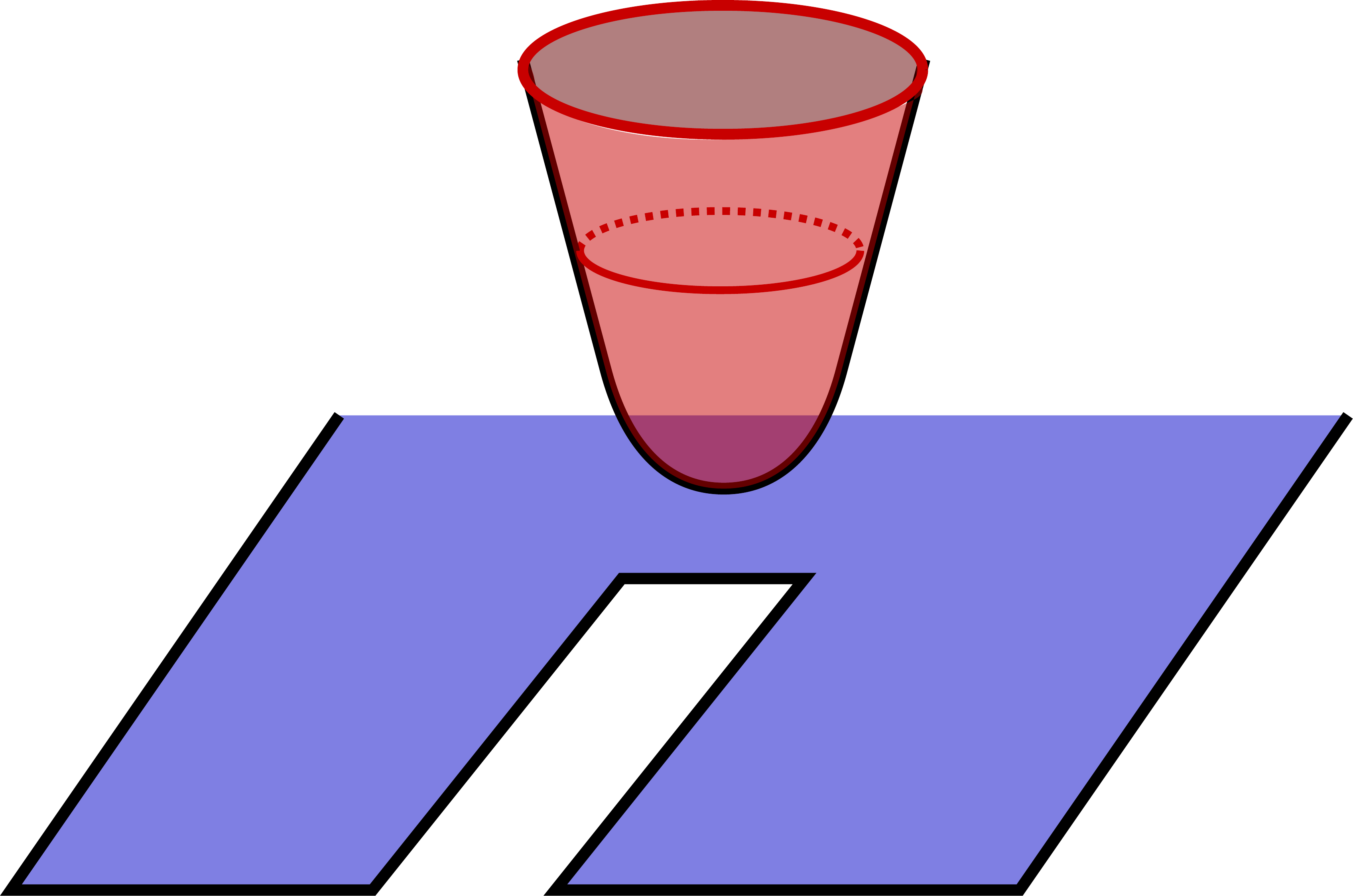}};]
         \node at (6,-.2){\includegraphics[width=.25\textwidth]{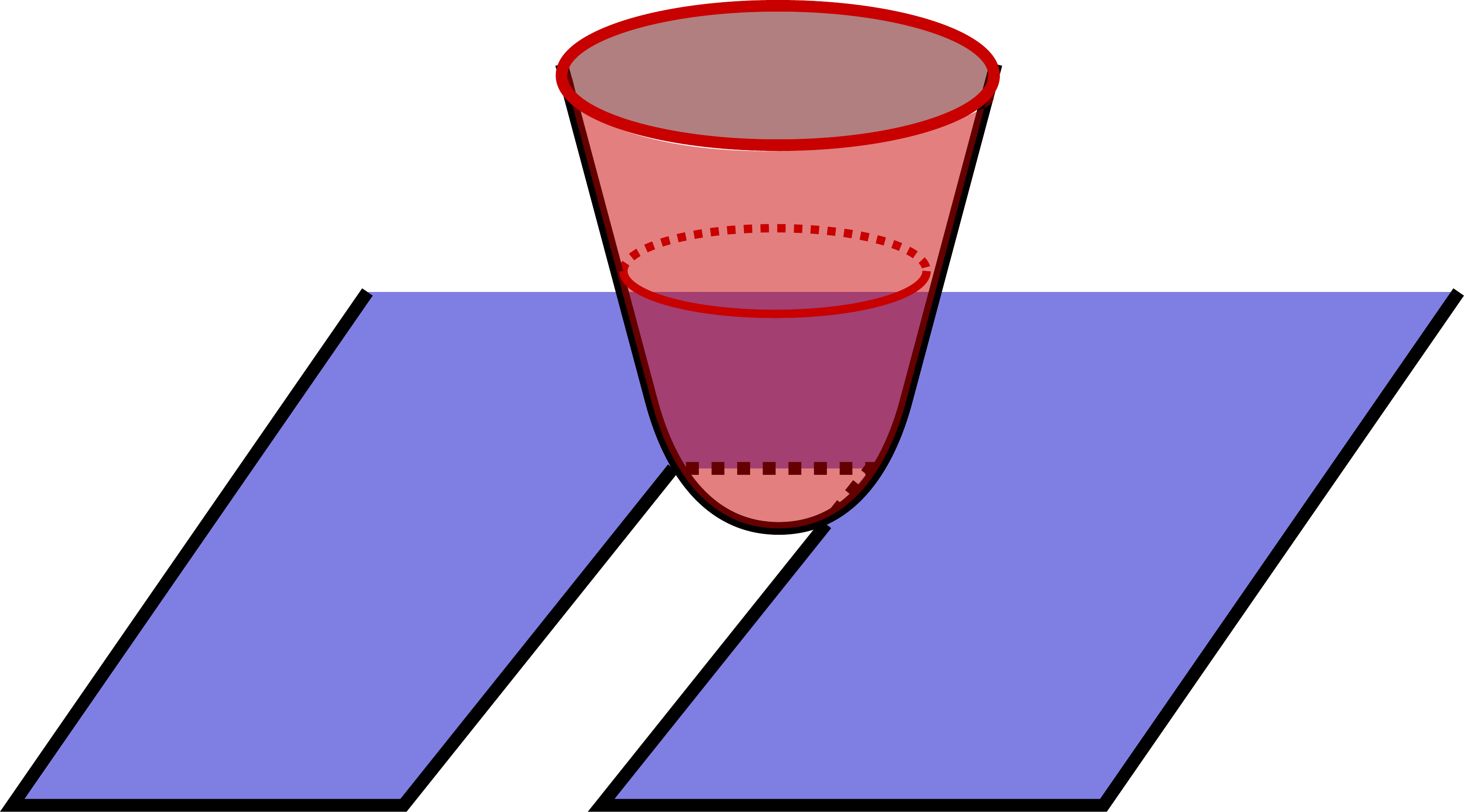}};]
         \node at (9,-.5){\includegraphics[width=.25\textwidth]{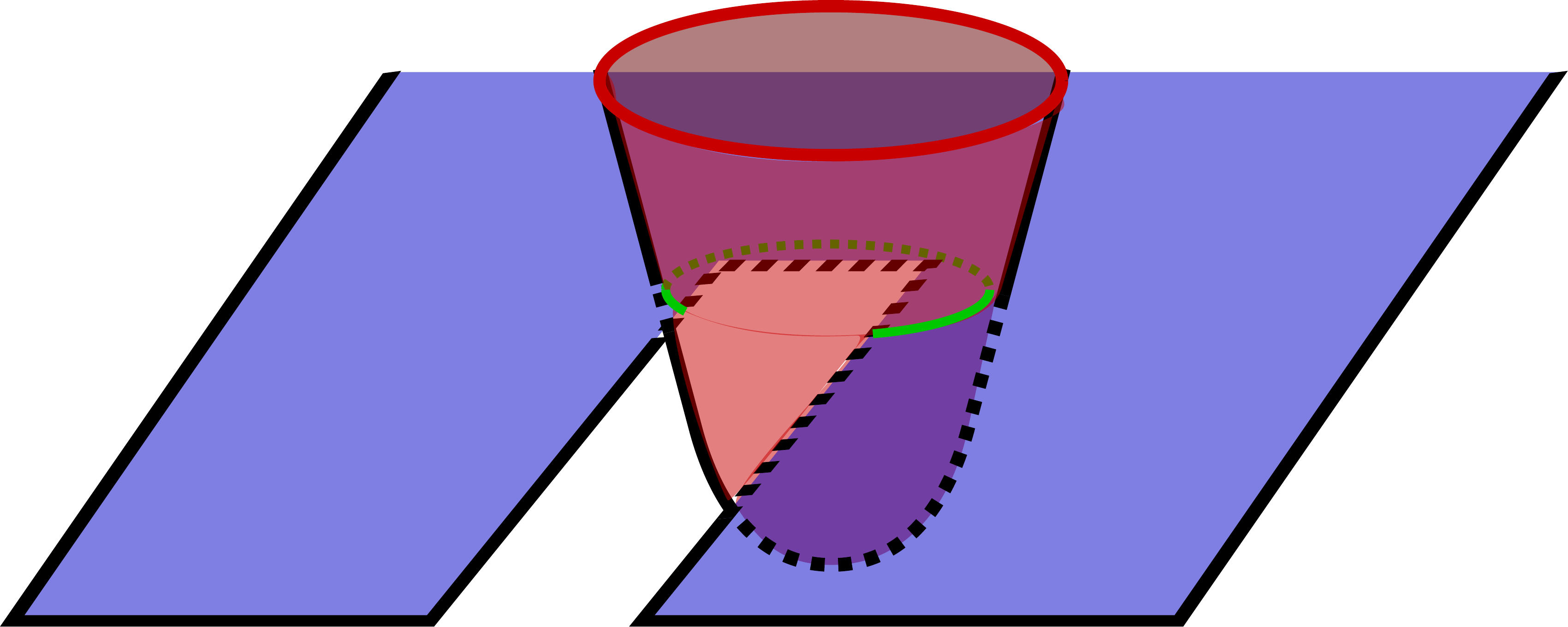}};]
         \node at (12,-.5){\includegraphics[width=.25\textwidth]{InteriorCP3.pdf}};]
         \end{tikzpicture}
         \caption{A push along an arc $\alpha$ replaces this type~\ref{CP:int} critical point by type~\ref{CP:bdry} critical points.}
         \label{fig: push subfig}
     \end{subfigure}
     
        \caption{}
        \label{fig: push along an arc}
\end{figure}

\begin{definition}\label{def:pushmove}
Let $\Phi$ be a collection of isotopies between surface systems $F$ and $F'$.  Pick an embedded arc  $\alpha\subseteq F_i$ with one endpoint in $\bdry F_i$ and which otherwise is interior to $F_i$ as well as an interval $[t_0,t_1]\subseteq (0,1)$.   Let $U\subseteq F_i$ be a regular neighborhood of $\alpha$  disjoint from 
$(\Phi_i^t)^{-1}(\bdry F_j^t)$
 for all $t\in [a,b]$ and all $j\neq i$.   Let $r^t:F_i\to F_i-U$ be an isotopy, supported in a small neighborhood of $U$, with $r^0=\Id_{F_i}$ and $r^1(F_i) \subseteq F_i-U$.   A new collection of isotopies, $(\Phi')^t$ is given by concatenating  $\Phi^t$ over $[0,t_0]$, $\Phi^{t_0}\circ r^t$ over $[0,1]$, $\Phi^t\circ r^1$ over $[t_0, t_1]$, $\Phi^{t_1}\circ r^{1-t}$ from $t=0$ to $t=1$, and $\Phi^t$ over $[t_1, 1]$. 
 A schematic appears in Figure~\ref{fig: push along an arc} and an explicit formula appears in \pref{eqn:push} below. 
The \emph{push along an arc move} replaces $\Phi$ by $\Phi'$.
\end{definition}

\begin{equation}\label{eqn:push}
\begin{array}{rcl}
(\Phi')_i(x,t) &=& \begin{cases}
\Phi_i(x,t) &
\text{if } 0\le t\le t_0-\epsilon
\\
\Phi_i\left( r^t\left(x, \frac{1}{\epsilon}(t-t_0+\epsilon ) \right), t_0-\epsilon\right)
&\text{if }t_0-\epsilon\le t\le t_0\\
\Phi_i\left(r^t(x,1), \frac{t_1-t_0+2\epsilon}{t_1-t_0}(t-t_0)+t_0-\epsilon \right)
&\text{if }t_0\le t\le t_1\\
\Phi_i\left(r^t\left(x,\frac{-1}{\epsilon}(t-t_1-\epsilon)\right), t_1+\epsilon\right)
&\text{if }t_1\le t\le t_1+\epsilon
\\
\Phi_i(x,t) &\text{if } t_1+\epsilon\le t\le 1
\end{cases}
\\
\text{For }j\neq i,~(\Phi')_j(x,t) &=& \begin{cases}
\Phi_j(x,t) &\text{if } 0\le t\le t_0-\epsilon\\
\Phi_j\left(x, t_0-\epsilon\right)
&\text{if }t_0-\epsilon\le t\le t_0\\
\Phi_j\left(x, \frac{t_1-t_0+2\epsilon}{t_1-t_0}(t-t_0)+t_0-\epsilon \right)
&\text{if }t_0\le t\le t_1\\
\Phi_j\left(x, t_1+\epsilon\right)
&\text{if }t_1\le t\le t_1+\epsilon
\\
\Phi_j(x,t) &\text{if } t_1+\epsilon\le t\le 1
\end{cases}
\end{array}
\end{equation}
There are some immediate observations.
\begin{itemize}

\item As demonstrated in Figure~\ref{fig: push along an arc}, if $\alpha$ contains a type~\ref{CP:int} critical point of $\Phi$ at time $t\in [t_0,t_1]$ then pushing along $\alpha$ over $[t_0,t_1]$ reduces the number of type~\ref{CP:int} critical points by at least one.  The push along an arc move can also be used to eliminate critical points of type~\ref{CP:triple} and \ref{CP:quadruple}.

\item Pushing along an arc introduces no new critical points of type~\ref{CP:int}, \ref{CP:triple}, or \ref{CP:quadruple}, although it can introduce critical points of type~\ref{CP:bdry} and \ref{CP:bdryTriple}.  
\end{itemize}

As a consequence, given any union of isotopies from one surface system to another, we can use the push along an arc move to eliminate all but two types of critical times/points.  

\begin{proposition}\label{lem: no interior CP's}
Let $F=F_1\cup F_2\cup \dots \cup F_n$ and $F'=F_1'\cup F_2'\cup \dots \cup F_n'$ be surface systems which are related by a collection of isotopies.
 There is a new collection of isotopies $\Phi'$ from $F$ to $F'$ which has only critical points of type~\ref{CP:bdry} and \ref{CP:bdryTriple}.
\end{proposition}

Thus, it suffices to explore collections of isotopies between C-complexes which have only critical times of types~\ref{CP:bdry} and \ref{CP:bdryTriple}  and  see how these change the surface system.
Using the push along an arc move, we may further control what intersections appear.  While we do so formally in Proposition~\ref{pairingcp}, we provide an idea here.  Since there are no critical points of type~\ref{CP:int}, the only way that a circle intersection can appear is when the two endpoints of a ribbon intersection join together at a type~\ref{CP:bdry} critical point.  By performing a push along an arc crossing this ribbon, we may split this ribbon before that critical time occurs, thus preventing the circle from ever appearing.  Therefore, we can avoid circle intersections in a collection of isotopies.  While we cannot prevent ribbon intersections from appearing in $F^t$, we can split a ribbon intersection into two clasps with a push along an arc as soon as that ribbon appears.  Similarly, we may eliminate triple points immediately after they appear.  There are two technical difficulties.  First, the various circles, ribbons, and triple points in $F^t$ may change location as $t$ changes.  The intersection we would like to simplify by a push along an arc may escape that arc.  Secondly, the various preimage points $(\Phi^t_i)^{-1}(\bdry F_j^t)\subseteq F_i$ will also wander about.  If one of them passes through the arc we use to push, then the resulting $\Phi'$ will no longer restrict to an isotopy on $\bdry F$.
  The following proposition resolves these issues and gives  control over the intersections appearing in $F^t$.

\begin{proposition}\label{pairingcp}
Let $F$ and $F'$ be surface systems related by a collection of isotopies.  There exists an additional collection of isotopies $\Phi$ from $F$ to $F'$ such that there are disjoint closed subintervals of $[0,1]$, 
$I_k=[a_k, b_k]$, for which either $a_k$ and $b_k$ are both   type~\ref{CP:bdry} critical points or are both   type~\ref{CP:bdryTriple}, and $\Phi$ has no other critical times. 
In addition, these subintervals satisfy the following conditions:
\begin{itemize}
\item For all $t \not \in \Cup I_k $,
$F^t$ is a C-complex. 
\item For all $t \in I_k$, $F^t$ has one ribbon intersection and no triple points or one triple point and no ribbon intersections.
\item for all $t\in [0,1]$, $F^t$ has no circle intersections.
\end{itemize}
\end{proposition}

\begin{proof}

Let $\Phi=\{\Phi_i\}$ be a collection of isotopies between C-complexes $F$ and $F'$.  By Proposition~\ref{lem: no interior CP's} we can assume that $\Phi$ has only critical points of type~\ref{CP:bdry} and \ref{CP:bdryTriple}.  Before we can explain how to eliminate circle intersections, we need to explain how to fix a single arc $\beta$ in a circle intersection which persists in $F_i^t\cap F_j^t$ for all $t$ in a reasonably long subinterval of $[0,1]$.  The argument will require a foray into smooth topology.

Let $\mathbb{F}_i = \{(\Phi_i^t(x),t)\mid x\in F_i, t\in [0,1]\} \subseteq S^3\times[0,1]$ be the trace of the isotopy $\Phi_i$.  After a small perturbation preserving the property that $\Phi$ is a collection of isotopies we arrange that $\bbF_i$ is transverse to $\bbF_j$, and to $\bdry \bbF_{j}$ for all $i\neq j$.
  We now have that $\bbX_{ij}:=\bbF_i\cap \bbF_j$ is a smooth 2-manifold in $S^3\times[0,1]$.  We may further arrange that $\bbF_i$ is transverse to $\bbX_{jk}$ and $\bdry \bbX_{jk}$.  It follows that $\bbX_{ijk} := \bbF_{i}\cap \bbX_{jk} =  \bbF_i\cap \bbF_j\cap \bbF_k$ is a smooth 1-manifold.  It is easy to check the following assertions:
\begin{enumerate}
\item $t$ is a type~\ref{CP:int} critical time with critical point $x\in F_i^t\cap F_j^t$ if and only if $(x,t)$ is a critical point for $p|_{\bbX_{ij}}$, the restriction of the projection $p:S^3\times [0,1]\to [0,1]$ to $\bbX_{ij}$.  
\item $t$ is a type~\ref{CP:bdry} critical time with critical point $x\in F_i^t\cap F_j^t$ if and only if $t\in (0,1)$ and $(x,t)$ is a critical point for $p|_{\bdry\bbX_{ij}}$.  
\item $t$ is a type~\ref{CP:triple} critical time with critical point $x\in F_i^t\cap F_j^t\cap F_k^t$ if and only if $(x,t)$ is a critical point for $p|_{\bbX_{ijk}}$. 
\item $t$ is a type~\ref{CP:bdryTriple} critical time with critical point $x\in F_i^t\cap F_j^t\cap F_k^t$ if and only if $(x,t)\in\bdry \bbX_{ijk}$.  
\item $t$ is a type~\ref{CP:quadruple} critical time with critical point $x\in F_i^t\cap F_j^t\cap F_k^t\cap F_\ell^t$ if and only if $(x,t)\in \bbF_i\cap \bbF_j\cap \bbF_k\cap \bbF_\ell$.  
\end{enumerate}
Notice that as $\bdry  \bbX_{ijk}$ and $\bbF_i\cap \bbF_j\cap \bbF_k\cap \bbF_\ell$ are each $0$-manifolds, any points in these will be critical points for any map to $[0,1]$.  

We  explain the first of these claims; the rest follow from the same logic.  Let $(x,t)\in \bbX_{ij}$.  Since $\bbF_i$ is the trace of an isotopy, there is a tangent vector to $\bbF_i$ at $(x,t)$, $\vect v_i\in T_{\bbF_{i}}(x,t)$, with $p_*(\vect v_i)\neq 0$.  Similarly, there is a $\vect v_j\in T_{\bbF_j}(x,t)$ with $p_*(\vect v_j)\neq 0$.  Since $T_{[0,1]}(t)$ is 1-dimensional, we may normalize $\vect v_i$ and $\vect v_j$ so that $p_*(\vect v_i - \vect v_j) = 0$.  This means that $\vect v_i - \vect v_j\in T_{S^3\times\{t\}}(x,t)$.  If $x$ is not a type~\ref{CP:int} critical point, then by the transversality of $F_i^t$ and $F_j^t$ in $S^3$, there are some $\vect u_i\in T_{F_i^t\times\{t\}}(x,t)$ and $\vect u_j\in T_{F_j^t\times\{t\}}(x,t)$ so that $\vect v_i - \vect v_j = \vect u_i-\vect u_j$.  
Observe that $\vect v_i-\vect u_i\in T_{\bbF_i}(x,t)$, $\vect v_j-\vect u_j\in T_{\bbF_j}(x,t)$, and $\vect v_i-\vect u_i = \vect v_j-\vect u_j$.  We now see that $\vect v_i-\vect u_i\in T_{\bbF_{i}}(x,t) \cap T_{\bbF_{j}}(x,t) = T_{\bbX_{ij}}(x,t)$ and $p_*(\vect v_i-\vect u_i) = p_*(\vect v_i)\neq 0$.  We conclude that $(x,t)$ is regular for $p|_{\bbX_{ij}}$.

Conversely, suppose that $(x,t)\in \bbX_{ij}$ is regular for $p|_{\bbX_{ij}}$.  Then there is some $\vect v\in T_{\bbX_{ij}}(x,t)$ with $p_*(\vect v)\neq 0$.  Let $\{\vect x_i, \vect y_i\}$ and $\{\vect x_j, \vect y_j\}$ be bases for $T_{F_i^t\times\{t\}}(x,t)$ and $T_{F_j^t\times\{t\}}(x,t)$, respectively. Then $\{\vect v, \vect x_i, \vect y_i\}$ and $\{\vect v, \vect x_j, \vect y_j\}$ give bases for $T_{\bbF_i}(x,t)$ and $T_{\bbF_j}(x,t)$, respectively.  Since $\bbF_i$ and $\bbF_j$ are transverse, $\{\vect v, \vect x_i, \vect y_i, \vect x_j, \vect y_j\}$ is a generating set for $T_{S^3\times[0,1]}(x,t)$.  It follows that $\{\vect x_i, \vect y_i, \vect x_j, \vect y_j\}$ generates $T_{S^3\times\{t\}}(x,t)$, so that $x$ is not a type~\ref{CP:int} critical point for $F_i^t\cap F_j^t$.  

Given a generic point $(x,s)\in F_i^s\cap F_j^s\times\{s\}\subseteq \bbX_{ij}$ we may flow in the direction of the gradient of $p|_{\bbX_{ij}}$ to produce a curve $(x(t),t)$ with $t\in [t_0, t_1]$ in $\bbX_{ij}$ running from a local minimum of $p|_{\bbX_{ij}}$ to a local maximum.  As 
$\Phi$ has no Type~\ref{CP:int} critical points, 
$p|_{\bbX_{ij}}$ has no critical points.  Thus,
either $t_0=0$ or $(x(t_0), t_0)$ is a type~\ref{CP:bdry} critical point of $\Phi$.  Similarly, either $t_1=1$ or $ (x(t_1), t_1)$ is a type~\ref{CP:bdry} critical point.
The only type~\ref{CP:bdry} critical point which result in a local minimum is a birth of a ribbon, as in Figure~\ref{fig: birth of a ribbon}.  The only type~\ref{CP:bdry} critical point which will result in a local maximum is a death of a ribbon (Figure~\ref{fig: birth of a ribbon} in reverse).  Pick $\epsilon>0$ small enough that $(t_0, t_0+\epsilon]$ and $[t_1-\epsilon, t_1)$ are disjoint and contain no critical times for $\Phi$.  The map $y(t) = (\Phi_i^t)^{-1}(x(t))$ with $t\in [t_0+\epsilon, t_1-\epsilon]$ gives an isotopy of a point interior to $F_i$.  By the isotopy extension theorem there is an isotopy $\Psi:F_i\times[0,1]\to F_i$ such that $\Psi^t(y(t_0+\epsilon)) = y(t)$ for all $t\in [t_0+\epsilon, t_1-\epsilon]$.  By replacing $\Phi_i^t$ with $\Phi_i^t\circ(\Psi^t)$ we arrange that $y(t)$ is constant in time.  

\begin{figure}[h]
     \centering
     \begin{subfigure}[b]{0.45\textwidth}
         \centering
         \begin{tikzpicture}
         \node at (0,0){\includegraphics[width=.9\textwidth]{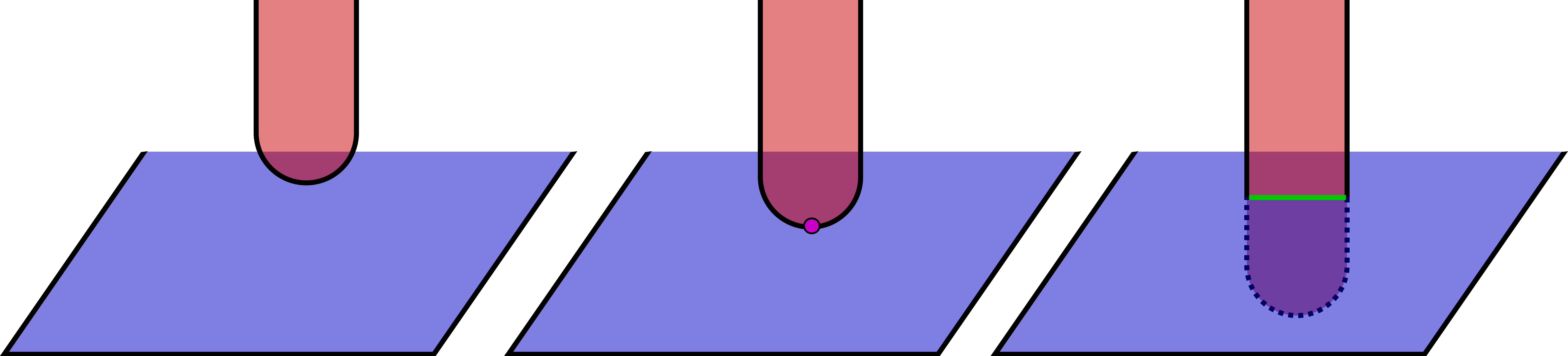}};]
         \end{tikzpicture}
         \caption{Birth of a ribbon.}
         \label{fig: birth of a ribbon}
     \end{subfigure}
     \begin{subfigure}[b]{0.45\textwidth}
         \centering
         \begin{tikzpicture}
         \node at (0,0){\includegraphics[width=.9\textwidth]{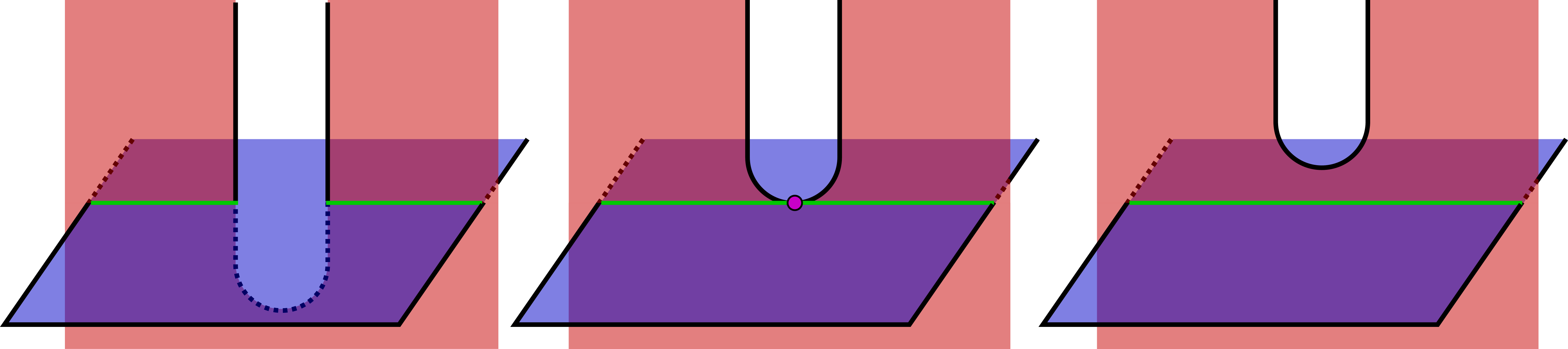}};]
         \end{tikzpicture}
         \caption{Merging two clasps into a ribbon.}
         \label{fig: merge clasps}
     \end{subfigure}
     \begin{subfigure}[b]{0.45\textwidth}
         \centering
         \begin{tikzpicture}
         \node at (0,0){\includegraphics[width=.9\textwidth]{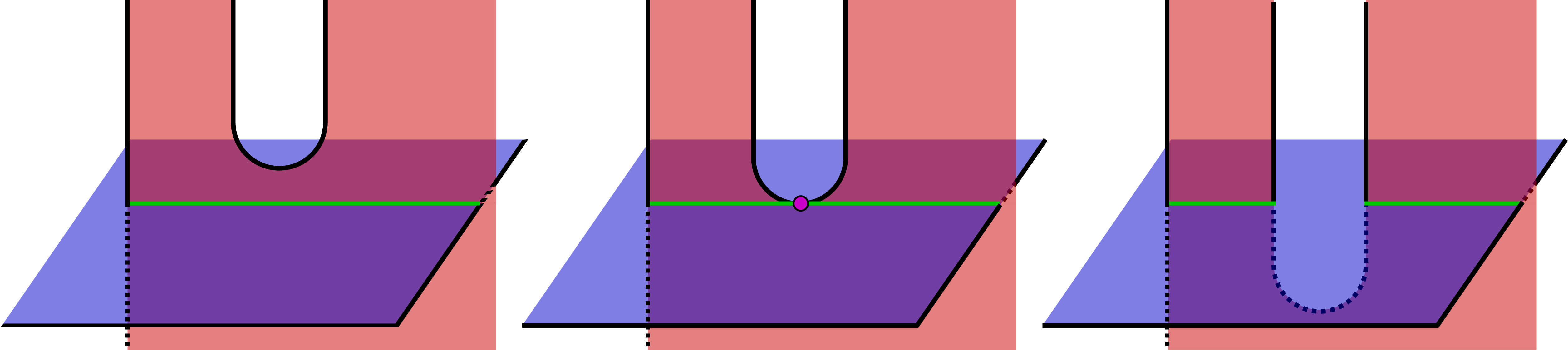}};]
         \end{tikzpicture}
         \caption{Splitting a clasp into a clasp and a ribbon.}
         \label{fig: split a clasp}
     \end{subfigure}
     \begin{subfigure}[b]{0.45\textwidth}
         \centering
         \begin{tikzpicture}
         \node at (0,0){\includegraphics[width=.9\textwidth]{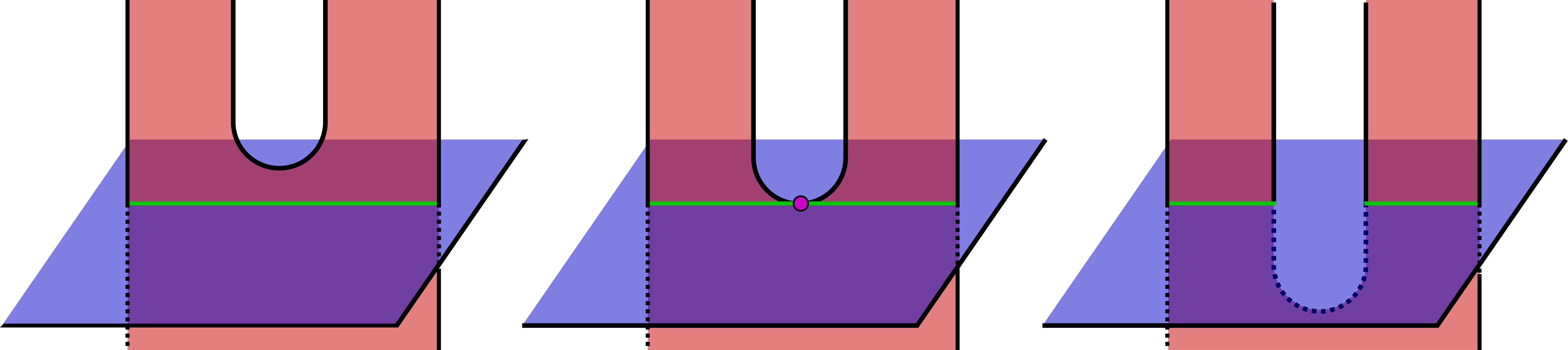}};]
         \end{tikzpicture}
         \caption{Splitting a ribbon into two ribbons.
         \\}
         \label{fig: split a ribbon}
     \end{subfigure}

        \caption{The ways that a type~\ref{CP:bdry} critical point can introduce a ribbon, assuming that a circle never appears.}
        \label{fig: ribbon appears}
\end{figure}

If we take a generic point $(x,s)\in F_i^s\cap F_j^s\times\{s\}$ together with a sufficiently small arc $a\subset F_i^s\cap F_j^s$ containing $x$ and follow the argument of the preceding paragraph, then we will produce an arc $\beta$ interior to $F_i$ so that $\Phi_i^t(\beta)\subseteq F_i^t\cap F_j^t$ for all $t\in [t_0+\epsilon, t_1-\epsilon]$ and $\Phi_i^s(\beta)=a$.  We will refer to this process as fixing the arc $\beta$ over the interval $[t_0+\epsilon, t_1-\epsilon]$.  

We are now ready to eliminate circle intersections.  Let $C\subseteq F_i^s\cap F_j^s$ be a circle intersection at time $s$.  As above, fix an arc $\beta$ interior to $F_i$ over the interval $[t_0+\epsilon, t_1+\epsilon]$ so that $\Phi_i^s(\beta)\subseteq C$ and for all $t\in [t_0+\epsilon, t_1+\epsilon]$, $\Phi_i^t(\beta)\subseteq F_i^t\cap F_j^t$.  The critical times $t_0$ and $t_1$ are either $0$ or $1$ or are births or deaths of ribbons.  Thus, for $t$ close to $t_0+\epsilon$, or $t_1+\epsilon$,  $\Phi_i^t(\beta)$ is contained in a ribbon or in a clasp.  Let $[s_0, s_1]\subseteq (t_0+\epsilon, t_1-\epsilon)$ be the maximal interval containing $s$ so that $\Phi_i^t(\beta)$ is contained in a circle intersection for all $t\in (s_0, s_1)$.  The interval $(s_0, s_1)$ can be thought of as the \emph{lifespan} of $C$.  Let $\alpha\subseteq F_i$ be an arc running from $\bdry F_i$ to $\beta$.  By picking $\alpha$ so that $\Phi_i(\alpha\times[s_0-\epsilon, s_1+\epsilon])$ is transverse to $\bdry \bbF_k^t$ for all $k$, we arrange that away from a finite set of times $q_1<\dots<q_r$, $\Phi_i^t(\alpha)$ is disjoint from $L_k$ for all $k\neq i$.  By pushing $\alpha$ off itself, we get another arc $\alpha'$ disjoint from $\alpha$ so that $\Phi_i^t(\alpha')$ is disjoint from $L_k$ for all $t\in (q_1-2\epsilon, q_1+2\epsilon)\cup\dots \cup (q_r-2\epsilon, q_r+2\epsilon)$ and all $k\neq i$.  

   Push along $\alpha$ over $[s_0-\epsilon, q_1-\epsilon]$, $[q_1+\epsilon, q_2-\epsilon]$, $\dots$, $[q_{r-1}+\epsilon, q_r-\epsilon]$, and $[q_r+\epsilon, s_1+\epsilon]$ and along $\alpha'$ over $[q_1-2\epsilon, q_1+2\epsilon]$, $\dots$, and $[q_r-2\epsilon, q_r+2\epsilon]$.  For all $t\in [s_0, s_1]$ the circle which contained $\Phi_i^t(\beta)$ has now been split into either one ribbon or into two ribbons.  Repeat this argument until all circle intersections have been removed for all times.

The reader will notice that $\Phi_i^t(\alpha)$ may contain some points in $F_i^t\cap F_j^t$ other than the point where it crosses $\Phi_i^t(\beta)$.  If it does, then we will split more arcs of intersection than we intended to.  This will not produce any new circle intersection, and so it is not relevant to the previous steps.  It later steps, however, it will be problematic. Fortunately, a similar technique to the paragraph above arranges that for each $t\in [0,1]$,  every component of complement in $F_i^t$ of the intersections with the other components of $F$ contains a point in $\bdry F_i$.  In brief, we push along arcs to reduce the number of components of $F_i - \Cup_{j\neq i} (\Phi_i^t)^{-1}(F_j^t)$ interior to $F_i$.

Now that we have eliminated all circle intersections, we
 shorten the lifespans of ribbon intersections.  We do so by arranging that that at any critical time the number of ribbons either changes from zero to one, from one to zero, or stays zero throughout.  Proceeding, let $s$ be a critical time not satisfying these conditions.  Then there is a ribbon intersection $R\subseteq F_i^s\cap F_j^s$ which persists for all $t\in (s-\epsilon, s+\epsilon)$ for a sufficiently small $\epsilon$.  More formally, we mean that if we start at a point on $R\times\{s\}\subseteq \bbX_{ij}$ and flow forward and backwards along the gradient to $p|_{\bbX_{ij}}$ then at least for $t\in (s-\epsilon, s+\epsilon)$ we will see only points $(x(t),t)$ with $x(t)$ sitting on a ribbon intersection in $F_{i}^t\cap F_j^t$.  Without loss of generality, we assume that $R$ is interior to $F_i^t$.  

As above we fix an arc $\beta \subseteq F_i$ so that $\Phi_i^s(\beta)\subseteq R$ and so that for some critical times $t_0$ and $t_1$ and all $t\in [t_0+\epsilon, t_1-\epsilon]$, $\Phi_i^t(\beta)\subseteq F_i^t\cap F_j^t$.  Recall that $\epsilon$ is picked small enough that $(t_0, t_0+\epsilon]$ and $[t_1-\epsilon, t_1)$ are disjoint and contain no critical times.    Pick a maximal subinterval $[s_0,s_1]\subseteq [t_0+\epsilon, t_1-\epsilon]$ containing $s$ so that for all $t\in (s_0, s_1)$, $\Phi_i^t(\beta)$ is contained in a ribbon interior to $F_i^t$.  

If $s_0=t_0+\epsilon$ then set $a_0=s_0$.  Otherwise, $s_0$ is a critical time when a ribbon appears, as in Figure~\ref{fig: ribbon appears}.  In this case set $a_0=s_0+\epsilon$.  Similarly, if $s_1=t_1-\epsilon$, then set $a_1=s_1$ and otherwise set $a_1=s_1-\epsilon$.  The interval $[a_0, a_1]$ is now an interval contained in the lifespan of the ribbon $r$ we saw at time $s$ and contains all critical times in ribbon's lifespan.  

Now we take advantage of the fact that we have arranged that for every $t\in [0,1]$, every component of $F_i - \Cup_{j\neq i} (\Phi_i^t)^{-1}(F_j^t)$ contains a point of $\bdry F_i$.  By a standard appeal to the Lebesgue number lemma, there exist some $a_0=q_0<q_1<\dots<q_n=a_1$ so that for each $\ell=1,\dots, n$ there is an arc $\alpha_\ell$ running from a point on $\bdry F_i$ to a point on $\beta$ which is otherwise disjoint from $\Sigma_i - \Cup_{j\neq i} (\Phi_i^t)^{-1}(F_j^t)$ for all $t\in [q_{\ell-1}, q_\ell]$.  These $q_\ell$ can be chosen so that $[q_{\ell}-\epsilon, q_{\ell}+\epsilon]$ contains no critical times.  

We now change $\Phi$ by pushing along each $\alpha_\ell$ over $[q_{\ell-1}+\epsilon, q_{\ell}-\epsilon]$.  Pushing along each of these arcs introduces exactly two new type~\ref{CP:bdry} critical times, one where a ribbon vanishes and another where it reappears.  Moreover, the ribbon containing $\Phi_i^t(\beta)$ is split into two clasps for all $t\in (q_{\ell-1}+\epsilon, q_{\ell}-\epsilon)$.  We have just reduced by one the number of ribbon intersections persisting at each critical time in $(a_0, a_1)$.   Iterate this until no ribbon intersection persists at any critical point.

We similarly arrange that no triple point persists at any critical time.  The argument is nearly identical to the procedure for controlling ribbon intersections and so we only highlight the differences.  Start with a triple point $x\in F_i^s\cap F_j^s\cap F_j^s$ which persists at a critical time $s$.  By flowing along the gradient to $p|_{\bbX_{ijk}}$ we fix over an interval a point $y\in F_i$ mapping to $x$.  A push along a collection of arcs from $\bdry F_i$ to $y$ reduces the number of triple points which persist at each critical time in the lifespan of the triple point $x$.  Depending on the precise implementation of this algorithm, new ribbon intersections may be introduced.  The techniques of the preceding paragraphs can now be used to split these into clasps without introducing new triple points.

We have now arranged that during the lifespan of any triple point or any ribbon intersection no other critical times can occur.  The existence of the claimed disjoint subintervals follows.  

\end{proof}

It remains to see how the paired critical points of Proposition~\ref{pairingcp} change the C-complex. Let $a$ be a type \ref{CP:bdry} critical time when a ribbon appears.  At the corresponding critical point we see one of the the local moves of Figure~\ref{fig: ribbon appears}.  As a consequence of Proposition~\ref{pairingcp}, at no time are there two ribbon intersections, so the move of Figure~\ref{fig: split a ribbon} does not occur.  As in Figure~\ref{fig: Split Clasp Redundant}, the move of Figure~\ref{fig: split a clasp} can be realized in terms of Figures~\ref{fig: birth of a ribbon}, and \ref{fig: merge clasps} and so we may assume it does not occur.    A type \ref{CP:bdry} critical time removes a ribbon by reversing these moves.  Thus, the effect of the paired critical times of type \ref{CP:bdry} is given by following one of these by the reverse of another.  Following Figure~\ref{fig: birth of a ribbon} with its own reverse can be obtained by an ambient isotopy.  Following Figure~\ref{fig: birth of a ribbon} by the reverse of~\ref{fig: merge clasps} results in a \ref{move: ribbon+push} move.  Following Figure~\ref{fig: merge clasps} with its reverse results in a \ref{move: push along different arc} move.

\begin{figure}[h]
         \centering
         \begin{tikzpicture}
         \node at (0,0){\includegraphics[width=.95\textwidth]{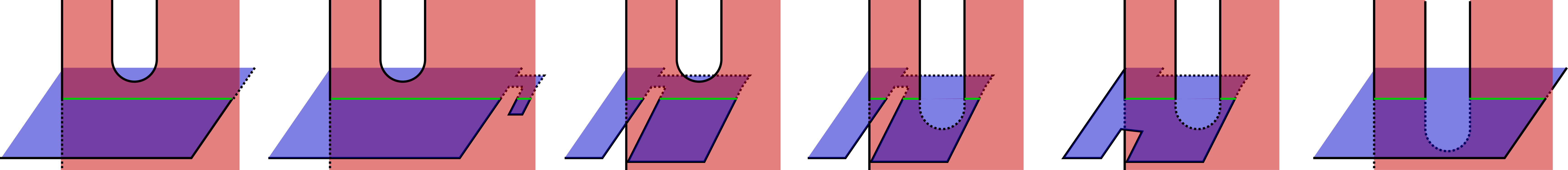}};]
         \node at (5.3,0){$\longrightarrow$};
         \node at (5.3,-1.2){Isotopy};
         \node at (2.8,0){$\longrightarrow$};
         \node at (2.8,-1.2){\pref{fig: merge clasps}};
         \node at (0.3,0){$\longrightarrow$};
         \node at (0.3,-1.2){\pref{fig: merge clasps}$^{-1}$};
         \node at (-2.05,0){$\longrightarrow$};
         \node at (-2.05,-1.2){Isotopy};
         \node at (-5,0){$\longrightarrow$};
         \node at (-5,-1.2){\pref{fig: birth of a ribbon}};
         \end{tikzpicture}
     
        \caption{
        The splitting of a clasp into a clasp and a ribbon (Figure \ref{fig: split a clasp}) is obtained by the birth of a ribbon (Figure \ref{fig: birth of a ribbon}), an ambient isotopy, splitting a ribbon into clasps (the reverse of Figure \ref{fig: merge clasps}), merging two clasps into a ribbon (Figure \ref{fig: merge clasps}) and finally an ambient isotopy.
        }
        \label{fig: Split Clasp Redundant}
\end{figure}
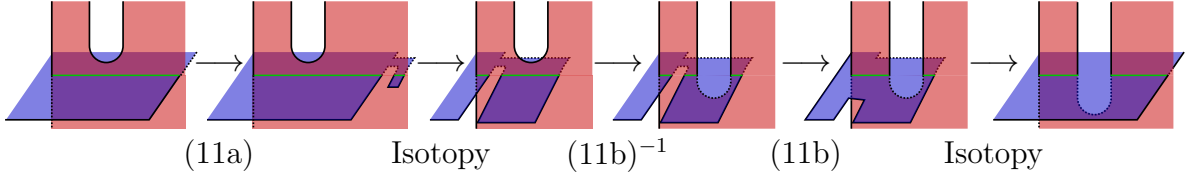

Next we analyze any  of the the paired type~\ref{CP:bdryTriple} critical times $a$ and $b$ of Proposition~\ref{pairingcp}.  At time $a$, a point $q\in F_i^a\cap \bdry (F_j^a\cap F_k^a)$ occurs.  As $\bdry (F_j^a\cap F_k^a)\subseteq \bdry F^a_j\cup \bdry F^a_k$, we assume that $q\in \bdry F^a_k$ without loss of generality. Since $F^a$ has only clasp intersections, $q$ is an endpoint of a clasp in $F^a_i\cap F^a_k$ and a clasp in $F^a_j\cap F^a_k$, and is interior to a clasp in $F^a_i\cap F^a_j$.  This moment is depicted in Figure~\ref{fig: Type IV crit.}.  The whole of Figure~\ref{fig: type IV paired} depicts the moments before this critical time, the lifespan of the resulting a triple point, the critical time $b$ when the triple point dies, and finally the moment after.  In Figure~\ref{fig: Find T3 move} we realize the total effect of these paired critical times in terms of \ref{move:T3} and the moves of Figure~\ref{fig: ribbon appears}.

\begin{figure}
     \centering
     \begin{subfigure}[b]{0.13\textwidth}
         \centering
         \begin{tikzpicture}
         \node at (0,0){\includegraphics[width=\textwidth]{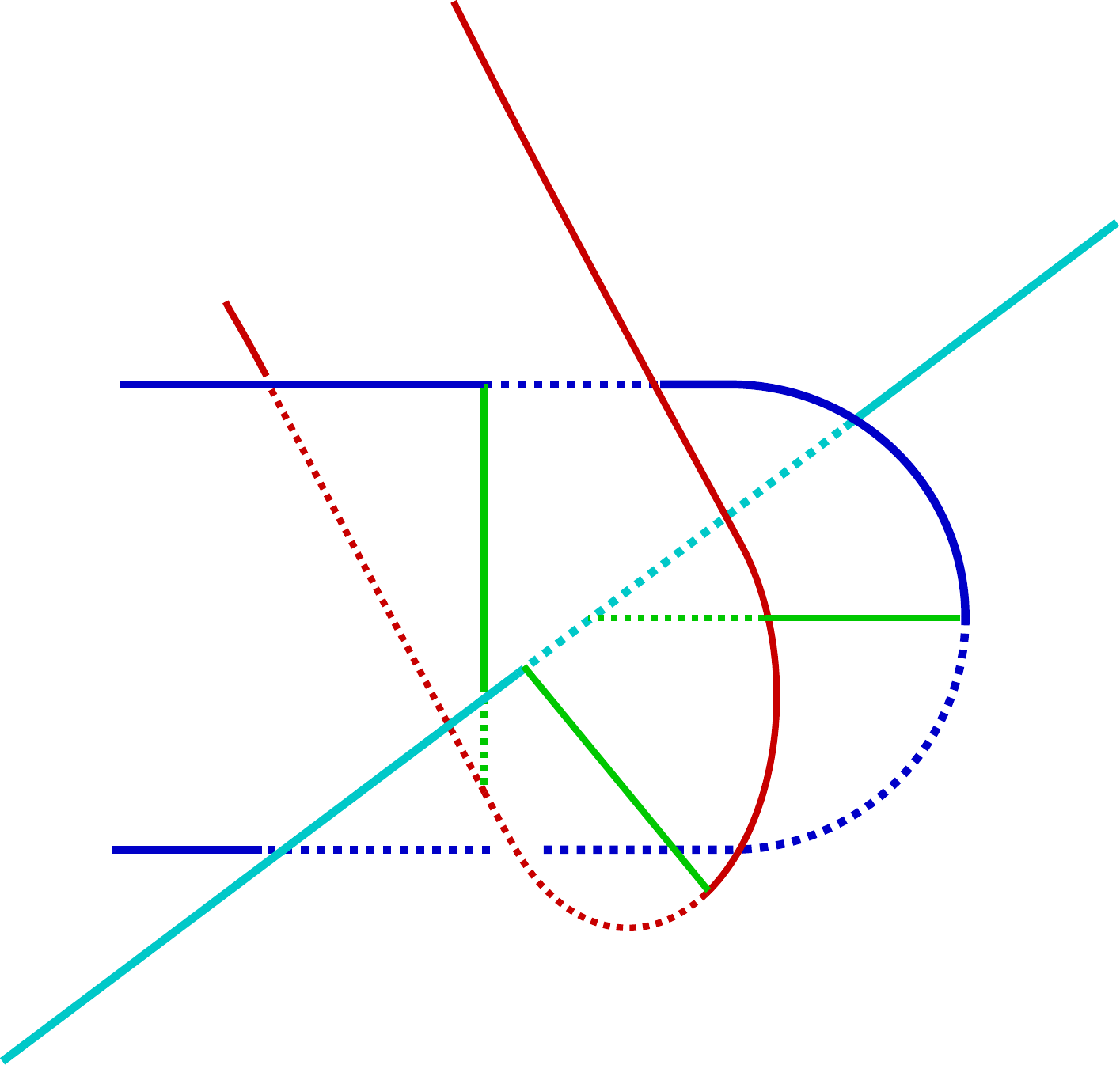}};
         \end{tikzpicture}
     \caption{}
     \label{fig: before Type IV crit.}
\end{subfigure}
     \begin{subfigure}[b]{0.13\textwidth}
         \begin{tikzpicture}
         \node at (0,0){\includegraphics[width=\textwidth]{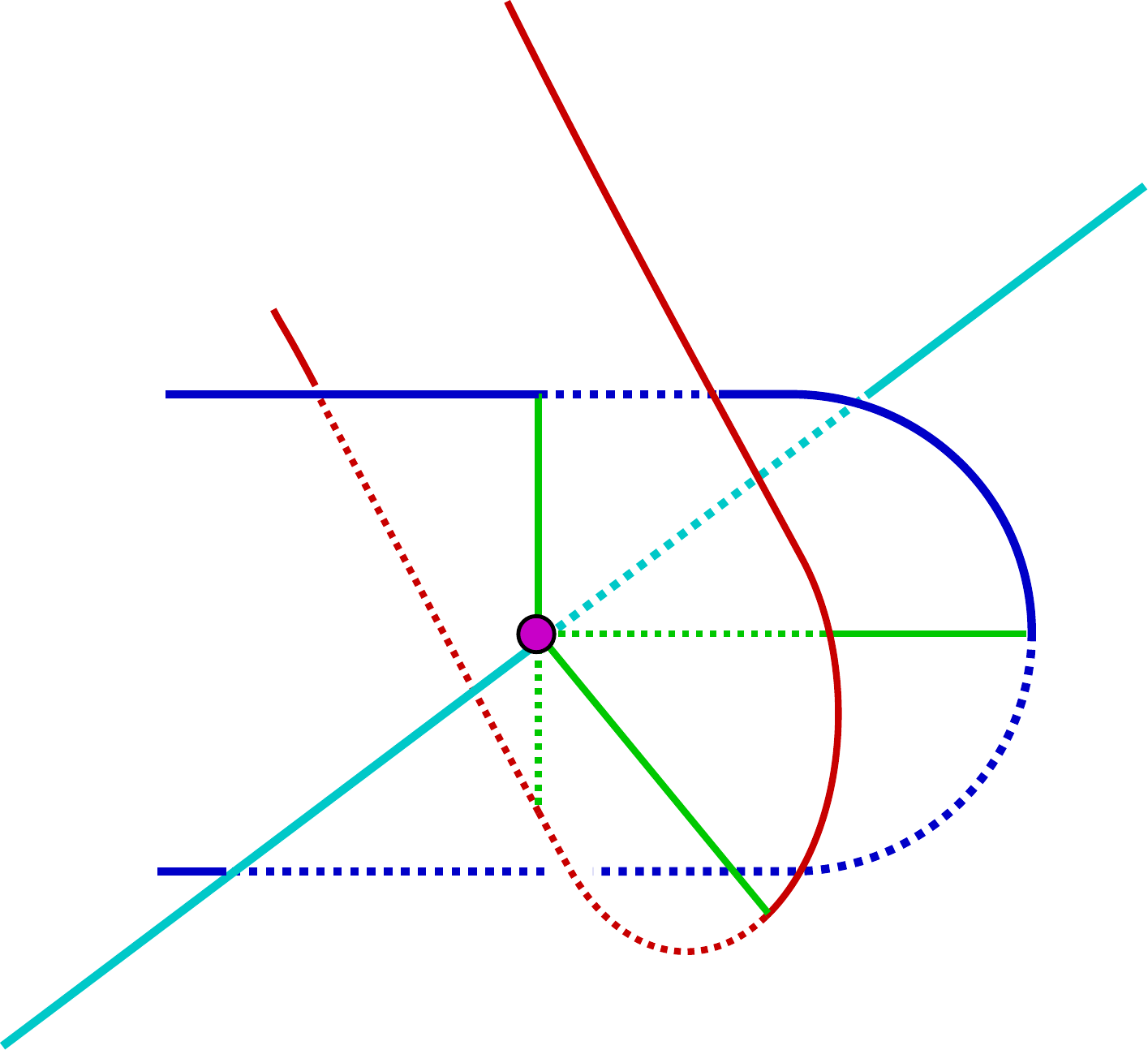}};
         \end{tikzpicture}
     \caption{}
     \label{fig: Type IV crit.}
\end{subfigure}
     \begin{subfigure}[b]{0.13\textwidth}
         \begin{tikzpicture}
         \node at (0,0){\includegraphics[width=\textwidth]{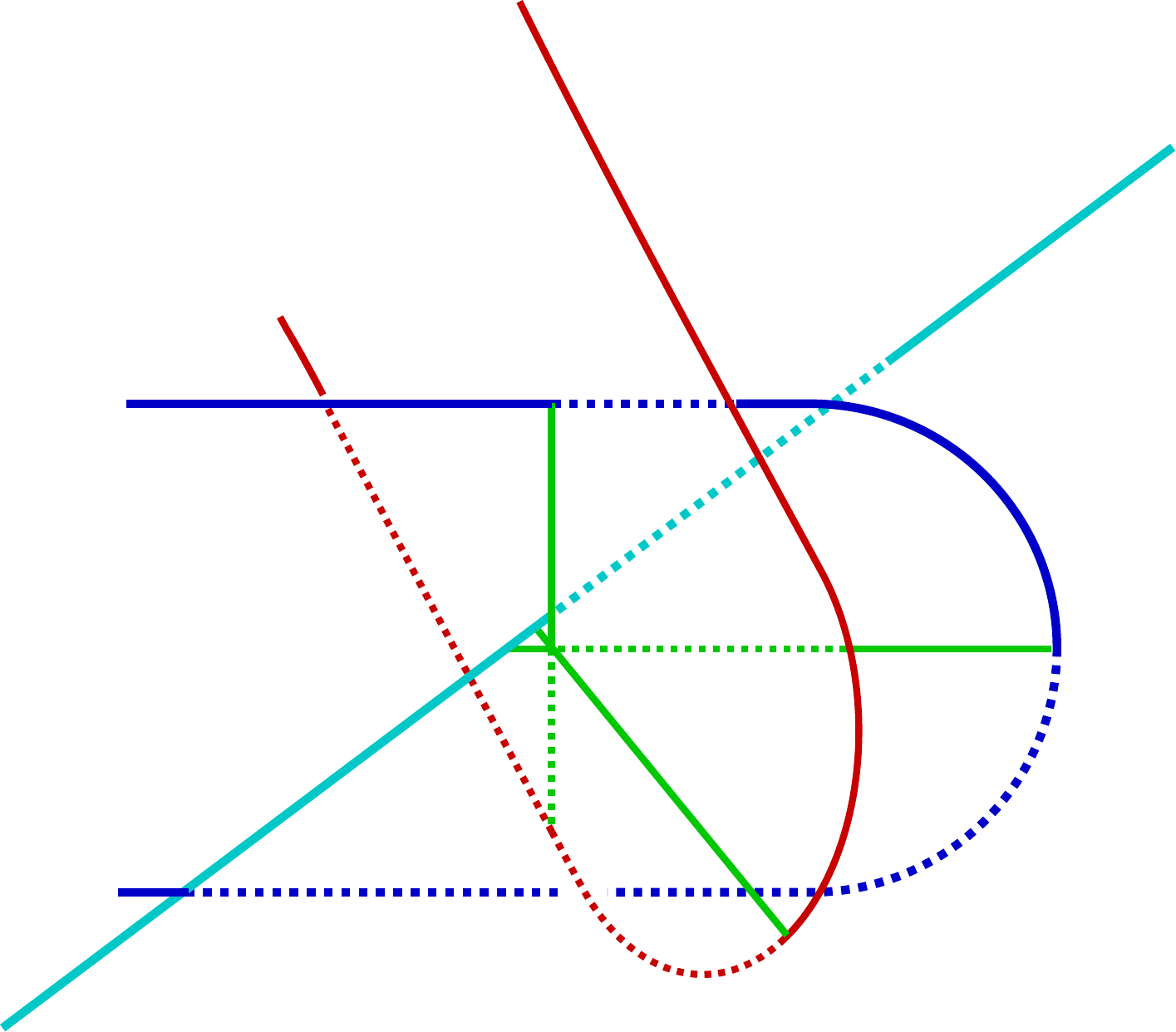}};
         \end{tikzpicture}
     \caption{}
     \label{fig: between1.}
\end{subfigure}
     \begin{subfigure}[b]{0.13\textwidth}
         \begin{tikzpicture}
         \node at (0,0){\includegraphics[width=\textwidth]{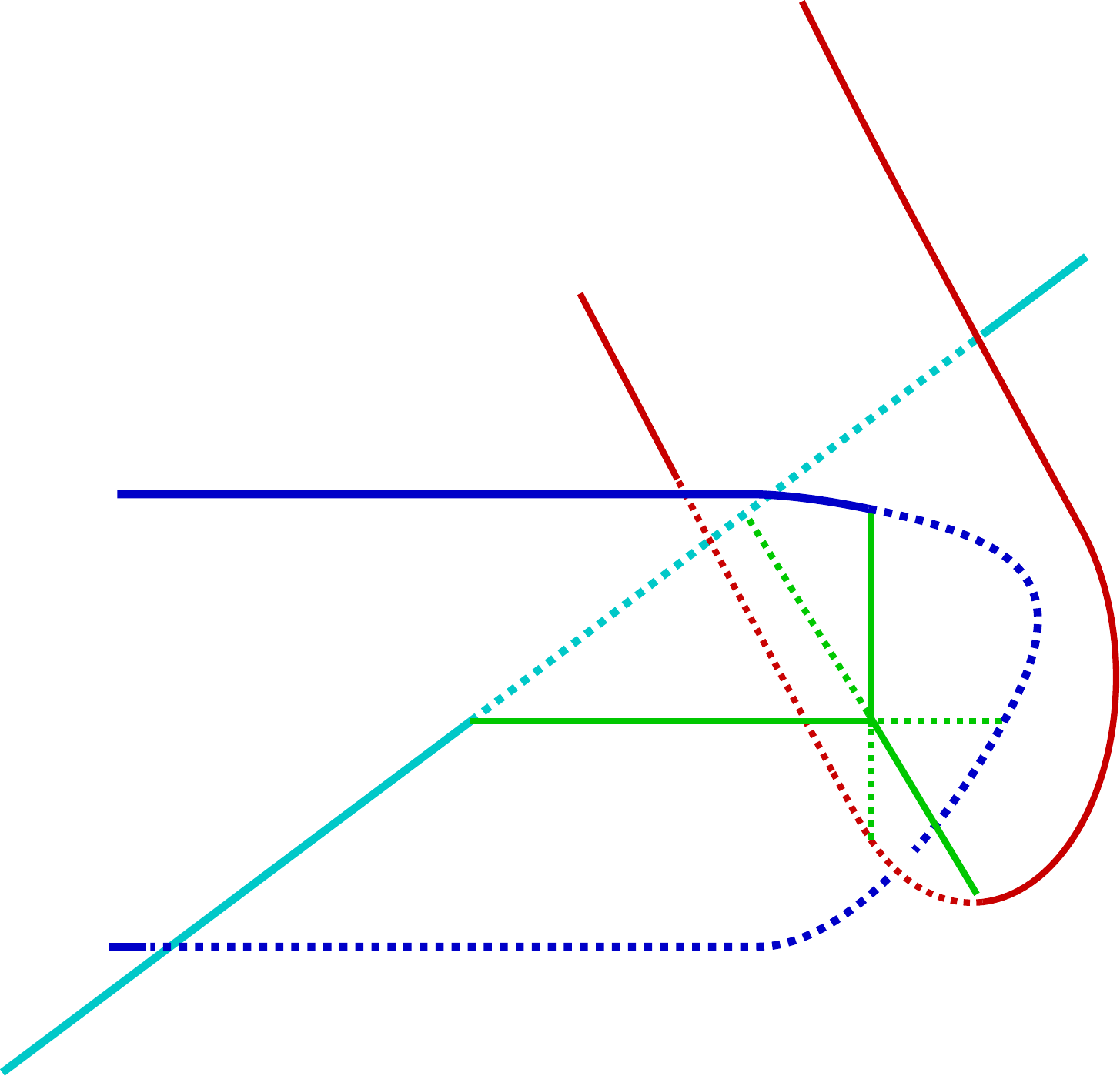}};
         \end{tikzpicture}
     \caption{}
     \label{fig: between2.}
\end{subfigure}
     \begin{subfigure}[b]{0.13\textwidth}
         \begin{tikzpicture}
         \node at (0,0){\includegraphics[width=\textwidth]{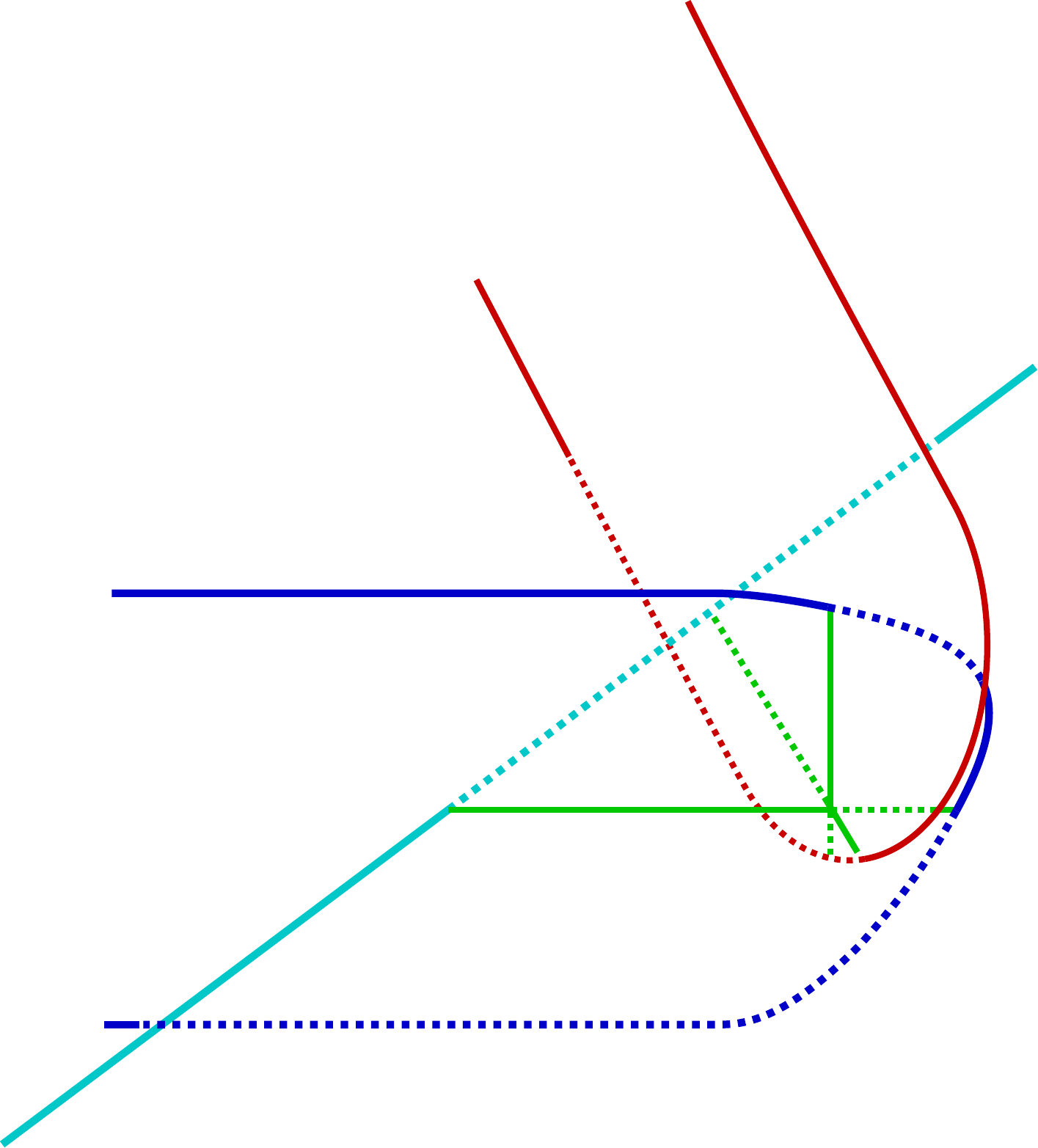}};
         \end{tikzpicture}
     \caption{}
     \label{fig: between3.}
\end{subfigure}
     \begin{subfigure}[b]{0.13\textwidth}
         \begin{tikzpicture}
         \node at (0,0){\includegraphics[width=\textwidth]{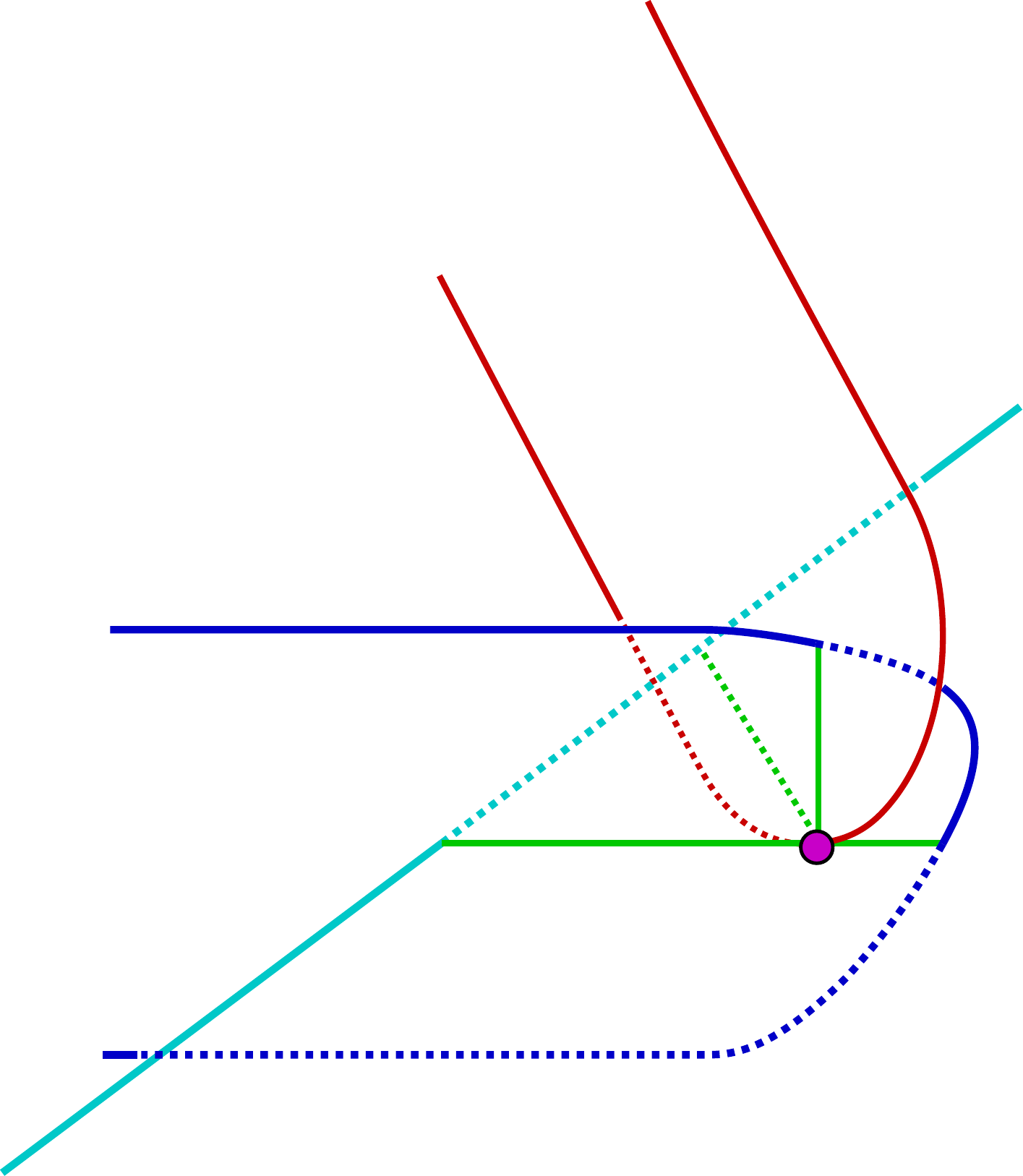}};
         \end{tikzpicture}
     \caption{}
     \label{fig: Type IV crit II.}
\end{subfigure}
     \begin{subfigure}[b]{0.13\textwidth}
         \begin{tikzpicture}
         \node at (0,0){\includegraphics[width=\textwidth]{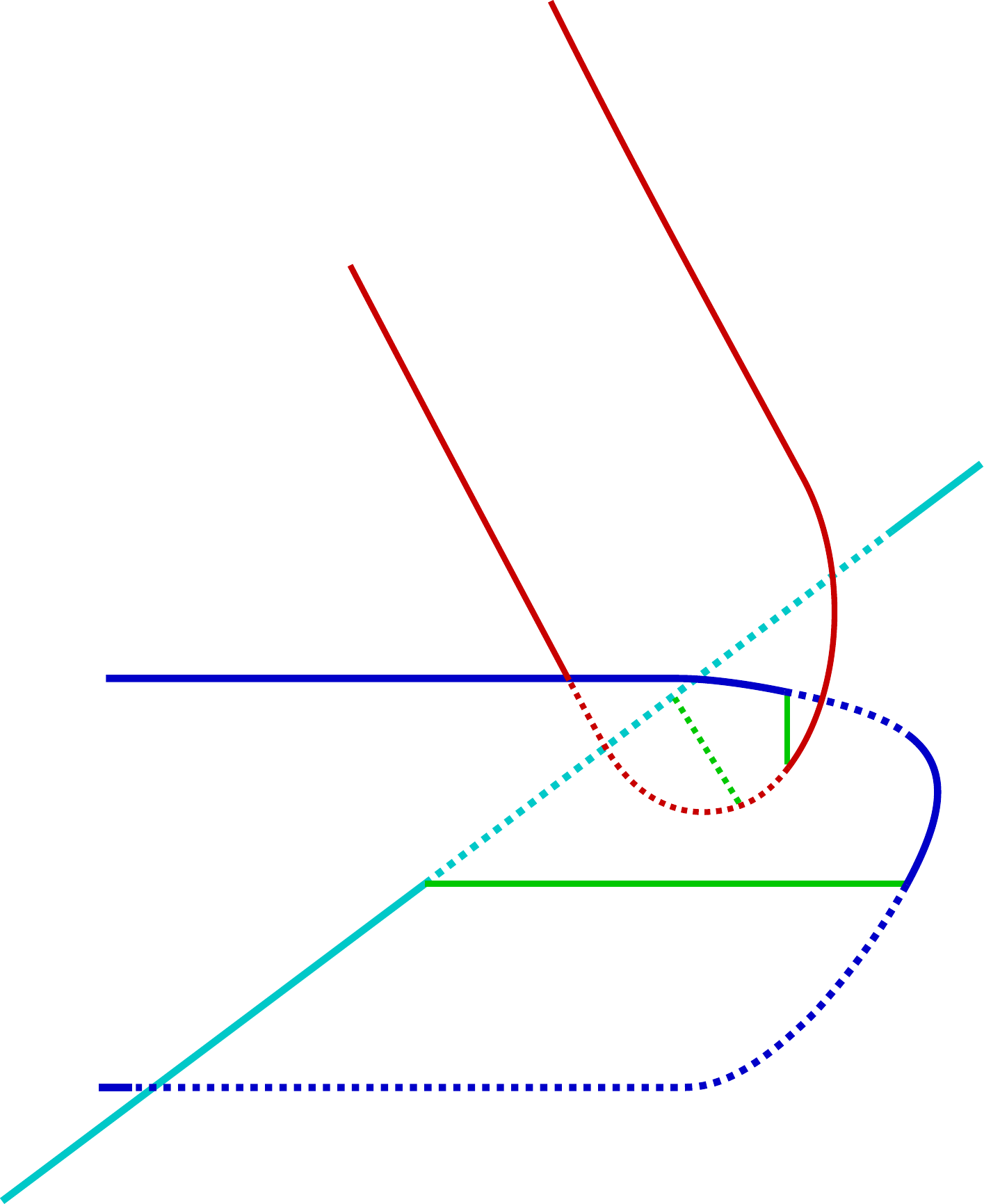}};
         \end{tikzpicture}
     \caption{}
     \label{fig: after type IV crit}
\end{subfigure}
     
        \caption{Left to right:  A triple point appears at a type~\ref{CP:bdryTriple} critical point.  The resulting triple point travels the length of a clasp.  The triple point vanishes at a second critical time.  To aid in readability, the surfaces are not visible. Clasps appear as green.}
        \label{fig: type IV paired}
\end{figure}

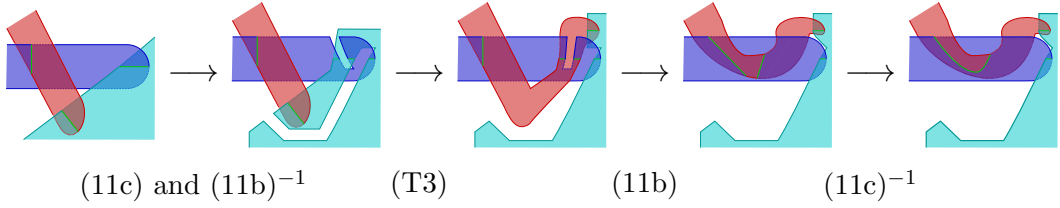
\begin{figure}
     \centering
     
         \begin{tikzpicture}
        \node at (0,0){\includegraphics[width=.12\textwidth]{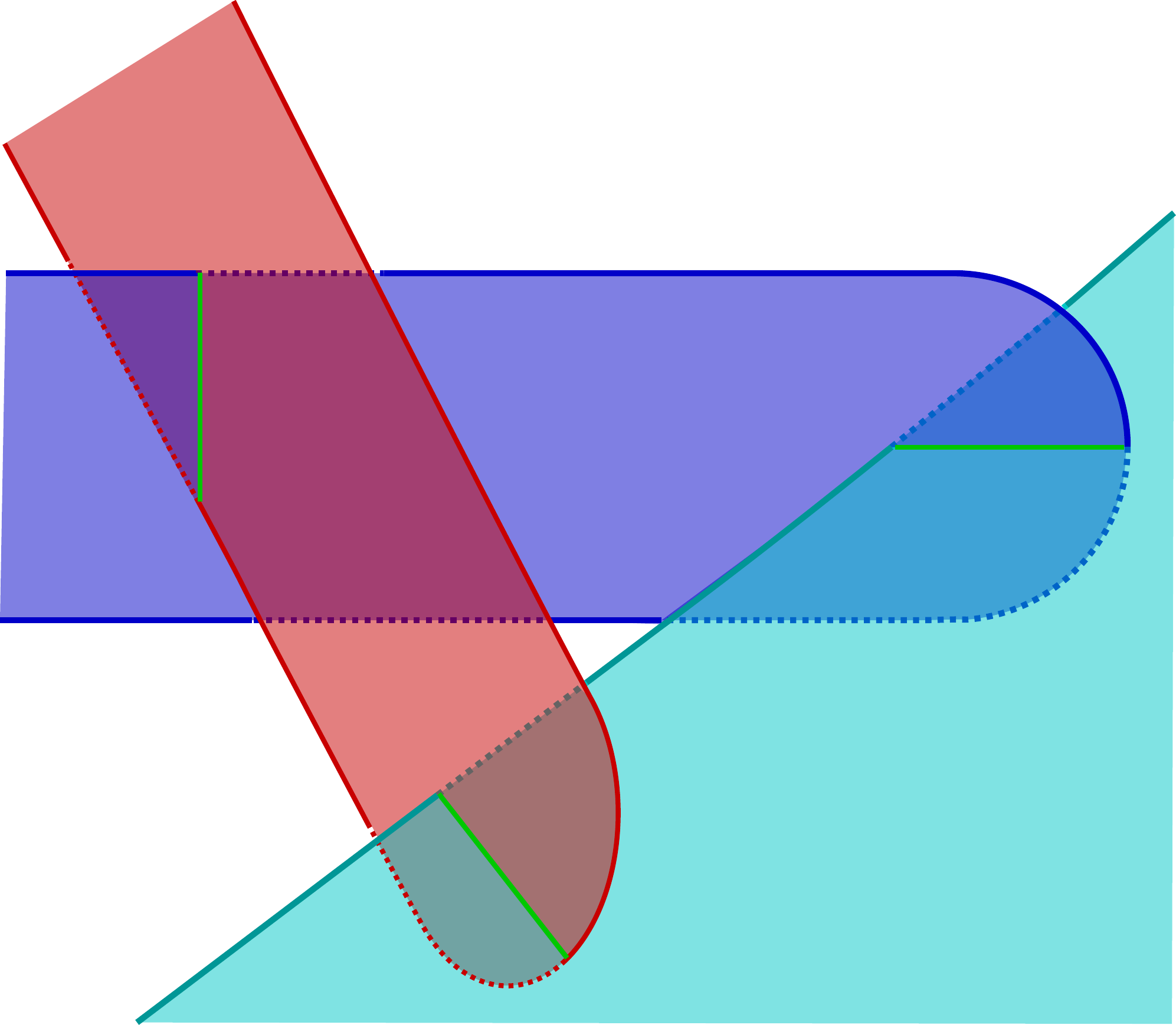}};
        \node at (1.5,0){$\longrightarrow$};
        \node[below] at (1.5, -1.12){\small{\pref{fig: split a clasp} and \pref{fig: merge clasps}$^{-1}$}};
        \node at (3,0){\includegraphics[width=.12\textwidth]{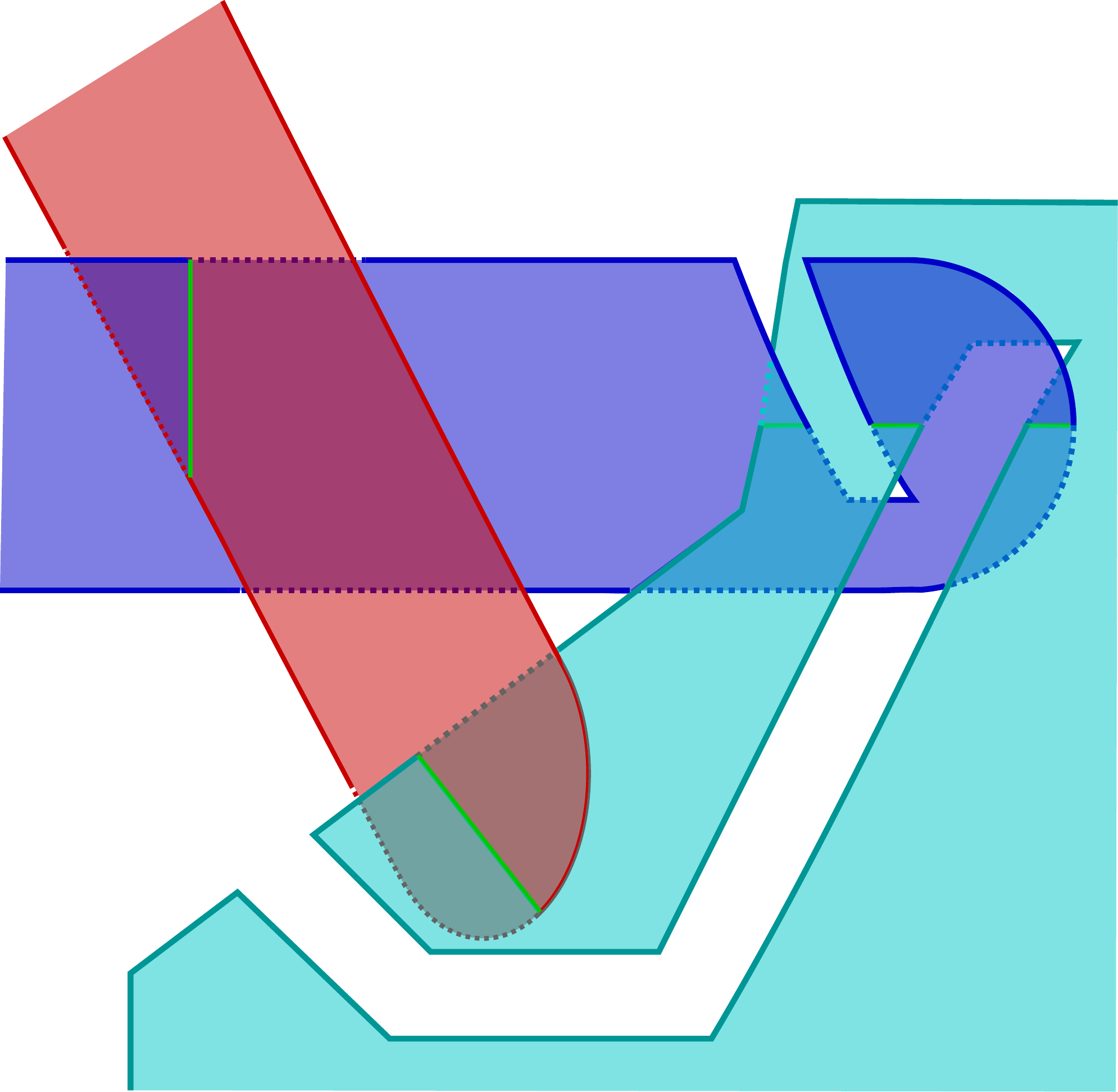}};
        \node at ({3+1.5},0){$\longrightarrow$};
        \node[below] at ({3+1.5}, -1.12){\small{\ref{move:T3}}};
        \node at (6,0){\includegraphics[width=.12\textwidth]{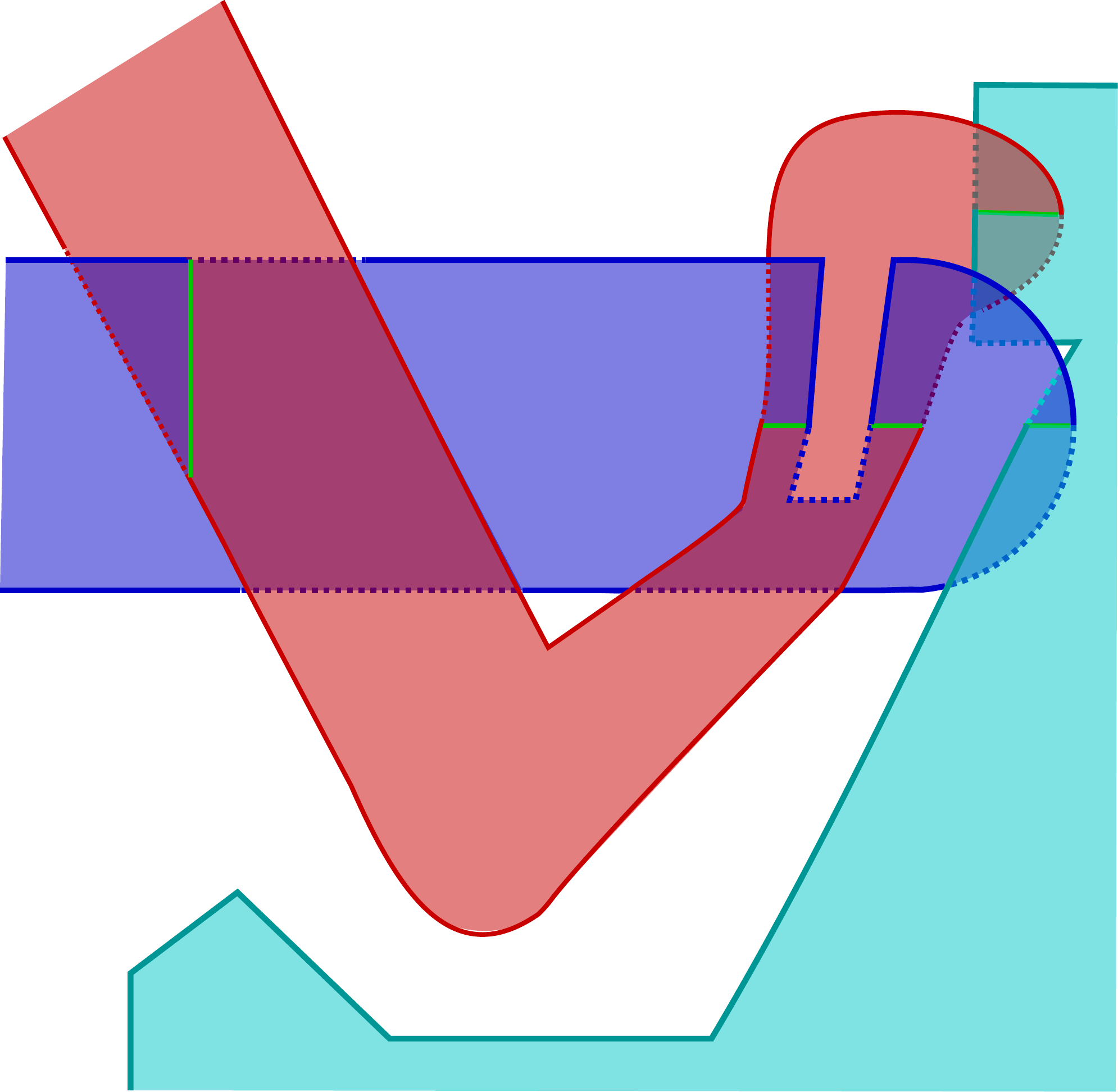}};
        \node at ({6+1.5},0){$\longrightarrow$};
        \node[below] at ({6+1.5}, -1.12){\small{\pref{fig: merge clasps}}};
        \node at (9,0){\includegraphics[width=.12\textwidth]{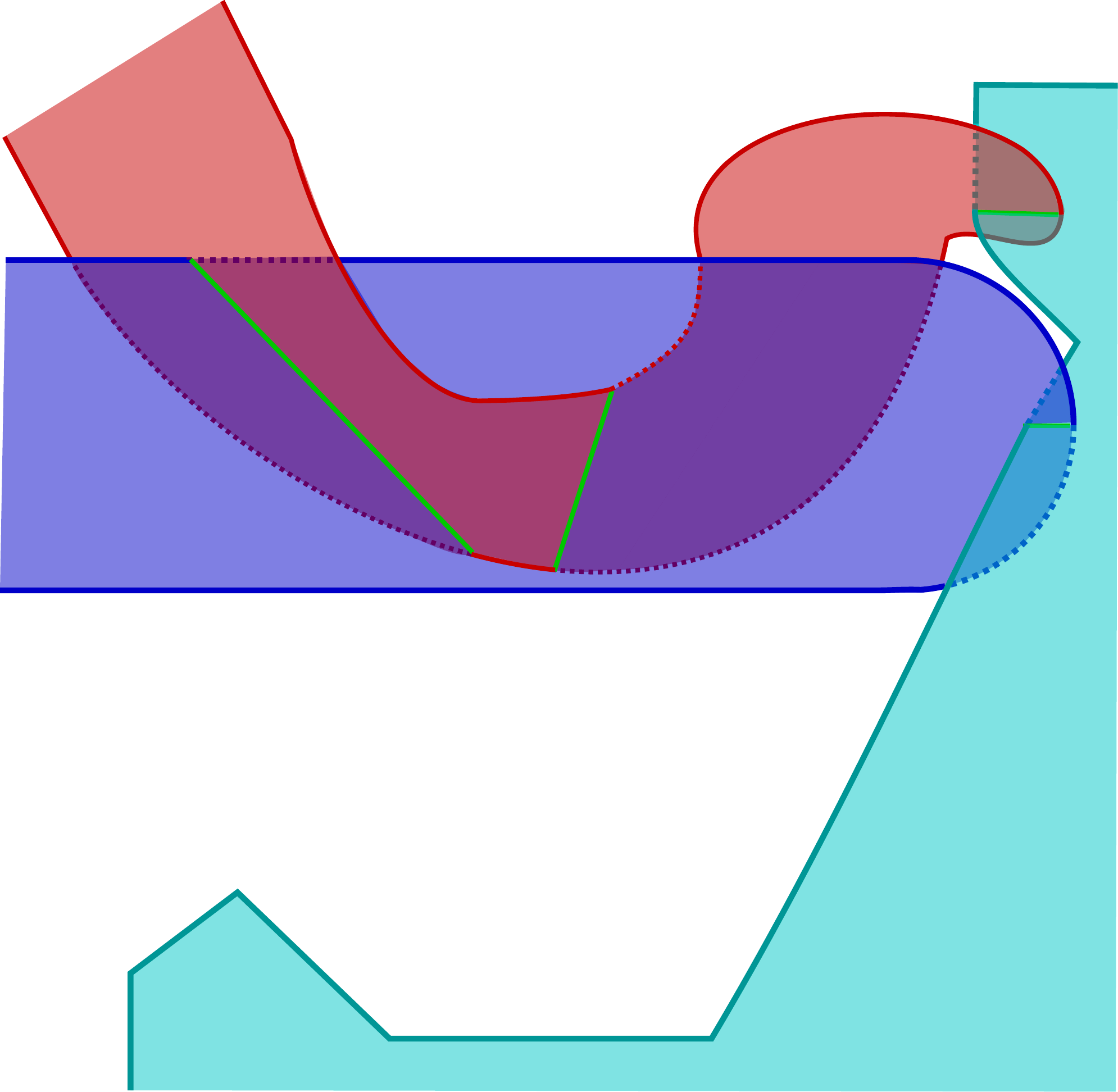}};
        \node at ({9+1.5},0){$\longrightarrow$};
         \node[below] at ({9+1.5},-1.12){\small{\pref{fig: split a clasp}$^{-1}$}};
        \node at (12,0){\includegraphics[width=.12\textwidth]{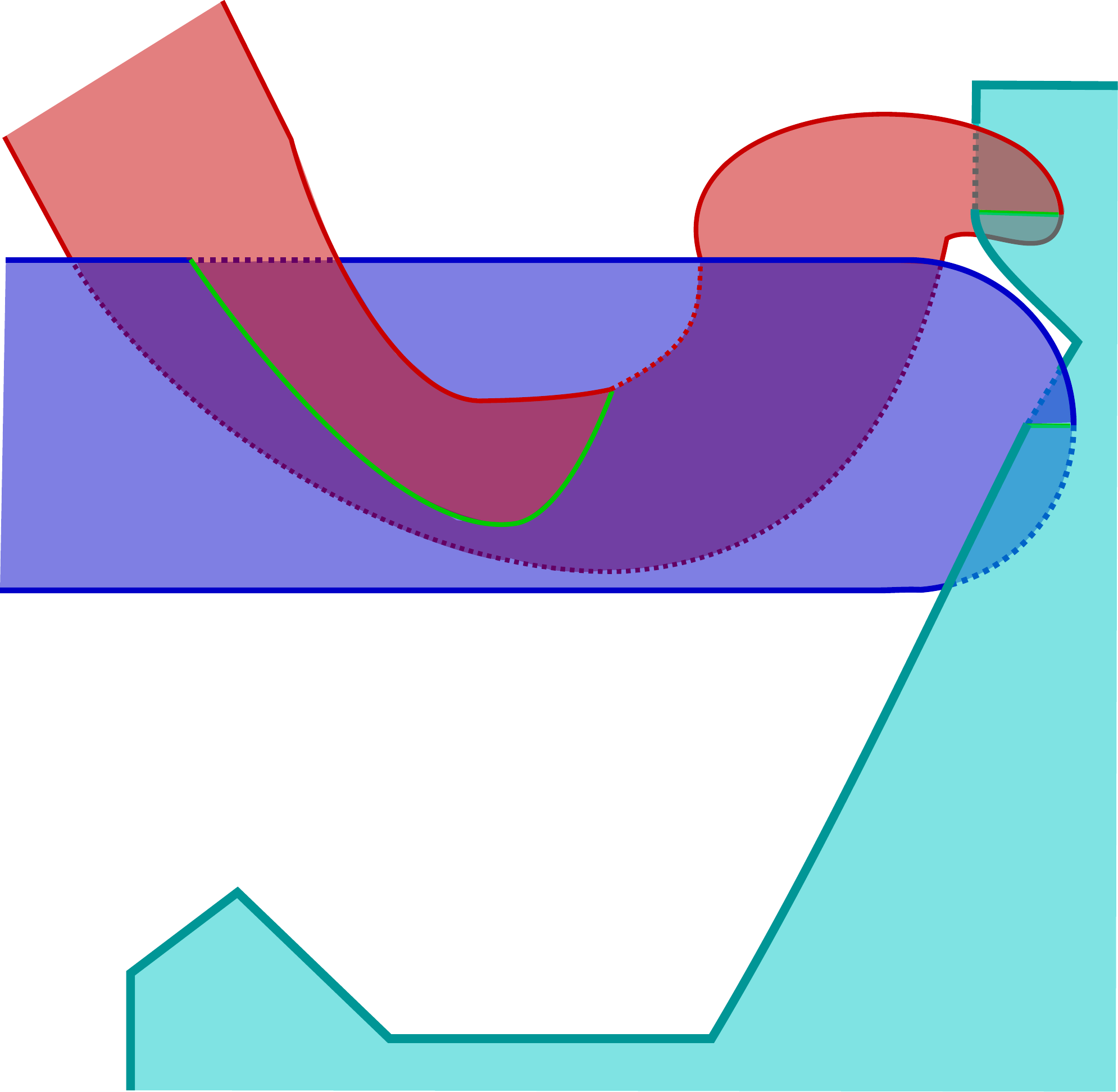}};
         \end{tikzpicture}
            
        \caption{A sequence of the moves of Figure~\ref{fig: ribbon appears}, their inverses, and the \ref{move:T3} move relating diagrams \ref{fig: before Type IV crit.} and \ref{fig: after type IV crit}.  Transitioning left to right:  Split a clasp into a clasp and a ribbon and split that ribbon into two claps; The \ref{move:T3} move; Merge two clasps into a ribbon; Merge a clasp and a ribbon into a clasp. 
}
        \label{fig: Find T3 move}
\end{figure}

\subsection{When does a collection of isotopies extend to an ambient isotopy?}\label{subsect: technical iso extn}

The final detail needed to complete the proof of Theorem~\ref{thm: corrected lemma} is Proposition~\ref{prop: get ambient isotopy}, which says that away from critical times, a collection of isotopies between pairs of surface systems extends to an ambient isotopy on the whole space. This result seems well-known; however, since we could find no explicit reference, we include an argument for completeness.

\begin{proof}[Proof of Proposition~\ref{prop: get ambient isotopy}.]
Let $F = F_1\cup\dots \cup F_{n}$ be a surface system and $\Phi = \{\Phi_i\}$ be a collection of isotopies from $F$ to another surface system.  Suppose also that for all $t\in [0,1]$ we have that $F^t = \Phi_1^t(F_1)\cup \dots \cup\Phi_n^t(F_n)$ is a surface system.  
Let $\bbF_i \subseteq S^3 \x [0,1]$ be the trace of $\Phi_i^t$, $\bbX_{ij} = \bbF_i \cap \bbF_j$, and $\bbX_{ijk} = \bbF_i \cap \bbF_j\cap \bbF_k$. 

As we observed in the proof of Proposition~\ref{pairingcp}, critical points of the restriction of $p:S^3\times[0,1]\to [0,1]$ to $\bbX_{ij}$ and to $\bbX_{ijk}$ correspond to times when $F^t$ fails to be a surface system. We have assumed that this never happens.  Thus, $p|_{\bbX_{jk}}$ and $p|_{\bbX_{ijk}}$ have no critical values.

Appealing to the regular interval theorem of morse theory reveals a diffeomorphism $\mathcal{Q}_{ijk}:(F_i\cap F_j\cap F_k)\times[0,1] \to \bbX_{ijk}$ for which $p(\mathcal{Q}_{ijk}(x,t)) = t$; see \cite[Lemma 2.10]{Morse}. Thus, $\mathcal{Q}_{ijk}$ gives a parametrization of $\bbX_{ijk}$ as an isotopy of $F_i\cap F_j\cap F_k$. Since there are no quadruple intersections,  we have that $\bigcup_{i,j,k} \mathcal{Q}_{ijk}$ is an isotopy of the collection of all triple points. This isotopy extends to an ambient isotopy $\mathcal{R}$ on $S^3$. 

Replace $\Phi^t$ by its composition with the inverse of $\mathcal{R}$,  $\{(\mathcal{R}^t)\inv \circ \Phi_i^t\}_i$.  We now have that $F_i^t \cap F_j^t \cap F_k^t = F_i \cap F_j \cap F_k$ for all $t\in [0,1]$ and all $i,j,k$. For any triple point $q \in F_i \cap F_j \cap F_k$,  take a 3-ball neighborhood $B_q$ of $q$, small enough to ensure that some identification $\mathcal{B}^t: \mathbb{B}^3 \rightarrow B_q$ sends the $xy-$unit disk to $F_i^t$, the $xz$-unit disk to $F_j^t$ and the $yz$-unit disk to $F_k^t$. The isotopy extension theorem concludes that the isotopy $\mathcal{B}$ extends to an ambient isotopy of $S^3$, which we also call $\mathcal{B}$. By replacing $\Phi$ with $\{(\mathcal{B}^t)\inv \circ \Phi_i^t\}$, we arrange that $F_i^t \cap B_p = F_i\cap B_p$ for all $t\in [0,1]$ and all $i$. 

In the same way, we arrange that $F^t$ is fixed in $\nu(\bdry F)$, a neighborhood of $\bdry F$. 
Let $U :=( \bigcup_{q} B_q) \cup \nu(\bdry F)$.  $(\Phi_i^t)|_{F_i\bk U}:(\Phi_i^t)|_{F_i\bk U}\to S^3\bk U$ is now an isotopy though neat embeddings.  The trace is given by $\bbF_i' := \bbF_i\bk(U\times[0,1])$. They key observation is that we have now removed all triple points.

We may now use the same technique as we did for the triple points to fix the double-intersections.  By applying the regular interval theorem to $\bbF_i'\cap \bbF_j'\subseteq \bbX_{ij}$, and then using that there are no triple points we  parametrize $\Cup_{i,j}\bbF_i'\cap \bbF_j'$ as the trace of an isotopy of $\Cup_{i,j}F_i\cap F_j\bk U$ though neat embeddings into $S^3\bk U$.  We extend this isotopy to an ambient isotopy on $S^3\bk U$ and then compose $(\Phi_i^t)|_{F_i\bk U}$ with the inverse of this ambient isotopy.  As a result we ensure that $F_i^t\cap F_j^t\bk U = F_i\cap F_j\bk U$ for all $t$.  Also similar to the treatment of triple points, we arrange that for a small neighborhood $W\subseteq S^3\bk U$ of $\Cup_{i,j}F_i\cap F_j\bk U$, $(F_i^t\bk U)\cap W = (F_i\bk U)\cap W$ for all $t\in [0,1]$ and all $i$. 

$(\Phi_i^t)|_{F_i\bk(U\cup W)}$ is an isotopy through neat embeddedings into $S^3\bk(U\cup W)$ and $(\Phi_i^t)(F_i\bk(U\cup W))$ is disjoint from $(\Phi_j^t)(F_j\bk(U\cup W))$ for all $i,j,t$.  Thus, $\Cup_i(\Phi_i^t)|_{F_i\bk(U\cup W)}$ gives an isotopy of a disjoint union of neatly embedded surfaces, $\Cup_i F_i\bk(U\cup W)$.  This extends to an ambient isotopy on $S^3 \bk (U \cup W)$. Extend this ambient isotopy by the identity over $W$ and then over $U$ to arrive at an ambient isotopy of $S^3$ that extends each $\Phi_i$ simultaneously.   
\end{proof}

\bibliographystyle{plain}

\bibliography{biblio}

\end{document}